\documentclass[12pt, reqno]{amsart}
\usepackage{amssymb,amsmath,dsfont}
\usepackage{pgfplots}
\pgfplotsset{compat=1.14}
\usepackage{enumitem}
\usepackage{calc}

\title[Coalescence of geodesics and the BKS midpoint problem]{Coalescence of geodesics and the BKS midpoint problem in planar first-passage percolation}
\date{}
\author{Barbara Dembin}
\address{Barbara Dembin\hfill\break
    D-MATH, ETH Z\"urich,
    Switzerland.}
\email{barbara.dembin@math.ethz.ch}
\author{Dor Elboim}
\address{Dor Elboim\hfill\break
    Department of Mathematics,
    Princeton University,
    New Jersey, United States.}
\email{delboim@math.princeton.edu}
\author{Ron Peled}
\address{Ron Peled\hfill\break
    School of Mathematical Sciences,
    Tel Aviv University,
    Israel. Institute for Advanced Study and Princeton University, USA.}
\email{peledron@tauex.tau.ac.il}

\usepackage[margin=1in]{geometry}
\usepackage[all]{xy}
\usepackage{amssymb}
\usepackage{amsfonts}
\usepackage{amsthm}
\usepackage{amsmath}
\usepackage{hyperref}
\usepackage{mathtools}
\usepackage{url}
\usepackage{xcolor}
\usepackage{dsfont}
\usepackage{pgfplots}
\usepackage{comment}

\newtheorem{thm}{Theorem}[section]
\DeclareMathOperator{\sides}{Sides}
\newtheorem{lem}[thm]{Lemma}  
\newtheorem{prop}[thm]{Proposition}

\newtheorem{remark}[thm]{Remark}
\newtheorem{claim}[thm]{Claim}

\newtheorem*{theorem*}{Theorem}

\newcommand{\PP}{\mathbb{P}}

\newcommand{\ZZ}{\mathbb{Z}}
\newcommand{\EE}{\mathbb{E}}

\newcommand{\cB}{\mathcal{B}}
\newcommand{\cE}{\mathcal{E}}
\newcommand{\ind}{\mathds{1}}
\newcommand{\ep}{\epsilon}
\mathtoolsset{showonlyrefs}

\numberwithin{equation}{section}

\theoremstyle{plain} 
\newcommand{\thistheoremname}{}
\newtheorem*{genericthm*}{\thistheoremname}
\newenvironment{namedthm*}[1]
  {\renewcommand{\thistheoremname}{#1}%
   \begin{genericthm*}}
  {\end{genericthm*}}

\newcommand{\tube}{\textup{Tube}}

\newcommand{\OB}{\Omega_{\textup{basic}}}
\newcommand{\OT}{\Omega_{\textup{Tal}}}

\begin{document}

\maketitle

\begin{abstract}
    We consider first-passage percolation on $\ZZ^2$ with independent and identically distributed weights whose common distribution is absolutely continuous with a finite exponential moment. Under the assumption that the limit shape has more than 32 extreme points, we prove that geodesics with nearby starting and ending points have significant overlap, coalescing on all but small portions near their endpoints. The statement is quantified, with power-law dependence of the involved quantities on the length of the geodesics.
    
    The result leads to a quantitative resolution of the Benjamini--Kalai--Schramm midpoint problem. It is shown that the probability that the geodesic between two given points passes through a given edge is smaller than a power of the distance between the points and the edge.

    We further prove that the limit shape assumption is satisfied for a specific family of distributions.

    Lastly, related to the 1965 Hammersley--Welsh highways and byways problem, we prove that the expected fraction of the square $\{-n,\ldots,n\}^2$ which is covered by infinite geodesics starting at the origin is at most an inverse power of $n$. This result is obtained without explicit limit shape assumptions. 
\end{abstract}


\section{Introduction}
First-passage percolation is a model for a random metric space, formed by a random perturbation of an underlying base space. Since its introduction by Hammersley--Welsh in 1965~\cite{HammersleyWelsh}, it has been studied extensively in the probability and statistical physics literature. We refer to \cite{Kesten:StFlour} for general background and to \cite{auffinger201750} for more recent results.

We study first-passage percolation on the square lattice $(\ZZ^2,E(\mathbb{Z}^2))$ with independent and identically distributed (IID) random environment. The model is specified by a \emph{weight distribution $G$}, a probability measure on the non-negative reals. Each edge $e\in E(\mathbb{Z}^2)$ is assigned a random passage time $t_e$ with distribution $G$, independently between edges. Then, each finite path $p$ in $\mathbb Z ^2$ is assigned the passage time
\begin{equation}\label{eq:time of a path}
    T(p):=\sum _{e\in p} t_e.
\end{equation}
A random metric $T$ on $\ZZ^2$ is defined by setting the \emph{passage time between $u,v\in \mathbb Z ^2$} to
\begin{equation}
    T(u,v):=\inf_{p} T(p),
\end{equation}
where the infimum ranges over all paths connecting $u$ and $v$. Any path achieving the infimum is termed a \emph{geodesic} between $u$ and $v$. A unique geodesic exists when $G$ is atomless (in particular, under our assumption~\eqref{eq:assumption ii} below) and will be denoted $\gamma(u,v)$ (and regarded as a subgraph of $\mathbb Z^2$). The focus of first-passage percolation is the study of the large-scale properties of the random metric $T$ and its geodesics.


\subsection{Results} We proceed to describe our main results. Background and further discussion is provided in Section~\ref{sec:background}.

Throughout we assume that $G$ possesses an exponential moment,
\begin{equation}\label{eq:assumption i}
    \exists \alpha >0 \ \text{ so that } \ \mathbb E \big[ \exp (\alpha t_e) \big]<\infty \tag{EXP} 
\end{equation}
and also that
\begin{equation}\label{eq:assumption ii}
    \quad G\text{ is absolutely continuous}. \tag{ABS}
\end{equation}

The first-order behavior of the metric $T$ is governed by the following result. Define the metric ball of radius $t$ by 
\begin{equation}
B(t):=\big\{\,v\in\ZZ^2:\, T(0,v)\leq t\,\big\}+ \left[-\frac{1}{2},\frac{1}{2}\right] ^2.
\end{equation}

\begin{namedthm*}{Limit Shape Theorem}[Cox and Durrett \cite{CoxDurrett}]
For any distribution $G$ satisfying \eqref{eq:assumption i} and \eqref{eq:assumption ii}, there exists a deterministic convex set $\mathcal B_G$ such that for all $\epsilon >0$,
\begin{equation}
     \mathbb P \Big( \exists t_0>0  \text{ such that }\  \forall t\geq t_0, \quad (1-\ep)\mathcal B_G \subseteq \frac{B(t)}{t}\subseteq (1+\ep)\mathcal B_G \Big) =1.
\end{equation}
\end{namedthm*}

The set $\mathcal B _G$ is called the \emph{limit shape} corresponding to $G$. The theorem holds under weaker assumptions but the above generality suffices for the purposes here.

\subsubsection{Coalescence of geodesics}
Our first result concerns the coalescence of geodesics with nearby starting and ending points. It is shown that, with high probability, such geodesics overlap almost entirely, differing only in short segments near their endpoints. The statement is quantified, obtaining power-law dependence of the involved quantities on the length of the geodesics.

We define $\sides(\mathcal B_G)$ as the number of sides of $\mathcal B_G$: if $\mathcal B_G$ is a polygon then $\sides(\mathcal B_G)$ is its number of edges while if $\mathcal B_G$ is not a polygon then  $\sides(\mathcal B_G):=+\infty$ (equivalently, $\sides(\mathcal B_G)$ is the number of extreme points of $\mathcal B_G$).
Our proof requires a lower bound on $\sides(\mathcal B_G)$.
This assumption is weaker than the condition that the limit shape be strictly convex, a condition which is believed, but not proved, to follow from assumption~\eqref{eq:assumption ii} (see~\cite[Question 11]{auffinger201750}). Theorem~\ref{thm:log over loglog sides} below identifies an explicit class of distributions $G$ satisfying~\eqref{eq:assumption i}, ~\eqref{eq:assumption ii} and the required lower bound on $\sides(\mathcal B_G)$ so that our result holds unconditionally for this class.

We write $\|\cdot\|$ for the $\ell_1$ norm on $\mathbb R^2$. For $n\ge 0$, set \begin{equation*}
    \Lambda_n:=[-n,n]^2\cap \ZZ^2.
\end{equation*}
\begin{thm}\label{mainthm2}
   Suppose $G$ satisfies \eqref{eq:assumption i}, \eqref{eq:assumption ii} and $\sides(\mathcal B_G)>32$. There exists $C>0$ (depending only on $G$) such that for each $0<\epsilon \le1/17$, each $\delta\ge0$ and all $y\in\ZZ^2$ with $\|y\|\ge 2$,
  \begin{equation}
      \mathbb P \Big(\exists u,z\in \Lambda_{\|y\|^{1/8 -\epsilon}}\, \exists v,w\in  (y+\Lambda_{\|y\|^{1/8 -\epsilon}})\quad \!\! \big| \gamma (u,v) \triangle \gamma (z,w) \big| > \|y\|^{1-\delta } \Big) \!\le \! \frac{C\log^2 \|y\|} {\|y\|^{\ep-\delta/8 }},
  \end{equation}
  where $p_1\triangle p_2$ is the set of edges belonging to exactly one of the paths $p_1, p_2$.
\end{thm}

The theorem thus shows that all geodesics which start at distance at most $\|y\|^{1/8-\epsilon}$ from the origin and end at distance at most $\|y\|^{1/8-\epsilon}$ from $y$ coalesce with high probability. In this sense, it is shown that the \emph{coalescence exponent} of first-passage percolation is at least $1/8$. This is the first result establishing the positivity of the coalescence exponent for an explicit class of weight distributions in first-passage percolation (using Theorem~\ref{thm:log over loglog sides}); see Section~\ref{sec:coalescence discussion} for further discussion.

We point out that the coalescence set of two geodesics is necessarily a path. This follows from the fact that there is a unique geodesic between every pair of points.
 See Figure~\ref{fig:galaxy} for simulation results showing the phenomenon of coalescence. 
 
\subsubsection{The influence of edges}

The passage time of the geodesic between given endpoints is naturally a function of the weights assigned to all edges. To what extent is this passage time influenced by the weight assigned to a specific edge? This notion is formalized here by the probability that the geodesic passes through that edge. 
It is clear that the influence of edges near the endpoints cannot be uniformly small, but it is not clear whether the influence must diminish for edges far from the endpoints. This issue was highlighted by Benjamini--Kalai--Schramm~\cite{BKS} in their seminal study of the variance of the passage time, where the following problem, later termed the BKS midpoint problem, was posed: Consider the geodesic between $0$ and $v$. Does the probability that it passes at distance $1$ from $v/2$ tend to zero as $\|v\|\to \infty$? A proof that this probability tends to zero as a power of $\|v\|$ (and analogous estimates for other edges) would simplify the argument of~\cite{BKS}.

The BKS midpoint problem on the square lattice was resolved positively by Damron--Hanson~\cite{damron2017bigeodesics} under the assumption that the limit shape boundary is differentiable and then resolved unconditionally by Ahlberg--Hoffman~\cite{ahlberghoffman}. While both resolutions apply to the more general setup of ergodic edge weights (rather than simply IID), they also share the drawback that no quantitative decay rate for the probability is obtained. As a consequence of our quantitative control on the coalescence of geodesics, we are able to prove power-law decay rates for the influence of edges. These apply, in particular, for the ``midpoint edges'', yielding a quantitative resolution of the BKS midpoint problem.

\begin{thm}\label{prop:probedgebulk}
Suppose $G$ satisfies \eqref{eq:assumption i}, \eqref{eq:assumption ii} and $\sides(\mathcal B_G)>40$. There exists $C>0$ (depending only on $G$) such that for all $u,v,z\in\ZZ^2$,
\begin{equation}\label{eq:influence of vertices}
    \mathbb P\big(z\in\gamma(u,v)\big)\le\frac{C\log^2 (D_z^{u,v}+2)}{(D_z^{u,v})^{1/16}},
\end{equation}
where $D_z^{u,v}:=\min\{\|u-z\|,\|v-z\|\}$. 
\end{thm}
A variant of the result may also be obtained under the weaker assumption $\sides(\mathcal B_G)>32$, see~\eqref{eq:BKS variant}.

It is clear that one cannot have a decay rate in~\eqref{eq:influence of vertices} which is uniform in $z$ at a given distance from $u$ and $v$ and is faster than a power law, since for any integer $0<k\le \frac{1}{2}\|u-v\|$ the geodesic $\gamma(u,v)$ must pass through at least one vertex $z$ with $D_z^{u,v}=k$.

\begin{figure}[htp]
    \centering
    \includegraphics[width=13cm]{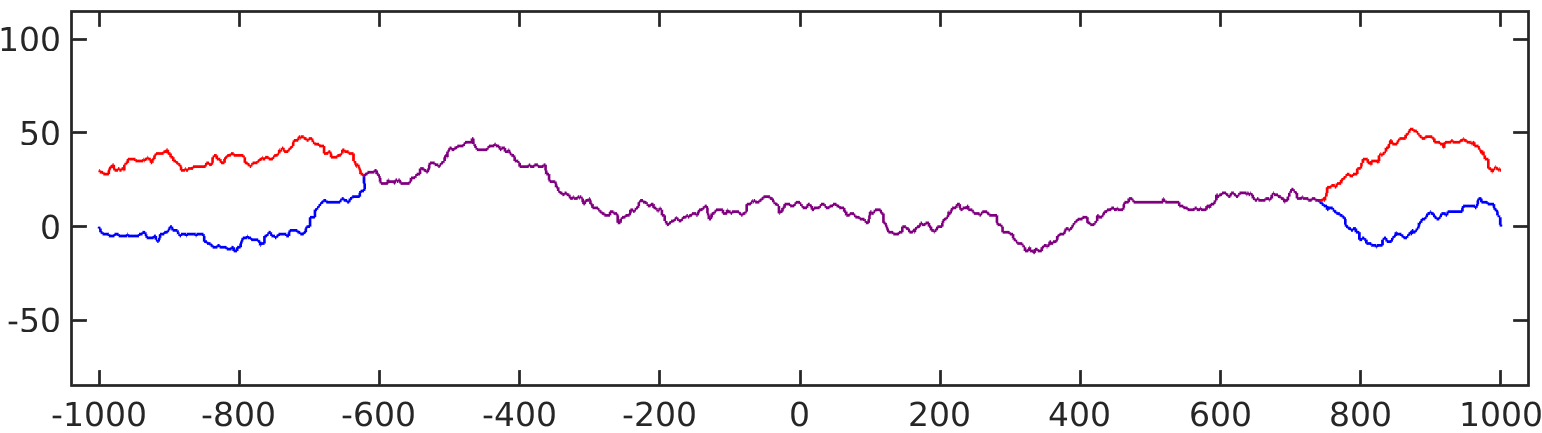}
    \includegraphics[width=13cm]{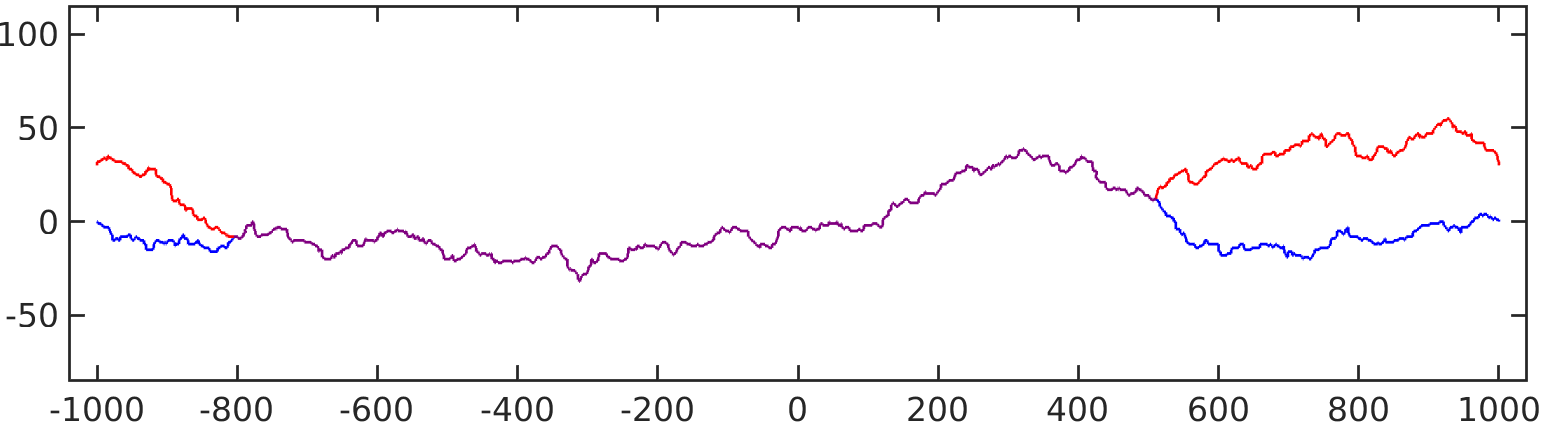}
    \includegraphics[width=13cm]{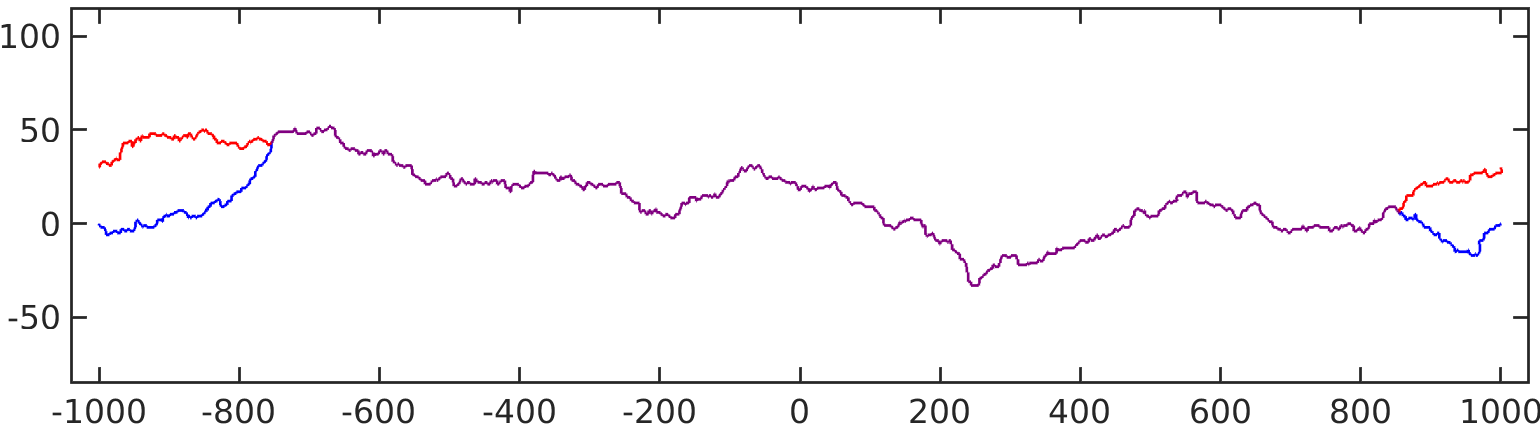}
    \includegraphics[width=13cm]{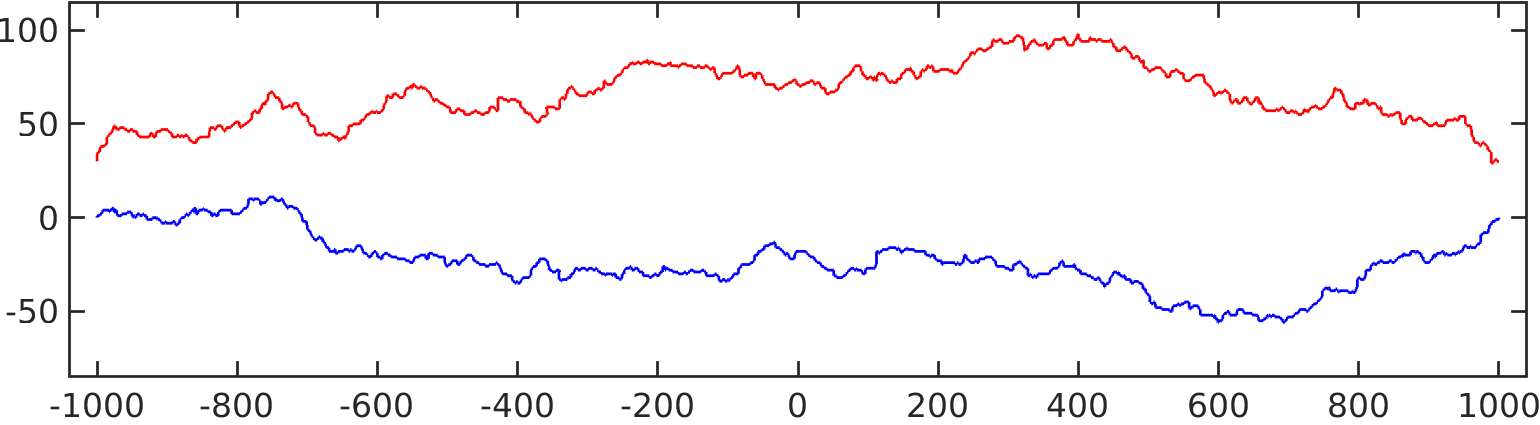}
    \caption{A computer simulation of the geodesics from $(-1000,0)$ to $(1000,0)$ (blue and purple) and from $(-1000,30)$ to $(1000, 30)$ (red and purple) in first-passage percolation on $\mathbb Z ^2$ with weight distribution uniform on $[0,1]$. The pictures depict the geodesics in independent samples of the environment. Theorem~\ref{mainthm2} states that nearby geodesics coalesce with high probability.\\
The geodesics in the fourth simulation did not coalesce and, moreover, were far from each other for most of the way. This is compatible with our results as Proposition~\ref{prop:attractive geodesics intro} shows that geodesics which stay close to each other for a significant amount of time have a very high probability to coalesce.
    \label{fig:galaxy}}
\end{figure}

\smallskip
Theorem~\ref{prop:probedgebulk} implies, as a special case, that the probability that the origin lies on the geodesic between $(-n,0)$ and $(n,0)$ is smaller than a power of $n$. In fact, the method used to derive Theorem~\ref{prop:probedgebulk} allows to prove a stronger fact: the probability that there exists an integer $s$ for which the origin lies on a geodesic from $(-n,s)$ to $(n,s)$ is smaller than a power of $n$. We state this fact in our next theorem.
\begin{thm}\label{thm:density}
Suppose $G$ satisfies \eqref{eq:assumption i}, \eqref{eq:assumption ii} and $\sides(\mathcal B_G)>40$. There exists $C>0$ (depending only on $G$) such that for all integer $n\ge2$,
\begin{equation}
    \mathbb P\Big( (0,0)\in\bigcup_{s\in\ZZ}\gamma \big( (-n,s),(n,s) \big) \Big) \le \frac{C\log ^2n}{n^{1/24}}.
\end{equation}
\end{thm}
Our methods also allow to prove related statements in which the horizontal geodesics are replaced by geodesics with a fixed slope.

We point out that the theorem is related to the well-known problem of proving that there are no infinite bigeodesics (doubly-infinite paths for which every finite sub-path is a geodesic) in first-passage percolation. Indeed, the latter problem can be rephrased as proving that the probability that there exist two points at distance $n$ from the origin such that the geodesic between the points passes through the origin, tends to zero with $n$.

\subsubsection{Highways and byways}
In their seminal paper~\cite{HammersleyWelsh}, Hammersley and Welsh coined the notions of highways and byways. An edge of $\mathbb{Z}^2$ is called a \emph{highway edge} if it belongs to a geodesic of the form $\gamma(0,z)$ for infinitely many values of $z$ (equivalently, if it belongs to an infinite geodesic starting from the origin). Non-highway edges are called \emph{byway edges}. Hammersley and Welsh asked whether the number of highway edges intersecting the circle of radius $R$ around the origin tends to infinity with $R$, and, if so, how fast?

Very recently, Ahlberg--Hanson--Hoffman~\cite{AhlbergHansonHoffman} obtained the first upper bound on the density of highway edges, proving that the probability that a given edge $e$ is a highway edge tends to zero with the distance of $e$ from the origin. This result is proved solely under the assumptions that the weight distribution $G$ is non-atomic and that the minimum of four independent samples from $G$ has a finite second moment. Moreover, the result is proved in a more general setup, when the edge weights $(t_e)_{e\in E(\mathbb{Z}^2)}$ are merely assumed to come from an ergodic distribution, rather than an IID distribution, which satisfies several additional assumptions.

Our next theorem provides the first \emph{quantitative} upper bound on the density of highway edges, showing that the expected proportion of highway edges in $\Lambda_n$ is at most an inverse power of $n$. Moreover, the result applies already to edges lying on long \emph{finite} geodesics. Significantly, in this application of our techniques there is no need for explicit assumptions on the limit shape $\mathcal B_G$.

To state the result, let
\begin{equation}
    \partial \Lambda_n:=\{x\in\ZZ^d : \|x\|_\infty=n\}
\end{equation}
and denote by $\mathcal T_n$ the union of all geodesics from $0$ to a point in $\partial \Lambda_n$, that is
\[\mathcal T_n:=\bigcup_{x\in \partial \Lambda_n}\gamma(0,x).\]

\begin{thm}\label{thm:density2}
Suppose $G$ satisfies~\eqref{eq:assumption i} and~\eqref{eq:assumption ii}. There exists $C>0$ (depending only on $G$) such that for all integer $n\ge 2$,
\begin{equation}
    \mathbb E\left[\frac{|\mathcal T_{4n}\cap \Lambda_n|}{|\Lambda_n|}\right] \le \frac{C\log^2 n}{n^{1/8}}.
\end{equation}
\end{thm}
Section~\ref{sec:assumptions discussion} briefly comments on possible relaxations of our assumptions~\eqref{eq:assumption i} and~\eqref{eq:assumption ii}.

\subsubsection{Many sides to the limit shape}\label{sec:unconditional}

The following theorem identifies a wide class of weight distributions for which the limit shape has many sides (so that the assumptions of Theorem~\ref{mainthm2}, Theorem~\ref{prop:probedgebulk} and Theorem~\ref{thm:density} are satisfied).

\begin{thm}\label{thm:log over loglog sides}
Let $X$ be a random variable supported on $[0,1]$ with $\text{Var}(X)=\sigma ^2 >0$. There exists $\epsilon _0 (\sigma )>0$, depending only on $\sigma$, such that the following holds for all $0<\epsilon <\epsilon _0(\sigma )$. Let $G$ be the distribution of $1+\epsilon X$. Then the limit shape corresponding to $G$ satisfies
\begin{equation}
   \sides(\mathcal B _G) \ge \frac{\log (1/\epsilon )}{\log \log (1/\epsilon )}.
\end{equation}
\end{thm}

Denote by $U[a,b]$ the uniform distribution on the interval $[a,b]$. Using Theorem~\ref{thm:log over loglog sides}, an explicit class of distributions satisfying the assumptions of Theorem~\ref{mainthm2}, Theorem~\ref{prop:probedgebulk} and Theorem~\ref{thm:density} is $G=U[1,1+\epsilon]$ for a sufficiently small $\epsilon>0$ (equivalently, $U[M,M+1]$ for a sufficiently large $M>0$, as multiplying the edge weights by a constant only dilates the limit shape).

\begin{remark}
The proof of Theorem \ref{thm:log over loglog sides} gives not only that there are many sides, but also that there are many sides close to the $(1,0)$ direction. More precisely, we prove that the limit shape has many extreme points between the directions $(1,0)$ and $(1,\sqrt{\ep})$.

We also mention that the proof of Theorem~\ref{thm:log over loglog sides} may be adapted to first-passage percolation on $\mathbb{Z}^d$ with $d>2$, yielding a similar lower bound for the number of sides of the limit shape (defined as the number of hyperfaces if the limit shape is a polytope and infinity otherwise).
\end{remark}

\subsection{Attractive geodesics}\label{c}
The main technical proposition underlying the proofs of our coalescence and highways and byways results is presented in this section (Proposition~\ref{prop:attractive geodesics intro} below). Roughly, it shows that if two geodesics spend significant amount of time near each other then they intersect. Its proof does not rely on the planar geometry and may be adapted also to geodesics in $\ZZ^d$ for $d>2$. Planarity is used when deducing Theorem~\ref{mainthm2} from the proposition, in order to verify that two geodesics with nearby starting and ending points will spend a significant amount of time near each other, with high probability. Planarity is similarly used when deducing Theorem~\ref{thm:density2}.

We wish to make precise the idea that a geodesic $\gamma$ is \emph{attractive} in the sense that any geodesic which spends significant amount of time near $\gamma$ must share an edge with $\gamma$. Our formalization of this idea is in~\eqref{eq:attractive geodesic event}; it requires the following definitions.

Denote $S_A:=A\times\mathbb R$ for a subset $A\subset\mathbb R$. For $x\in\mathbb R$, we shorthand $S_{\{x\}}$ to $S_x$.

For a finite path $p$ in $\ZZ^2$: 
\begin{itemize}
    \item Write $X(p)$ for the interval whose endpoints are the $x$-coordinates of the endpoints of $p$. Precisely, if $p$ has endpoints $(t_1, s_1)$ and $(t_2,s_2)$, with $t_1\le t_2$, then $X(p):=[t_1,t_2]$.
    \item For $x\in X(p)$, let $f_p(x)$ be such that $(x,f_p(x))$ is the \emph{first} intersection point of $p$ with $S_x$; we refer to the points $(x,f_p(x))$ as \emph{pioneer points} of $p$. 
    \item Given $r>0$, the \emph{$r$-tube of (the pioneer points of) $p$} is the set
\begin{equation}
    \tube_r(p):=\{(x,y)\in\ZZ^2\colon x\in X(p), |y-f_p(x)|\le r\}.
\end{equation}
    \item Given an interval $J=[a,b]$ with integer $a,b\in X(p)$ and a second path $q$ in $\ZZ^2$, we say that \emph{$q$ is $r$-close to $p$ on $J$} if the following conditions hold:
\begin{enumerate}
    \item $q$ has a vertex $u\in \tube_r(p)\cap S_a$ and a vertex $v\in \tube_r(p)\cap S_b$.
    \item In the sub-path of $q$ between $u$ and $v$, the number of edges with both endpoints in $\tube_r(p)\cap S_J$ is at least $\frac{1}{2}|J|$.
\end{enumerate}
\begin{figure}[!ht]
\def\svgwidth{1\textwidth}
 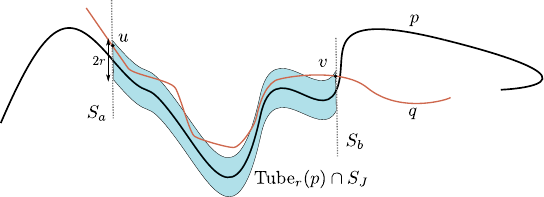
 \caption[figtube]{\label{fig:tube}Illustration of the event that the path $q$ is $r$-close to the path $p$ on the interval $J$. The blue region depicts $\tube_r(p)\cap S_J$. This event will be used in the attractive geodesics proposition, Proposition~\ref{prop:attractive geodesics intro}, where $p$ will have length of order $L$ and $J$ will have length of order $m$.}
\end{figure}
\end{itemize}
The proposition below states that a geodesic is attractive with high probability, provided that it satisfies the following technical requirement of \emph{bounded slope}: For $\rho,m>0$, we say that a finite path $p$ in $\ZZ^2$ has $(\rho,m)$-bounded slope if for all $x_1,x_2\in X(p)$ satisfying $|x_1-x_2|\ge m$ it holds that $|f_p(x_1)-f_p(x_2)|\le \rho |x_1-x_2|$. In words, the slope of $p$ between its pioneer points is bounded above by $\rho$ for every pair of pioneer points with horizontal separation at least $m$. This requirement is discussed further following the statement of the proposition.

\begin{prop}[Attractive Geodesics]\label{prop:attractive geodesics intro} Suppose $G$ satisfies \eqref{eq:assumption i} and \eqref{eq:assumption ii}.
Let $\rho>0$. There exist $C_\rho,c_\rho>0$ and $0<\alpha_\rho\le 1$, depending only on $G$ and $\rho$, such that the following holds.

 Consider the geodesic
\begin{equation}\label{eq:geodesic in attractive geodesic prop}
\gamma:=\gamma((0,0),(L,s))
\end{equation}
for integer $L>0$ and $s$.
Let $(I_i)_{i=0}^{N-1}$ be intervals of the form $I_i = [a_i, a_{i+1}]$ where $0=a_0<a_1<\cdots<a_N=L$ are integers and with $m\le |I_i|\le 2m$ for some $m$. Define the event
\begin{equation}\label{eq:attractive geodesic event}
    \textup{Att}_{r,\xi}:=\left\{ \!\! \begin{array}{c}\text{every geodesic which is $r$-close to $\gamma$ on at least $(1-\xi)N$}\\ \text{of the intervals $(I_i)_{i=0}^{N-1}$ shares an edge with $\gamma$}\end{array} \!\! \right\}.
\end{equation}
Then 
\begin{equation}\label{eq:attractive geodesic probability estimate}
    \mathbb P\big( (\textup{Att}_{r,\xi})^c\cap\{\text{$\gamma$ has $(\rho,m)$-bounded slope}\} \big) \le C_\rho e^{-c_\rho(\log L)^2}
\end{equation}
when
\begin{equation}\label{eq:xi def}
    \xi:=\frac{\alpha_\rho}{\sqrt{r}\log L}
\end{equation}
and the parameters satisfy
\begin{equation}\label{eq:assumptions on r}
    1\le r\le \alpha_\rho\min\left\{\frac{N}{\log^2L},\frac{N^2}{\log^6L}\right\}\quad\text{and}\quad \max\{r,\log^2 L\}\le \alpha_\rho \sqrt{\frac{m}{r}}.
\end{equation}
\end{prop}

In our application, the parameters $r,N,m$ will be chosen as suitable powers of $L$, so that, in particular, assumption~\eqref{eq:assumptions on r} is satisfied.

To deduce from the proposition that geodesics are typically attractive, we need to prove that they typically have $(\rho,m)$-bounded slope. This is handled by the next result, for which we require the following limit shape assumption,
\begin{equation}\label{eq:assumption not ell1}
    \text{the limit shape }\mathcal B _G \text{ is not a dilation of the $\ell_1$ unit ball}.\tag{N$\ell_1$}
\end{equation}
For horizontal geodesics (in the sense of~\eqref{eq:horizontal geodesic} below), the assumption may be waived.

\begin{prop}\label{prop:nobigjumps}Suppose $G$ satisfies \eqref{eq:assumption i}, \eqref{eq:assumption ii} and \eqref{eq:assumption not ell1}. Let $\delta>0$. There exist $C,c,\rho>0$, depending only on $G$, and $C_\delta>0$, depending only on $G$ and $\delta$, such that the following holds.
Consider the geodesic
\begin{equation}
\gamma:=\gamma((0,0),(n,s))
\end{equation}
for integer $n>0$ and $s$.
Assume the `at most 45-degree slope' condition
\begin{equation}\label{eq:at most 45 degree slope}
    |s|\le n.
\end{equation}
Then, for all $m\ge C_\delta n^{\frac{1}{2}+\delta}$,
\begin{equation}\label{eq:regularity prob}
    \mathbb P\big( \text{$\gamma$ does not have $(\rho,m)$-bounded slope}\big) \le Ce^{-c(\log n)^2}\,.
\end{equation}
Moreover, if
\begin{equation}\label{eq:horizontal geodesic}
    m\ge |s|+\sqrt{n}  \log^2 n
\end{equation} then~\eqref{eq:regularity prob} holds also without the assumption~\eqref{eq:assumption not ell1}.
\end{prop}
We can relax the restriction~\eqref{eq:at most 45 degree slope} (to at most $90-\delta$ degree slope) with extra limit shape assumptions.
The proof under condition~\ref{eq:horizontal geodesic} (without assumption~\eqref{eq:assumption not ell1}) uses only the convexity and the lattice symmetries of the limit shape.

We make several remarks regarding the results of this section.

First, the notion of attractive geodesic is not invariant to rotations of $\ZZ^2$, as the $x$-axis plays a special role in the definition of $r$-closeness (the interval $J$ in its definition should be thought of as a subset of the $x$-axis). We can thus define a notion of ``vertically attractive geodesic'' by exchanging the role of the $x$ and $y$ axes in our definitions, and our statements will apply just as well for this notion. This fact is especially relevant for the 45-degree assumption~\eqref{eq:at most 45 degree slope} as one sees that if this assumption is not satisfied by $\gamma$, then it will be satisfied once the $x$ and $y$ axes are exchanged. In this sense~\eqref{eq:at most 45 degree slope} is not a serious restriction.
In the proofs of Theorem~\ref{mainthm2} and Theorem~\ref{thm:density2}, thanks to this symmetry of the lattice, we can assume without loss of generality that the geodesic satisfies the 45-degree assumption. 

Proposition~\ref{prop:nobigjumps} ensures that the geodesic does not make ``big jumps'' with high probability so that, in particular, any sub-path of the geodesic does not make ``big jumps''. In the proof of Theorem~\ref{mainthm2}, we apply Proposition~\ref{prop:nobigjumps} to the whole geodesic, and then apply  Proposition~\ref{prop:attractive geodesics intro} to suitable sub-paths of the geodesic near its endpoints in order to prove that these sub-paths are typically attractive for suitable choices of $r$ and $m$.

A second related observation is that if one first rotates the $\ZZ^2$ lattice by $45$ degrees, thus making the line $y=-x$ into the new $x$-axis, one obtains yet another notion of attractive geodesic (where the interval $J$ in the definition of $r$-closeness should be thought of as a subset of the line $y=-x$ in the original coordinate system). The proofs of our propositions apply also in this rotated coordinate system. It is then worthwhile to note that if the limit shape in the original coordinate system was a dilation of the $\ell_1$ ball then after the rotation the limit shape will be a dilation of the $\ell_\infty$ ball, allowing to apply Proposition~\ref{prop:nobigjumps}. In this sense, a version of our results holds without need to verify assumption~\eqref{eq:assumption not ell1}. This observation is used in our proof of Theorem~\ref{thm:density2} in order to obtain the result without explicit limit shape assumptions.

\subsection{Overview of the proofs}
We briefly explain here how some of our main theorems are proved.

 \begin{figure}[!ht]
\def\svgwidth{0.7\textwidth}
 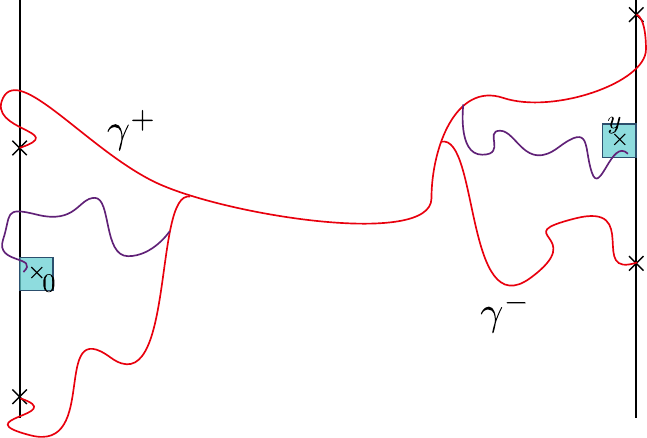
 \caption[figcoalescence]{\label{fig:coalescenceintro}Illustration of the proof of Theorem \ref{mainthm2}}
\end{figure}
Theorem~\ref{mainthm2} shows that geodesics which start near $(0,0)$ and end near a point $y\in\mathbb{Z}^2$ will coalesce with high probability. We prove it using a \emph{trapping strategy}, showing that all such geodesics stay, with high probability, between two reference coalescing geodesics $\gamma^+$ and $\gamma^-$, thereby forcing the coalescence event by the planar geometry (see Figure \ref{fig:coalescenceintro}). The reference geodesic $\gamma^+$ ($\gamma^-$) starts and ends at a suitably chosen distance $h$ above (below) $(0,0)$ and $y$. For the reference geodesics to form a trap, we need to ensure that they stay \emph{ordered}, meaning that $\gamma^+$ is always above $\gamma^-$ (in a suitable sense), and that they stay away from the neighborhoods of $(0,0)$ and $y$. We prove that these properties are satisfied with high probability, when $h$ is somewhat large, using our assumptions on the limit shape (see Proposition \ref{prop:limit}). The coalescence of the reference geodesics is proved using the attractive geodesics proposition, applied to sub-geodesics of $\gamma^-$ of length $L\ll\|y\|$ located at the extremities of $\gamma^-$: To this end, first, Proposition~\ref{prop:nobigjumps} is used to verify that the reference geodesics have bounded slope with high probability (assuming WLOG that $\gamma^-$ satisfies the `45-degree slope' condition as in the first remark after Proposition~\ref{prop:nobigjumps}). Second, the planar geometry, translation invariance of the lattice and the fact that the reference geodesics remain ordered are used to prove that for $r\gg h$, using Markov's inequality, the reference geodesics are $r$-close to each other above each segment with high probability.

Theorem~\ref{prop:probedgebulk} shows that the probability of the event $E_z^{u,v}$ that a vertex $z$ lies on the geodesic between the vertices $u$ and $v$ is small, when $z$ is separated from $u$ and $v$. The theorem is deduced from the coalescence result, Theorem~\ref{mainthm2}, using an \emph{averaging trick} (see Figure \ref{figBKS}) as used in the later proofs of the BKS-type concentration bound by Damron--Hanson--Sosoe \cite{Damron2015}. Translation invariance of the lattice shows that the probability of $E_z^{u,v}$ equals the probability of $E_{z+w}^{u+w,v+w}$ for every $w$. This gives, in particular, that
\begin{equation}\label{eq:averaging trick}
    \mathbb{P}(E_z^{u,v}) = \frac{1}{|\Lambda_\ell|}\sum_{w\in\Lambda_\ell}\mathbb{P}(E_{z+w}^{u+w,v+w}) = \mathbb{E}\bigg[\frac{1}{|\Lambda_\ell|}\sum_{w\in\Lambda_\ell}\mathds 1_{E_{z+w}^{u+w,v+w}}\bigg], 
\end{equation}
where $\Lambda_\ell$ is a discrete square of side length $\ell$. For $w\in\Lambda_\ell$, with $\ell$ suitably small, Theorem~\ref{mainthm2} shows that all geodesics of the form $\gamma(u+w,z+w)$ coalesce and all geodesics of the form $\gamma(z+w, v+w)$ coalesce, with high probability. When this happens, then all geodesics of the form $\gamma(u+w, v+w)$ for which $E_{z+w}^{u+w,v+w}$ occurs (with $w\in\Lambda_\ell$) must coincide on the box $z+\Lambda_\ell$. Therefore, on this event, the quantity inside the expectation in~\eqref{eq:averaging trick} does not exceed the order $\frac{1}{\ell}$ (with high probability, since geodesics only spend order $\ell$ time in $z+\Lambda_\ell$), yielding the required bound.

\begin{figure}[!ht]
\def\svgwidth{0.7\textwidth}
 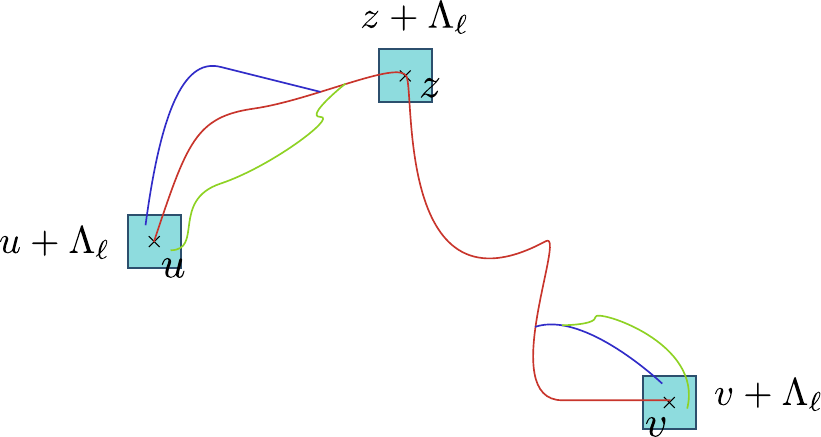
 \caption[figBKS]{\label{figBKS}Illustration of the proof of Theorem \ref{prop:probedgebulk}}
\end{figure}


The highways and byways result of Theorem~\ref{thm:density2} is deduced directly from the attractive geodesics proposition, Proposition~\ref{prop:attractive geodesics intro}, without relying on the coalescence result of Theorem~\ref{mainthm2}. It is first noted that if two geodesics which start at the origin follow non-identical paths in $\Lambda_n$, then their continuations as they exit $\Lambda_n$ must be disjoint (since there is a unique geodesic between every two points). However, by the attractive geodesics proposition, disjoint geodesics cannot be close to each other for a long time. Consequently, due to planarity and the limited area in the annulus $\Lambda_{4n}\setminus \Lambda_n$, it follows that most pairs of geodesics starting at the origin and heading in a similar direction must not separate before exiting $\Lambda_n$. As noted above, no explicit limit shape assumption is used in the proof - while assumption~\eqref{eq:assumption not ell1} is used in the proof of Proposition~\ref{prop:nobigjumps} (and the conclusion of Proposition~\ref{prop:nobigjumps} is needed when applying the attractive geodesics proposition), the assumption may be avoided by also considering $45$-degree rotations of the $\mathbb{Z}^2$ lattice.

The main idea in the proof of Theorem~\ref{thm:log over loglog sides} is to take advantage of the fact that in some directions there are more deterministic paths of a given length. For example, there is a unique path of length $2n$ from $(0,0)$ to $(2n,0)$ while there are $\binom{2n}{n}$ paths of length $2n$ from $(0,0)$ to $(n,n)$. We also use the fact that the weight distribution is a small perturbation of a constant in order to argue that the geodesics are close to being shortest paths in the graph $\mathbb Z ^2$. Using this, we prove that two directions which are not very close cannot be on the same flat edge of the limit shape. 

The proof of the attractive geodesics proposition in Section~\ref{sec:attractive geodesics proof} starts with a main steps part (Section~\ref{sec:main steps}) which may serve as an overview of it. Among the ingredients in the proof is a Mermin--Wagner style argument which is explained in Section~\ref{sec:mermin--wagner}.

\subsection{Discussion, extensions and open problems}\label{sec:background}

\subsubsection{Coalescence of geodesics}\label{sec:coalescence discussion}
Theorem~\ref{mainthm2} proves that geodesics of length $n$ whose starting and ending points are at distance $h=n^{1/8-\ep}$ coalesce with high probability. We briefly review here the literature on similar results.

The most progress has been achieved for ``exactly-solvable models'': \emph{Directed last-passage percolation} models (in two dimensions) for which exact formulas have been found for the basic statistics. In these models the exponent $2/3$ was shown to govern the coalescence (as is also predicted for first-passage percolation). The first result is due to  W\"{u}trich \cite{wuthrich2002asymptotic} who proved that with high probability geodesics will not coalesce at distance $h=n^{2/3+\epsilon }$. This was later improved by Pimentel \cite{pimentel2016duality} to show that the geodesics will not coalesce with uniformly positive probability when $h=Cn^{2/3}$ for any $C>0$. Basu, Sarkar and Sly \cite{basu2019coalescence} proved that $2/3$ is the right exponent by showing that geodesics at distance $Cn^{2/3}$ will coalesce with probability tending to $1$ as $C\downarrow 0 $. Zhang \cite{zhang2020optimal}, and independently Balázs, Busani and Seppäläinen \cite{balazs2021local} and Seppäläinen and Shen \cite{seppalainen2020coalescence} improved the quantitative bounds and estimated the probability of coalescence up to constants as $C\to 0 $ and as $C\to \infty$. See also \cite{seppalainen2021global,basu2019coalescence,Seppalainen,Seppalainen2021busemann,BatesGangulyHammond} for the study of infinite geodesics in exactly-solvable models.

In the \emph{first-passage percolation} setting there are no quantitative and unconditional coalescence results such as Theorem~\ref{mainthm2} (taking into account Theorem~\ref{thm:log over loglog sides}). In fact, the only quantitative result we are aware of is that of Alexander~\cite{alexander2020geodesics} who obtained statements with precise exponents, but under strong assumptions which are currently proved only in the exactly-solvable models (in particular, they are not known to hold for any weight distribution in the first-passage percolation setting). A non-quantitative coalescence result was proved by Licea and Newman~\cite{licea1996geodesics,newman1995surface} for \emph{infinite geodesics} (semi-infinite paths for which every finite sub-path is a geodesic) in two dimensions. They showed that for almost all directions $\theta \in [0,2\pi )$, any two \emph{infinite geodesics} with asymptotic direction $\theta$ must coalesce. 
Their results were strengthened by Damron--Hanson~\cite{damron2014busemann} and Ahlberg--Hoffman~\cite{ahlberghoffman}:
For $\theta\in[0,2\pi)$ denote by $v_\theta$ the unique point in $\partial \mathcal B_G$ in the direction $\theta$. Damron--Hanson proved that in two dimensions, for any direction $\theta\in[0,2\pi)$ such that the limit shape is differentiable at $v_\theta$, there exist no disjoint infinite geodesics with $\theta$ as a direction. 
Ahlberg--Hoffman developed an ergodic theory of random coalescing geodesics in dimension 2. They proved that the properties of coalescence described in \cite{damron2014busemann} are not only valid for some geodesics but are in some sense valid for a dense set of geodesics. 
The results of~\cite{damron2014busemann,ahlberghoffman} are also non-quantitative, but have the advantage of applying to the more general setup of ergodic edge weights. 

We further refer to \cite{newman1997topics,wehr1997number, wehr1998absence, hoffman2005coexistence, hoffman2008geodesics,damron2014busemann, auffinger2015limiting,damron2017bigeodesics,janjigian2019geometry,schmid2021mixing,schmid2022mixing,elboim2022mixing} and the survey \cite{auffinger201750} for additional results on the geometry of geodesics in first- and last-passage percolation.

\smallskip
An interesting direction for extending Theorem~\ref{mainthm2} is to prove a quantitative coalescence result for \emph{infinite} geodesics. To this end, one would naturally need a coalescence result in which the distance to coalescence does not depend on the overall length of the geodesics. The following is an example of such a statement: Let $\gamma_{n,h}$ be the geodesic from $(0,0)$ to $(n,h)$. Prove that for some $C>0$ depending on $G$, universal $\alpha_1,\alpha_2>0$ and all $h,n>0$,
\begin{equation}\label{eq:quick coalescence}
    \mathbb{P}(|\gamma_{n,0}\triangle\gamma_{n,h}|\ge h^{\alpha_1})\le Ch^{-\alpha_2}.
\end{equation}
The obstacles in adapting our argument to prove~\eqref{eq:quick coalescence} are to control the vertical fluctuations of the geodesics (as in Proposition~\ref{prop:nobigjumps}) and to create suitable `traps' for the geodesics as in Section \ref{sec:proofmainthm}. We believe these may be overcome by relying on stronger limit shape assumptions than those in Theorem~\ref{mainthm2} but have not pursued this extension here.

\subsubsection{The influence of edges}
As previously mentioned, the works of Damron--Hanson~\cite{damron2014busemann} and Ahlberg--Hoffman~\cite{ahlberghoffman} provided a non-quantitative resolution of the BKS midpoint problem. In addition, the aforementioned work of Alexander~\cite{alexander2020geodesics} provided quantitative bounds for the BKS problem under strong assumptions which are currently known to hold only in the exactly-solvable models.

\smallskip
The exponent $1/16$ in Theorem \ref{prop:probedgebulk} follows from optimizing between the different parameters in our proof. The correct exponent is expected to be $2/3$, the same as the exponent predicted to govern the transversal fluctuations of geodesics.

\smallskip
 Theorem~\ref{prop:probedgebulk} can be seen as a bound on the ``first-order influence of edges'' in the sense that it bounds the probability that a single given edge is in the geodesic. It is natural to ask also about ``higher-order influences'' in the sense of asking about the probability that several edges are simultaneously in the geodesic. In this direction, we offer the following natural problem: Does the correlation between the choices of first and last edges in the geodesic from $(0,0)$ to $(n,0)$ tends to zero as $n\to \infty $?
 
 \smallskip
 It is conjectured that in two-dimensional first-passage percolation there are no infinite \emph{bigeodesics}, at least when the edge weight distribution $G$ is continuous (the problem originates from Furstenberg; see~\cite[(9.22)]{Kesten:StFlour}). This has been rigorously established in some of the \emph{exactly-solvable last-passage percolation} models \cite{Balazs2020nonexistence,BasuHoffmanSly2022nonexistence,Seppalainen2021busemann,Groathouse2021nonexistence}. In first-passage percolation,
 Alexander~\cite{alexander2020geodesics} proved the non-existence of bigeodesics in all dimensions $d\ge2$ under the strong assumptions mentioned above. Under an assumption on $G$ and on the curvature of the limit shape, Newman \cite{newman1995surface} proved that any infinite geodesic admits almost surely an asymptotic direction.
Licea and Newman~\cite{licea1996geodesics,newman1995surface} rule out the existence of bigeodesics with both ends in fixed directions (outside a set of null measure) in two dimensions.
Their results are strengthened by Damron--Hanson~\cite{damron2014busemann} and Ahlberg--Hoffman~\cite{ahlberghoffman}: they proved that in two dimensions, for any direction $\theta\in[0,2\pi)$ such that the limit shape is differentiable at $v_\theta$, there exist no infinite bigeodesics with $\theta$ as a direction.

\subsubsection{Highways and byways} Besides the work of Ahlberg--Hanson--Hoffman~\cite{AhlbergHansonHoffman} mentioned above, we are aware of only one earlier study of upper bounds in the highways and byways problem. Coupier~\cite{coupier2018sublinearity} considers a class of random trees embedded in $\mathbb{R}^2$ and studies the number of intersection points of a large circle around the origin with semi-infinite paths in the tree. His framework covers both the tree of first-passage paths starting at the (closest point to the) origin in the \emph{isotropic}, Poisson-process based, first-passage percolation model of Howard and Newman~\cite{howard1997euclidean} and the tree of last-passage paths from the origin in directed last-passage percolation on the lattice. In both cases, Coupier obtains a non-quantitative upper bound on the number of intersections, of a similar flavor to that obtained by Ahlberg--Hanson--Hoffman, with the bound in the directed last-passage percolation case proved under the assumption that the limit shape is strictly concave and differentiable. Coupier further discusses a random tree of a very different nature, formed by \emph{local rules}. For this tree, which he terms radial Poisson tree, a quantitative, power-law upper bound on the number of intersections is obtained. 


\subsubsection{Limit shape properties}\label{sec:discussion limit shape}
Theorem~\ref{thm:log over loglog sides} proves that the limit shape has many sides for a particular class of distributions. We are only aware of few related results in the literature, as we now describe.

Damron--Hochman~\cite{DamronHochman}, relying on results of Marchand~\cite{marchand} and Kesten~\cite{Kesten:StFlour}, construct an \emph{atomic} distribution whose limit shape has an infinite number of sides. 

Basdevant--Gouéré--Théret~\cite{basdevant2021first} determined the first-order behavior of the limit shape corresponding to the weight distribution $\mathrm{Bernoulli}( 1-\ep)$ as $\ep$ tends to $0$. Using their result one can show that the number of extreme points of the limit shape corresponding to $\text{Bernoulli}(1-\ep)$ tends to infinity as $\epsilon \downarrow 0$.

\subsubsection{Related models}
Analogs of our main results (Theorem~\ref{mainthm2},  Theorem~\ref{prop:probedgebulk} and Theorem~\ref{thm:density}) continue to hold for point-to-line geodesics (i.e., geodesics from nearby starting points to the same line will coalesce with high probability), requiring only notational modifications in the proof. 

The proofs of our main results should also extend, with minimal changes, to directed first- and last-passage percolation models in a planar geometry (under the added condition that the slope between the starting and ending points of the geodesics under study is bounded away from the maximal and minimal allowed slopes). In fact, the proofs should simplify in this setting, since if a directed geodesic starts and ends above another directed geodesic then they must deterministically preserve their order throughout.

Lastly, one may also consider first-passage percolation in the ``slab'' $S_{[0,n]}$  for some integer $n>0$ (this may be thought of as first-passage percolation on $\mathbb{Z}^2$ in which the weights of the edges not fully contained in $S_{[0,n]}$ are set to infinity). Analogs of our main results can also be proved in this setting, with minimal modifications to our proofs, for geodesics connecting the sides of the slab (i.e., starting at $(0,y_1)$ and ending at $(n,y_2)$ for some $y_1, y_2$). Similarly to the directed models, geodesics connecting the sides of the slab deterministically preserve their ordering in this setting. The lower bound on the required number of sides of the limit shape in this setting stems solely from the analog of Proposition~\ref{prop:nobigjumps}, so the coalescence result should hold \emph{without any limit shape assumptions} in the case of horizontal geodesics (i.e., geodesics satisfying~\eqref{eq:horizontal geodesic}).

\subsubsection{Higher dimensions and minimal surfaces}
For first-passage percolation on $\mathbb{Z}^d$ with dimension $d\ge 3$, Proposition \ref{prop:attractive geodesics intro} remains true with minimal change to the proof, as long as condition~\eqref{eq:assumptions on r} is suitably modified. Among other things, the new condition needs to imply that the ``cost'' to connect two geodesics separated by distance $r$, which is of order $r$ in the way we argue (in any dimension), is smaller than the increase generated by the Mermin--Wagner style argument (Section \ref{sec:mermin--wagner})), which is of order $\sqrt{m/r^{d-1}}$.

However, planarity is crucially used in the proof of Theorem~\ref{mainthm2} to keep the ordering between the geodesics (i.e., to show that if a geodesic has its endpoints above those of another geodesic then it will very likely remain above the other geodesic throughout). Ordering, in turn, is used to ensure (using Markov's inequality) that geodesics with nearby starting and ending points spend significant time near each other with high probability (so that Proposition \ref{prop:attractive geodesics intro} is applicable). As the ordering is lost in dimensions $d\ge3$, we do not know how to apply Proposition \ref{prop:attractive geodesics intro} in order to deduce coalescence.

\smallskip
We mention that while the BKS midpoint problem is open in dimensions $d\ge 3$, partial results are available~\cite{dembin2023ell} as well as results under assumptions which are still unverified~\cite{alexander2020geodesics}.

\smallskip
The above regards first-passage percolation on $\mathbb{Z}^d$ for $d\ge 3$. There is also a different extension of first-passage percolation to higher dimensions, in which ``time'' is taken to be higher dimensional. In this version the object of study is a \emph{minimal surface in a random environment}. Such minimal surfaces model the domain walls in the disordered ferromagnet (the random-bond Ising model); see~\cite[Section 1.1]{bassan2023non} and~\cite[Section 1.4]{dembin2024minimal} for background and~\cite{dembin2024superconcentration,bassan2023non,dembin2024minimal} for recent work on transversal and ground energy fluctuations. The problems of coalescence and midpoint delocalization make sense also in this context and would be interesting to explore.

\subsubsection{The assumptions}\label{sec:assumptions discussion}
The assumption~\eqref{eq:assumption i} is used to ensure that the probability that the passage time between $u,v\in\ZZ^2$ is larger than $\rho \|u-v\|$, for some constant $\rho>0$, is exponentially small in $\|u-v\|$ (see Claim~\ref{claim:time}). Weaker decay rates may also suffice in our arguments.

The second assumption~\eqref{eq:assumption ii} is mostly used in Claim~\ref{claim:real analysis} as part of the proof of our Mermin--Wagner style result. It may be possible to push our arguments to a class of non absolutely-continuous distributions but we have not attempted to do so.

The assumed lower bound on the number of sides of the limit shape is required in order to have sufficient control on the geometry of geodesics to produce the ``traps'' used in the proof of Theorem~\ref{mainthm2}. In particular, we want to ensure that a geodesic that has endpoints far above another geodesic is unlikely to go below that other geodesic.

\subsection{Reader's guide}
The rest of the paper is organized as follows. In the next section we prove Proposition~\ref{prop:attractive geodesics intro}, which is the main technical ingredient in our proofs.
Theorem~\ref{mainthm2}, Theorem~\ref{prop:probedgebulk}, Theorem~\ref{thm:density} and Theorem~\ref{thm:density2} are then deduced in Section~\ref{sec:proofmainthm}. The first three of these theorems further require control on the amount of time that a geodesic spends ``going in a wrong direction'', as stated in Proposition~\ref{prop:limit}. This control is achieved in Section~\ref{sec:the limit shape} where we study the geometry of geodesics and prove Proposition~\ref{prop:nobigjumps} and Proposition~\ref{prop:limit} using our assumptions on the limit shape. Finally, in Section~\ref{sec:log} we establish the lower bound on the number of sides of the limit shape given in Theorem~\ref{thm:log over loglog sides}. The latter proof is independent of the rest of the paper.

\section{Proof of the attractive geodesics proposition}\label{sec:attractive geodesics proof}
In this section we prove Proposition~\ref{prop:attractive geodesics intro}. We assume that $L$ is sufficiently large for the arguments (as a function of the distribution $G$ and the parameter $\rho$) as the constants in the proposition may be adjusted to fit smaller $L$. We also assume throughout that $G$ satisfies~\eqref{eq:assumption i} and~\eqref{eq:assumption ii} and we continue with the notation of the proposition.

\subsection{Main steps}\label{sec:main steps} In this section we give an overview of the proof of the proposition, postponing the proofs of some of the intermediate steps to later sections.
\subsubsection{Attractive intervals}
 Consider the geodesic
\begin{equation}
\gamma:=\gamma((0,0),(L,s))
\end{equation}
for integer $L>0$ and $s$.
Call an interval $J\subset[0,L]$ with integer endpoints \emph{attractive} if every geodesic which is $r$-close to $\gamma$ on $J$ necessarily has an edge in common with $\gamma$. The following containment of events is immediate,
\begin{equation}\label{eq:attractive intervals containment}
    \left\{\sum_{i=0}^{N-1} \mathds 1_{\{I_i\text{ is attractive}\}}>\xi N\right\}\subset \textup{Att}_{r,\xi}.
\end{equation}
We thus focus on giving a lower bound for the probability of the left-hand side event. As the first step we develop a sufficient condition for an interval to be attractive.

The following basic bound, controlling the passage time and length of geodesics, will be helpful. 
\begin{lem}\label{lem:basic geometric control}
There exist $C,c,\rho_1,\rho_2>0$, depending only on $G$, such that the following holds. Let $\OB$ be the event that for all $u,v\in[-L^2,L^2]^2\cap\ZZ^2$ 
it holds that
\begin{equation}\label{eq:passage time and length estimates}
    T(u,v)\le \rho_1\max\{\|u-v\|_1,\log^2 L\}\quad\text{and}\quad|\gamma(u,v)|\le\rho_2\max\{\|u-v\|_1,\log^2 L\},
\end{equation}
where we write $|p|$ for the number of edges in a path $p$. Then
\begin{equation}
    \mathbb P(\OB)\ge 1 - Ce^{-c\log^2 L}.
\end{equation}
\end{lem}
The lemma is proved in Section \ref{sec:auxiliary lemmas}. The notation $\rho_1,\rho_2$ is reserved throughout our argument to the constants from the lemma.

To make use of the lemma for the geodesic $\gamma$, we first note that when $\gamma$ has $(\rho,m)$-bounded slope then $|f_\gamma(L)|\le\rho L$. Thus, as we've assumed that $L$ is large as a function of $\rho$, the subgeodesic of $\gamma$ between $(0,0)$ and $(L,f_\gamma(L))$ satisfies the estimates~\eqref{eq:passage time and length estimates}. 

Let
\begin{equation}
    \mathcal{P}:=\{p\colon \text{$p$ is a path in $\ZZ^2$ with $\mathbb P(\gamma=p)>0$}\},
\end{equation}
so that, in particular, each $p\in\mathcal{P}$ is a path from $(0,0)$ to $(L,s)$ and $\gamma\in\mathcal{P}$ almost surely.

Let $J=[a,b]\subset[0,L]$ be an interval with integer endpoints and $p\in\mathcal{P}$. Write
\begin{equation}\label{eq:T p J def}
    T_p(J):=T((a,f_p(a)), (b,f_p(b)))
\end{equation}
for the passage time from the pioneer point of $p$ above $a$ to the pioneer point of $p$ above $b$ (using the geodesic between these two points, which may differ from $p$).

A central role in our analysis is played by the following notion of the \emph{restricted passage time} $\bar{T}_p(J)$, defined as the minimal passage time among (simple) paths $q$ satisfying
\begin{enumerate}
    \item $q$ is edge-disjoint from $p$.
    \item One endpoint of $q$ is in $\tube_r(p)\cap S_a$ and the other is in $\tube_r(p)\cap S_b$.
    \item The number of edges of $q$ with both endpoints in $\tube_r(p)\cap S_J$ is at least $\frac{1}{2}|J|$.
    \item $|q|\le \rho_2\max\{\|u-v\|_1, \log^2 L\}$ where $u$ and $v$ are the endpoints of $q$.
\end{enumerate}
(setting $\bar{T}_p(J):=\infty$ if no such path exists). The set of paths satisfying these properties is denoted by $Q_p(J)$.

The following is our sufficient condition for the attractiveness of $J$: Let
\begin{equation}\label{eq:definition of Omega(J)}
    \Omega(J):=\{\bar{T}_\gamma(J)> T_\gamma(J)+2\rho_1 \max\{r,\log^2 L\}\}.
\end{equation}
Then
\begin{equation}\label{eq:sufficient condition for attrativeness}
    \text{on $\Omega(J)\cap\OB\cap\{\gamma$ has $(\rho,m)$\text{-bounded slope}$\}$ it holds that $J$ is attractive}.
\end{equation}
Let us prove this. Assume $\OB$ and assume that $\gamma$ has $(\rho,m)$-bounded slope. We show that the existence of a geodesic $\gamma'$ which is $r$-close to $\gamma$ on $J$ and is edge-disjoint from $\gamma$ implies that $\Omega(J)$ does not occur. First, it follows from the properties of $\gamma'$ that it contains a subgeodesic $\gamma''$ connecting some $(a,y_1)\in \tube_r(\gamma)\cap S_a$ to some $(b,y_2)\in \tube_r(\gamma)\cap S_b$ which satisfies properties (1),(2),(3) above with $q=\gamma''$ and $p=\gamma$. Moreover, we claim that $\gamma''$ also satisfies property (4) so that it belongs to $Q_\gamma(J)$. This follows from $\OB$, as the endpoints of $\gamma''$ are in $[-L^2,L^2]^2$ by our upper bound~\eqref{eq:assumptions on r} on $r$ (with $\alpha_\rho\le 1$, say) and since $\gamma$ has $(\rho,m)$-bounded slope. Second, since $\gamma''$ is a geodesic, its passage time must be at most that of the path going along the geodesic from $(a,y_1)$ to $(a,f_\gamma(a))$, then along $\gamma$ from $(a,f_\gamma(a))$ to $(b,f_\gamma(b))$ and finally along the geodesic from $(b,f_\gamma(b))$ to $(b,y_2)$. On $\OB$, the latter path has passage time at most $T_\gamma(J)+2\rho_1 \max\{r,\log^2 L\}$. Since $\bar{T}_\gamma(J)\le T(\gamma'')$, we conclude that $\Omega(J)$ does not hold.

With the sufficient condition~\eqref{eq:sufficient condition for attrativeness} in hand, and taking into account the containment~\eqref{eq:attractive intervals containment} and Lemma~\ref{lem:basic geometric control}, we see that Proposition~\ref{prop:attractive geodesics intro} follows from the following statement: There exist $C_\rho,c_\rho>0$, depending only on $G$ and $\rho$, such that
\begin{equation}\label{eq:attractive interval prob2}
    \mathbb{P}\left(\left\{\sum_{i=0}^{N-1} \mathds 1_{\Omega(I_i)}\le\xi N\right\}\cap\{\gamma\text{ has $(\rho,m)$-bounded slope}\}\right)\le C_\rho e^{-c_\rho \log^2 L}.
\end{equation}

The next sections present the proof of this estimate, which relies on the following ingredients: 
\begin{enumerate}
    \item An upper bound for the passage time $T_\gamma(I_i)$ of many of the intervals $(I_i)$. The main observation here is that Talagrand's concentration inequality self-improves when applied to sub-geodesics of $\gamma$ due to the concavity of the square root function. 
    \item A lower bound for the probability that a restricted passage time $\bar{T}_p(I_i)$ is long. This uses a Mermin--Wagner style argument (perturbing the edge passage times) to obtain lower bounds on the fluctuations of $\bar{T}_p(I_i)$.
    \item The Harris correlation inequality for monotonic events in independent variables.
\end{enumerate}

\subsubsection{The passage time of $\gamma$ on many of the intervals $I_i$ is short}

For $u,v\in\ZZ^2$, write
\begin{align}
    E(u,v) &:= \mathbb E[T(u,v)],\label{eq:expected passage time def}\\
    D(u,v) &:= T(u,v) - E(u,v)
\end{align}
for the expected passage time between $u$ and $v$ and the deviation from the expectation. We may note that $E$ is a (deterministic) metric on $\ZZ^2$, since $T$ is a (random) metric on $\ZZ^2$. Talagrand's concentration inequality provides the following control on $D(u,v)$.
\begin{lem}\label{lem:Talagrand control}
There exist $C,c>0$, depending only on $G$, such that the following holds. Let $\OT$ be the event that for all $u,v\in[-L^2,L^2]^2\cap\ZZ^2$,
\begin{equation}\label{eq:Talagrands bound}
        |D(u,v)|\le \sqrt{\|u-v\|_1}\log L.
\end{equation}
Then
\begin{equation}
    \mathbb P(\OT)\ge 1 - Ce^{-c\log^2 L}.
\end{equation}
\end{lem}
The lemma is proved in Section \ref{sec:auxiliary lemmas}.

The observation made in this section, stated in~\eqref{eq:deviation improvement on average} below, is that~\eqref{eq:Talagrands bound} may be improved `on average' for sub-geodesics of $\gamma$ due to the concavity of the square root function.

For an interval $J=[a,b]\subset[0,L]$ with $a,b\in\ZZ$ and a path $p\in\mathcal{P}$, let
\begin{equation}
    E_p(J):=E((a,f_p(a)), (b,f_p(b)))
\end{equation}
be the expected passage time between the pioneer points of $p$ above the endpoints of $J$. We think of $E_p(J)$ as a deterministic function of the path $p$ and when we write $E_\gamma (J)$ we simply substitute the random path $\gamma $ inside this function (so that $E_\gamma(J)$ is a random variable, different from the deterministic quantity $\mathbb E[T_\gamma(J)]$). 
Define $p[J]$ to be the subpath of $p$ between the points $(a,f_p(a))$ and $(b,f_p(b))$, and define
\begin{equation}\label{eq:Dp(J) def}
    D_p(J):=T(p[J])-E_p(J)\,
\end{equation}
where the time of a path was defined in \eqref{eq:time of a path}. Note that $T_\gamma(J)=T(\gamma[J])$ almost surely as $\gamma$ is a geodesic.
The quantity $D_\gamma(J)$ is a measure of the deviation of the passage time of the sub-geodesic of $\gamma$ between the pioneer points at $a$ and $b$.

It is straightforward to check that the following statements hold almost surely,
\begin{align}
    &T_\gamma([0,L])=\sum_{i=0}^{N-1} T_\gamma(I_i),\\
    &E_\gamma([0,L]) \le \sum_{i=0}^{N-1} E_\gamma(I_i)
\end{align}
(the inequality follows since $E(\cdot,\cdot)$ is a metric).
Consequently,
\begin{equation}\label{eq:deviation improvement on average}
    \sum_{i=0}^{N-1} D_\gamma(I_i)\le D_\gamma([0,L]) = D((0,0),(L,f_\gamma(L)))\quad\text{almost surely}.
\end{equation}
We conclude that on $\OT\cap\{\gamma\text{ has }(\rho,m)\text{-bounded slope}\}$,
\begin{equation}\label{eq:deviation bounds for gamma}
    \sum_{i=0}^{N-1} D_\gamma(I_i)\le \sqrt{(1+\rho)L}\log L\quad\text{and}\quad
    \min_{0\le i\le N-1} D_\gamma(I_i)\ge -\sqrt{2m(1+\rho)}\log L.
\end{equation}

The assertion~\eqref{eq:deviation bounds for gamma} is harnessed in the following way. It is straightforward that if~\eqref{eq:deviation bounds for gamma} holds then for each $\tau$ satisfying \begin{equation}\label{eq:tau requirement}
\tau\ge\frac{\sqrt{(1+\rho)L}\log L}{N}
\end{equation}
we have that either $\Omega_{1,\tau}$ or $\Omega_{2,\tau}$ occurs, with
\begin{alignat}{1}
    \Omega_{1,\tau}&:=\left\{\left|\left\{0\le i\le N-1\colon D_\gamma(I_i)\le 6\tau\right\}\right|\ge \frac{1}{2} N\right\},\label{eq:Omega 1 tau def}\\
    \Omega_{2,\tau}&:=\left\{\sum_{i\colon -\sqrt{2m(1+\rho)}\log L\le D_\gamma(I_i)<-\tau}D_\gamma(I_i)\le -\tau N\right\}.\label{eq:Omega 2 tau def}
\end{alignat}
We will use this conclusion with $\tau = \sqrt{\frac{m}{r}}$, noting that~\eqref{eq:tau requirement} is satisfied due to our assumption~\eqref{eq:assumptions on r} (choosing $\alpha_\rho$ sufficiently small). For a path $p\in\mathcal{P}$, it will be convenient to denote by $\Omega_{1,\tau}(p)$ and $\Omega_{2,\tau}(p)$ the events appearing in~\eqref{eq:Omega 1 tau def} and~\eqref{eq:Omega 2 tau def}, respectively, in which all occurrences of $D_\gamma(I_i)$ are replaced by $D_p(I_i)$.

\subsubsection{The restricted passage time is long with non-negligible probability}

In this section we provide lower bounds for the probability that a restricted passage time $\bar{T}_p(I_i)$ is long and further discuss the independence properties of the restricted passage times.

\begin{lem}\label{lem:long restricted passage time}
There exists $c_\rho>0$, depending only on $G$ and $\rho$, such that the following holds. 
For each path $p\in\mathcal{P}$ having $(\rho,m)$-bounded slope, each $0\le i\le N-1$ and each $5\rho_1 \max\{r,\log^2 L\}\le t\le \sqrt{2m(1+\rho)}\log L$,
\begin{align}
    &\mathbb P\left(\bar{T}_p(I_i)\ge E_p(I_i) - t+2\rho_1 \max\{r,\log^2L\}\right)\ge \frac{c_\rho t}{\sqrt m\log L},\label{eq:Talagrand lower bound}\\
    &\mathbb P\left(\bar{T}_p(I_i)\ge E_p(I_i) + 6\sqrt{\frac{m}{r}}+2\rho_1 \max\{r,\log^2 L\}\right)\ge \frac{c_\rho}{\sqrt{r}\log L}.\label{eq:MW lower bound}
\end{align}
\end{lem}
The proof of Lemma~\ref{lem:long restricted passage time}, relying on Talagrand's concentration inequality, is given in Section~\ref{sec:auxiliary lemmas}. 
The proof of~\eqref{eq:MW lower bound} additionally uses a ``Mermin--Wagner style argument'' developed in Section~\ref{sec:mermin--wagner}. On an intuitive level, the argument yields that the distribution of $\bar{T}_p(I_i)$ ``contains a Gaussian component with variance of order $\frac{m}{r}$'' (see Lemma~\ref{lem:MW for T p J} for the precise result). This implies the following statement.
\begin{lem}\label{lem:MW corollary for T p J}
Let $p\in\mathcal{P}$ have $(\rho,m)$-bounded slope and let $J\subset[0,L]$ be an interval with integer endpoints satisfying $m\le |J|\le 2m$. There exist $C_\rho,c_\rho>0$, depending only on $G$ and $\rho$, such that for each $0\le\alpha\le c_\rho\sqrt{m r}$ and each real $a$,
\begin{equation}\label{eq:probability for increasing weights}
    \sqrt{\mathbb{P}\left(\bar{T}_p(J)\ge a+\alpha\sqrt{\frac{m}{r}}\right)}\ge e^{-C_\rho\alpha^2} \big( \mathbb{P}\left(\bar{T}_p(J)\ge a\right) - e^{-m}\big) .
\end{equation}
\end{lem}
The argument leading from Lemma~\ref{lem:MW corollary for T p J} to the inequality~\eqref{eq:MW lower bound} is explained in Section~\ref{sec:auxiliary lemmas}. 

We remark that the ``standard deviation lower bound $\sqrt{\frac{m}{r}}$'' is obtained as the ratio between the length of $J$ and the square root of the volume of the $r$-tube of $p$ above $J$ (using in the process that paths in $Q_p(J)$ must spend a significant fraction of their time in the $r$-tube). This is of the same nature as the relation $\chi\ge \frac{1 - (d-1)\xi}{2}$ on $\mathbb{Z}^d$, between the fluctuation exponent $\chi$ and transversal exponent $\xi$, obtained by Wehr--Aizenman~\cite[Section 6]{wehr1990fluctuations} and Newman--Piza~\cite[Theorem 5]{newman1995divergence}. Our arguments may also be used to obtain such a relation. 

We also remark that it would have been helpful to know the natural fact that $T_p(I_i)\ge E_p(I_i)$ occurs with probability bounded away from zero uniformly in $L,m,r$. Such a fact would both simplify and lead to better probability lower bounds in Lemma \ref{lem:long restricted passage time}.

\smallskip

Recall that our goal is to prove the probability bound~\eqref{eq:attractive interval prob2}. This requires showing that several of the restricted passage times $\bar{T}_\gamma(I_i)$ are simultaneously large. To this end, the following independence property is handy: Set $\rho_3:=2\rho_2(1+\rho)$. For each $p\in \mathcal P$ and subset $\mathcal{I}\subset\{0,1,\ldots, N-1\}$,
\begin{equation}\label{eq:independence of restricted passage times}
    \text{if $|i_1-i_2|\ge 2\rho_3$ for all distinct $i_1,i_2\in\mathcal{I}$ then }\left(\bar{T}_p(I_i)\right)_{i\in\mathcal{I}}\text{ are independent}.
\end{equation}
Indeed, recall that $\bar{T}_p(I_i)$ is the minimal passage time among the paths in $Q_p(I_i)$, and that the paths $q\in Q_p(I_i)$ have endpoints with $x$ coordinates $a_i$ and $a_{i+1}$ and satisfy \[|q|\le\rho_2((a_{i+1} - a_i )(1+\rho)+2r)\le\rho_3(a_{i+1}-a_i)\] (using here that $a_{i+1}-a_i\ge m\ge\log^2 L$ and $2r\le m\le a_{i+1}-a_i $ by~\eqref{eq:assumptions on r} with $\alpha_\rho\le 1$). Such paths $q$ thus stay in the slab $S_{[a_i - \alpha,\, a_{i+1} + \alpha]}$ with $\alpha = \frac{1}{2}(\rho_3-1)(a_{i+1}-a_i)\le(\rho_3-1)m$. Thus, $\bar{T}_p(I_i)$ is measurable with respect to the weights of the edges with both endpoints in $S_{[a_i - (\rho_3-1) m,\, a_{i+1}+(\rho_3-1) m]}$, from which~\eqref{eq:independence of restricted passage times} follows as edges have independent weights.
\smallskip
Lemma~\ref{lem:long restricted passage time} and the independence property~\eqref{eq:independence of restricted passage times} will be used in the following way. For a path $p\in\mathcal{P}$ having $(\rho,m)$-bounded slope and a vector of reals $d = (d_i)_{i=0}^{N-1}$ define the event
\begin{equation}\label{eq:Epd def}
    \mathcal{E}_{p,d}:=\left\{\sum_{i=0}^{N-1} \mathds 1_{\bar{T}_p(I_i)> E_p(I_i) + d_i+2\rho_1 \max\{r,\log^2 L\}}\le\xi N\right\}
\end{equation}
(where we recall from~\eqref{eq:xi def} that $\xi = \frac{\alpha_\rho}{\sqrt{r}\log L}$). The following bounds the probability of $\mathcal{E}_{p,d}$.
\begin{prop}\label{prop:probability for long restricted passage times}
If $\alpha_\rho$ is sufficiently small (as a function of $G$ and $\rho$) then
\begin{equation}
    \mathbb P(\mathcal{E}_{p,d})\le e^{-\frac{1}{4} \xi N}
\end{equation}
when 
\begin{equation}\label{eq:assumptions on d}
\begin{split}
    &\text{either}\;\;\left|\left\{0\le i\le N-1\colon d_i\le 6\sqrt{\frac{m}{r}}\right\}\right|\ge \frac{1}{2}N\;\;\\
    &\text{or}\;\;\sum_{i\colon -\sqrt{2m(1+\rho)}\log L\le d_i<-\sqrt{\frac{m}{r}}} d_i\le -\sqrt{\frac{m}{r}}N.
\end{split}
\end{equation}
\end{prop}
The proof uses the following special case of Chernoff's bound (see, e.g.,~\cite[equation (7)]{hagerup1990guided}). Let $X_1,\ldots, X_k$ be independent random variables taking values in $\{0,1\}$ and write $\mu:=\sum_{i=1}^k\mathbb{E}(X_i)$ for the expectation of their sum. Then
\begin{equation}\label{eq:Chernoff bound}
    \mathbb{P}\left(\sum_{i=1}^k X_i\le\frac{1}{2}\mu\right)\le e^{-\frac{1}{8}\mu}.
\end{equation}
\begin{proof}[Proof of Proposition~\ref{prop:probability for long restricted passage times}]If the first condition in~\eqref{eq:assumptions on d} holds then there exists a subset $\mathcal{I}\subset\{0,1\ldots, N-1\}$ such that $d_i\le 6\sqrt{\frac{m}{r}}$ for $i\in\mathcal{I}$, $|i_1-i_2|\ge 2\rho_3$ for all distinct $i_1,i_2\in\mathcal I$ and $|\mathcal{I}|\ge \frac{N}{2\lceil 2\rho_3\rceil}$. The variables $(\bar{T}_p(I_i))_{i\in\mathcal{I}}$ are then independent by~\eqref{eq:independence of restricted passage times}. Thus, by~\eqref{eq:Chernoff bound},
\begin{equation}
    \mathbb P(\mathcal{E}_{p,d})\le \mathbb{P}\left(\sum_{i\in\mathcal{I}} \mathds 1_{\bar{T}_p(I_i)\ge E_p(I_i) + 6\sqrt{\frac{m}{r}}+2\rho_1 \max\{r,\log^2 L\}}\le\xi N\right)\le e^{-\frac{1}{8}\mu_1}\le e^{-\frac{1}{4}\xi N},
\end{equation}
where $\mu_1:=\sum_{i\in\mathcal{I}}\mathbb{P}(\bar{T}_p(I_i)\ge E_p(I_i) + 6\sqrt{\frac{m}{r}}+2\rho_1 \max\{r,\log^2 L\})\ge \frac{c_\rho N}{2\lceil2\rho_3\rceil\sqrt{r}\log L}$ by~\eqref{eq:MW lower bound} and where we use that $\mu_1\ge 2\xi N$ when $\alpha_\rho$ is chosen sufficiently small (here $c_\rho$ is the constant from~\eqref{eq:MW lower bound}).

Similarly, if the second condition in~\eqref{eq:assumptions on d} holds then there exists $\mathcal{I}\subset\{0,1\ldots, N-1\}$ such that $-\sqrt{2m(1+\rho)}\log L\le d_i<-\sqrt{\frac{m}{r}}$ for $i\in\mathcal{I}$, $|i_1-i_2|\ge 2\rho_3$ for all distinct $i_1,i_2\in\mathcal I$ and $\sum_{i\in\mathcal{I}}d_i \le -\frac{1
}{\lceil2\rho_3\rceil}\sqrt{\frac{m}{r}} N$. The variables $(\bar{T}_p(I_i))_{i\in\mathcal{I}}$ are again independent by~\eqref{eq:independence of restricted passage times}. Thus, by~\eqref{eq:Chernoff bound},
\begin{equation}
    \mathbb P(\mathcal{E}_{p,d})\le \mathbb{P}\left(\sum_{i\in\mathcal{I}} \mathds 1_{\bar{T}_p(I_i)\ge E_p(I_i) + d_i+2\rho_1 \max\{r,\log^2 L\}}\le\xi N\right)\le e^{-\frac{1}{8}\mu_2}\le e^{-\frac{1}{4}\xi N},
\end{equation}
where 
\begin{equation}
\mu_2=\!\sum_{i\in\mathcal{I}}\mathbb{P}(\bar{T}_p(I_i)\ge E_p(I_i) + d_i +2\rho_1 \max\{r,\log^2 L\})\!\ge \!\frac{c_\rho }{\sqrt{m}\log L}\!\sum_{i\in\mathcal{I}}|d_i|\!\ge \!\frac{c_\rho N}{\lceil 2\rho_3\rceil\sqrt{r}\log L}
\end{equation}
by~\eqref{eq:Talagrand lower bound} (checking that $\sqrt{\frac{m}{r}}\ge 5\rho_1\max\{r,\log^2L\}$ by~\eqref{eq:assumptions on r} when $\alpha_\rho$ is sufficiently small), and where we use that $\mu_2\ge 2\xi N$ when $\alpha_\rho$ is sufficiently small (here $c_\rho$ is the constant from~\eqref{eq:Talagrand lower bound}).
\end{proof}

\subsubsection{Monotonic events}
Recall that Harris' inequality (generalized to dependent variables by the FKG inequality) states that two increasing events in independent random variables are non-negatively correlated \cite{harris_1960}. In this section we explain the use that we make of this inequality in our context.

Let $p\in\mathcal{P}$. Observe that the event $\{\gamma=p\}$ is decreasing in the weights $(t_e)_{e\in p}$ and increasing in the weights $(t_e)_{e\notin p}$. Thus, by Harris' inequality, it is non-negatively correlated with every event sharing the same monotonicity properties. We employ the following variant of this observation.
\begin{lem}\label{lem:Harris use}
Let $p\in\mathcal{P}$. Let $E$ be an event which is decreasing in the weights $(t_e)_{e\notin p}$ (for every fixed value of $(t_e)_{e\in p}$). Then the following inequality of conditional probabilities holds almost surely,
\begin{equation}\label{eq:Harris use}
    \mathbb P\big( \{\gamma=p\}\cap E\,|\,(t_e)_{e\in p} \big) \le \mathbb P \big( \{\gamma = p\}\,|\,(t_e)_{e\in p} \big) \cdot\mathbb P \big(E\,|\,(t_e)_{e\in p} \big) .
\end{equation}
\end{lem}
\begin{proof}
The IID structure of the environment implies that the $(t_e)_{e\notin p}$ remain independent after conditioning on $(t_e)_{e\in p}$. As $\{\gamma = p\}$ is an increasing event in $(t_e)_{e\notin p}$ while $E$ is decreasing in these variables, we may apply
Harris' inequality in the conditional probability space to obtain~\eqref{eq:Harris use}.
\end{proof}

\subsubsection{Putting all the ingredients together}
In this section we explain how the results of the previous sections are combined to prove~\eqref{eq:attractive interval prob2}, from which the attractive geodesic proposition, Proposition~\ref{prop:attractive geodesics intro}, follows. We assume throughout that the constant $\alpha_\rho$ in~\eqref{eq:assumptions on r} is taken sufficiently small for the arguments.

Write $\mathcal{P}_{\rho,m}:=\{p\in\mathcal{P}\colon p\text{ has $(\rho,m)$-bounded slope}\}$. Also set $D_p:=(D_p(I_i))_{0\le i\le N-1}$ (recalling the definition of $D_p(J)$ from~\eqref{eq:Dp(J) def}).
First,
\begin{alignat}{1}
    &\mathbb{P}\left(\left\{\sum_{i=0}^{N-1} \mathds 1_{\Omega(I_i)}\le\xi N\right\}\cap\{\gamma\text{ has $(\rho,m)$-bounded slope}\}\right)\\
    &=\sum_{p\in\mathcal{P}_{\rho,m}}\mathbb{P}\left(\left\{\sum_{i=0}^{N-1} \mathds 1_{\Omega(I_i)}\le\xi N\right\}\cap\{\gamma=p\}\right)\\
    &=\sum_{p\in\mathcal{P}_{\rho,m}}\mathbb{P}\left(\mathcal{E}_{p,D_p}\cap\{\gamma=p\}\right)\\
    &\le Ce^{-c\log^2 L}+\sum_{p\in\mathcal{P}_{\rho,m}}\mathbb{P}\left(\mathcal{E}_{p,D_p}\cap\OT\cap\{\gamma=p\}\right)
\end{alignat}
with the second equality following by comparing the definition~\eqref{eq:definition of Omega(J)} of $\Omega(J)$, the definition~\eqref{eq:Dp(J) def} of $D_p(J)$ and the definition~\eqref{eq:Epd def} of $\mathcal{E}_{p,d}$, and with the inequality following from Lemma~\ref{lem:Talagrand control}. Second, for each $p\in\mathcal{P}_{\rho,m}$,
\begin{alignat}{1}
    &\mathbb{P}\left(\mathcal{E}_{p,D_p}\cap\OT\cap\{\gamma=p\}\right)\\
    &=\mathbb{P}\left(\mathcal{E}_{p,D_p}\cap\OT\cap\{\gamma=p\}\cap\left(\Omega_{1,\sqrt{\frac{m}{r}}}\cup\Omega_{2,\sqrt{\frac{m}{r}}}\right)\right)\\
    &=\mathbb{P}\left(\mathcal{E}_{p,D_p}\cap\OT\cap\{\gamma=p\}\cap\left(\Omega_{1,\sqrt{\frac{m}{r}}}(p)\cup\Omega_{2,\sqrt{\frac{m}{r}}}(p)\right)\right)\\
    &\le\mathbb{P}\left(\mathcal{E}_{p,D_p}\cap\{\gamma=p\}\cap\left(\Omega_{1,\sqrt{\frac{m}{r}}}(p)\cup\Omega_{2,\sqrt{\frac{m}{r}}}(p)\right)\right)=(*)
\end{alignat}
with the first equality following from the discussion after~\eqref{eq:deviation bounds for gamma} and the second equality following from the definition of $\Omega_{1,\tau}(p)$ and $\Omega_{2,\tau}(p)$ following~\eqref{eq:Omega 2 tau def}, making use of the intersection with the event $\{\gamma=p\}$. Third, we condition on the passage time of the edges on the path $p$ and observe that $D_p$ and hence also the $\Omega_{i,\tau}(p)$ events are measurable with respect to this conditioning. Thus,
\begin{alignat}{1}
    (*)&=\mathbb{E}\left(\mathds 1_{\Omega_{1,\sqrt{\frac{m}{r}}}(p)\cup\Omega_{2,\sqrt{\frac{m}{r}}}(p)}\,\mathbb{P}\left(\mathcal{E}_{p,D_p}\cap\{\gamma=p\}\right)\,\Big|\,(t_e)_{e\in p})\right)\\
    &\le\mathbb{E}\left(\mathds 1_{\Omega_{1,\sqrt{\frac{m}{r}}}(p)\cup\Omega_{2,\sqrt{\frac{m}{r}}}(p)}\,\mathbb{P}\left(\mathcal{E}_{p,D_p}\,\Big|\,(t_e)_{e\in p}\right)\mathbb{P}\left(\gamma=p\,\Big|\,(t_e)_{e\in p}\right)\right)\\
    &\le e^{-\frac{1}{4}\xi N}\mathbb{E}\left(\mathds 1_{\Omega_{1,\sqrt{\frac{m}{r}}}(p)\cup\Omega_{2,\sqrt{\frac{m}{r}}}(p)}\,\mathbb{P}\left(\gamma=p\,\Big|\,(t_e)_{e\in p}\right)\right)\\
    &\le e^{-\frac{1}{4}\xi N}\mathbb{P}\left(\gamma=p\right)
\end{alignat}
where the first inequality follows from Lemma~\ref{lem:Harris use} (as $\mathcal{E}_{p,D_p}$ is decreasing in $(t_e)_{e\notin p}$, for each fixed value of $(t_e)_{e\in p}$) and the second inequality follows from Proposition~\ref{prop:probability for long restricted passage times} (using the IID structure of the environment, as $\mathcal{E}_{p,d}$ is independent of $(t_e)_{e\in p}$ while $D_p$ is measurable with respect to these variables), noting that the $\Omega_{i,\tau}(p)$ events exactly ensure that condition~\eqref{eq:assumptions on d} holds.

Putting the previous displayed equations together, we finally conclude that
\begin{alignat}{1}
    &\mathbb{P}\left(\left\{\sum_{i=0}^{N-1} \mathds 1_{\Omega(I_i)}\le\xi N\right\}\cap\{\gamma\text{ has $(\rho,m)$-bounded slope}\}\right)\\
    &\le Ce^{-c\log^2 L}+e^{-\frac{1}{4}\xi N}\sum_{p\in\mathcal{P}_{\rho,m}}\mathbb{P}\left(\gamma=p\right)\le Ce^{-c\log^2 L}+e^{-\frac{1}{4}\xi N}.
\end{alignat}
This concludes the proof of~\eqref{eq:attractive interval prob2}, and hence of Proposition~\ref{prop:attractive geodesics intro}, once we note that $\xi N\ge c_\rho\log^2 L$ by~\eqref{eq:assumptions on r}, for some $c_\rho>0$ depending only on $G$ and $\rho$.

\subsection{Basic lemmas}\label{sec:auxiliary lemmas}
In this section we prove that the events $\OB$ (Lemma~\ref{lem:basic geometric control}) and $\OT$ (Lemma~\ref{lem:Talagrand control}) occur with high probability, and also prove Lemma \ref{lem:long restricted passage time} using Lemma~\ref{lem:MW corollary for T p J}.

Lemma \ref{lem:Talagrand control} is deduced from Talagrand's concentration inequality.
\begin{thm}[Talagrand's inequality \cite{Talagrand1995}]\label{thm:talagrand}There exist  $C,c>0$, depending only on $G$, such that
  for all $u,v\in \mathbb Z ^2 $ we have that 
  \begin{equation}
     \forall t\ge 0:\qquad \mathbb P \Big(  \left| T(u,v)-\mathbb E [T(u,v)] \right|   \ge t \sqrt{\|u-v\|_1} \Big) \le Ce^{-ct^2}.
  \end{equation}
\end{thm}
\begin{proof}[Proof of Lemma \ref{lem:Talagrand control}]We have
\begin{equation}
    \OT^c\subset \bigcup_{u,v\in[-L^2,L^2]^2\cap \ZZ^2}\left\{  |T(u,v)-\mathbb E [T(u,v)]|>\sqrt{\|u-v\|_1}\log L\right\}.
\end{equation}
The result follows from Theorem \ref{thm:talagrand} and a union bound.
\end{proof}
To prove Lemma \ref{lem:basic geometric control}, we need the following two claims.
\begin{claim}\label{claim:length} There exist $C,c,\rho_2>0$, depending only on $G$, such that for every $u,v\in \mathbb Z ^2$ and $n\ge \rho_2\|u-v\|_1$, we have 
  \begin{equation}
      \mathbb P \big( |\gamma(u,v) |\ge n  \big) \le Ce^{-cn}.
  \end{equation}
\end{claim}
\begin{claim}\label{claim:time} There exist $c,\rho_1>0$, depending only on $G$, such that for every $u,v\in \mathbb Z ^2$ and $n\ge \rho_1\|u-v\|_1$, we have 
  \begin{equation}
      \mathbb P \big(T(u,v)\ge n  \big) \le e^{-cn}.
  \end{equation}
\end{claim}

\begin{proof}[Proof of Lemma \ref{lem:basic geometric control}]
We have
\begin{equation}
\begin{alignedat}{1}
    \OB^c\subset \bigcup_{u,v\in[-L^2,L^2]^2\cap \ZZ^2} \Big(&\left\{  T(u,v)>\rho_1\max\{\|u-v\|_1,\log^2L\}\right\}\\
    &\cup \left\{ |\gamma(u,v)|>\rho_2\max\{\|u-v\|_1,\log^2 L\} \right\}\Big).
\end{alignedat}
\end{equation}
The result follows from Claim~\ref{claim:length}, Claim~\ref{claim:time} and a union bound.
\end{proof}
We proceed to prove the claims.
\begin{proof}[Proof of Claim \ref{claim:time}]
Let $p$ be a deterministic path between $u$ and $v$ such that $|p|=\|u-v\|_1$. For instance, one can choose the path that first goes straight in the vertical direction and then straight in the horizontal direction.
For each $\alpha>0$, Markov's inequality and the independence of the edge weights yield that every $n\ge \rho_1\|u-v\|_1$,
\begin{equation}
\begin{split}
    \mathbb P(T(u,v)\ge n)&\le  \mathbb P(T(p)\ge n)=\mathbb P\left(e^{\alpha T(p)}\ge e^{\alpha n}\right)\\
    &\le \left(\mathbb E e^{\alpha t_e}\right)^{\|u-v\|_1} e^{-\alpha n}\le \left(\left(\mathbb E e^{\alpha t_e}\right)^{1/\rho_1} e^{-\alpha}\right)^n.
\end{split}
\end{equation}
The claim follows by choosing $\alpha$ to be the constant from our assumption~\eqref{eq:assumption i} and choosing $\rho_1$ sufficiently large.
\end{proof}
To prove Claim \ref{claim:length} we need the following result due to Kesten, a corollary of Proposition 5.8 in \cite{Kesten:StFlour}.
\begin{thm}\label{thm:kesten}Suppose the edge-weight distribution $G$ satisfies $G(\{0\})< p_c(2)$ (with $p_c(2)=1/2$ the critical probability for bond-percolation on $\mathbb{Z}^2$). Then there exist $C,c,\rho_0>0$, depending only on $G$, such that 
\[\forall n\geq 0: \qquad \mathbb P \left(\begin{array}{c}\text{There exists a (simple) path $q$ starting }\\\text{from $(0,0)$ such that $|r|\geq n$ and $T(r)<\rho_0n$}\end{array}\right)\leq C\exp(-cn)\,.\]
\end{thm}
\begin{proof}[Proof of Claim \ref{claim:length}]
For each $u,v\in \mathbb Z^2$ and $n>0$,
\begin{equation}
    \mathbb P(|\gamma(u,v)|\geq n)\le \mathbb P \left(\!\!\!\begin{array}{c}\text{There exists a (simple) path $q$ starting }\\\text{from $u$ such that $|r|\geq n$ and $T(r)<\rho_0n$}\end{array}\!\!\right)+\mathbb P(T(u,v)\ge \rho_0 n)
\end{equation}
The claim thus follows from Theorem~\ref{thm:kesten} (using translation invariance) and Claim~\ref{claim:time} with $\rho_2 := \rho_1 /\rho_0$.
\end{proof}

Let us now prove Lemma \ref{lem:long restricted passage time}. The following preliminary claim shows that $\bar{T}_p(J)$ is unlikely to be much smaller than $T_p(J)$.

\begin{claim}[Connection cost]\label{claim:connection} There exist $C,c>0$, depending only on $G$, such that the following holds. For each path $p\in\mathcal P$ and each interval $J=[a,b]\subset [0,L]$ with integer endpoints,
\begin{equation}
    \mathbb P \big( \bar{T}_p(J) \ge T_p(J) -2\rho_1\max\{r,\log^2 L\} \big) \ge 1-Ce^ {-c\log ^ 2L}\,.
\end{equation}
\end{claim}
\begin{proof}
By the triangle inequality,
\begin{equation}
\begin{alignedat}{1}
    T_p(J)&\le \min_{\substack{u\in \tube_r(p)\cap S_a\\v\in \tube_r(p)\cap S_b}} T((a,f_p(a)),u)+T(u,v)+T(v,(b,f_p(b)))\\
    &\le \bar{T}_p(J)+\max_{u\in \tube_r(p)\cap S_a} T((a,f_p(a)),u)+\max_{v\in \tube_r(p)\cap S_b}T(v,(b,f_p(b))).
\end{alignedat}
\end{equation}
The claim thus follows from Claim~\ref{claim:time}, applied with $n=\rho_1\max\{r,\log^2 L\}$, and a union bound (using~\eqref{eq:assumptions on r} with $\alpha_\rho\le 1$).
\end{proof}

\begin{proof}[Proof of Lemma \ref{lem:long restricted passage time}, inequality \eqref{eq:Talagrand lower bound}]
Let $J=[a,b]\subset [0,L]$ be an interval with $a,b\in \mathbb Z$ and $m\le|J|\le 2m$. Let $t\in[5\rho_1 \max\{r,\log^2 L\}, \sqrt{2m(1+\rho)}\log L]$. Set $X:=T_p(J)-E_p(J)$. Taking into account Claim~\ref{claim:connection} and the fact that $t\ge 5\rho_1 \max\{r,\log^2 L\}$ we see that it suffices to prove that
\begin{equation}\label{eq:zeta inequality1}
    \zeta:=\mathbb{P}\left(X\ge -\frac{t}{5}\right)\ge \frac{c_\rho t}{\sqrt{m}\log L}
\end{equation}
for some $c_\rho>0$ (using that $Ce^{-c\log ^ 2L}\le \frac{t}{\sqrt{m}\log L}$ for large $L$ by~\eqref{eq:assumptions on r} with, say, $\alpha_\rho\le 1$).

Talagrand's concentration inequality, Theorem~\ref{thm:talagrand}, together with the fact that $p$ has $(\rho,m)$-bounded slope and the fact that $|J|\le 2m$ imply that
\begin{equation}
    \mathbb{E}\left(X\cdot\mathds{1}_{|X|\ge \sqrt{2m(1+\rho)}\log L}\right)\le Ce^{-c\log^2 L}.
\end{equation}
Therefore,
\begin{equation}
    \begin{split}
     0=\mathbb E [X] &=  \mathbb E \big[ X\cdot \mathds 1 \big\{ X < -t/5 \big\}\big] 
     +\mathbb E \big[ X\cdot \mathds 1 \big\{  -t/5 \le X< \sqrt{2m(1+\rho)}\log L  \big\} \big]\\&+ \mathbb E \big[ X\cdot \mathds 1 \big\{  X \ge \sqrt{2m(1+\rho)}\log L \big\}  \big]\\
     &\le (1-\zeta)(-t/5)+\zeta \sqrt{2m(1+\rho)}\log L +Ce^{-c\log ^2 L},
    \end{split}
\end{equation}
which implies~\eqref{eq:zeta inequality1}.
\end{proof}

\begin{proof}[Proof of Lemma \ref{lem:long restricted passage time}, inequality \eqref{eq:MW lower bound}, using Lemma~\ref{lem:MW corollary for T p J}] Let $J=[a,b]\subset [0,L]$ be an interval with integer endpoints satisfying $m\le|J|\le 2m$. Set $X:=T_p(J)-E_p(J)$ and define
\begin{equation}\label{eq:zeta inequality2}
    \zeta:=\mathbb{P}\left(X\ge 7\sqrt{\frac{m}{r}}\right) \quad \text{and} \quad \zeta':=\mathbb{P}\left(X\le -7\sqrt{\frac{m}{r}}\right).
\end{equation}
Note that by~\eqref{eq:assumptions on r} with small $\alpha_\rho$ we have that
\begin{equation}\label{eq:sqrt m r and rho1 max}
    \sqrt{\frac{m}{r}}\ge 4\rho_1 \max\{r,\log^2 L\}.
\end{equation}

Suppose first that $\zeta' \le 3/4$. Then, by Claim~\ref{claim:connection} and~\eqref{eq:sqrt m r and rho1 max},
\begin{equation}
    \mathbb P \bigg( \bar{T}_p(J) \ge E_p(J)- 8\sqrt{\frac{m}{r}} \bigg) \ge \frac{1}{5}.
\end{equation}
Using Lemma~\ref{lem:MW corollary for T p J} with $\alpha=15$ and $a=E_p(J)-8\sqrt{m/r}$ we obtain
\begin{equation}
    \sqrt{\mathbb P \bigg( \bar{T}_p(J) \ge E_p(J)+ 7\sqrt{\frac{m}{r}} \bigg)} \ge e^{-225C_\rho}\left(\frac{1}{5} - e^{-m}\right)\ge c_\rho.
\end{equation}
The bound in \eqref{eq:MW lower bound} follows from the last estimate using~\eqref{eq:sqrt m r and rho1 max}.

Suppose next that $\zeta ' > 3/4$. In this case, as in the previous proof,
\begin{equation}\label{eq:zeta and zeta prime 1}
    \begin{alignedat}{1}
     &0=\mathbb E [X] =  \mathbb E \big[ X\cdot \mathds 1 \big\{ X \le -7\sqrt{m/r} \big\}\big] +\mathbb E \big[ X\cdot \mathds 1 \big\{ |X| < 7\sqrt{m/r} \big\}  \big] +\\
     &+\mathbb E \big[ X\cdot \mathds 1 \big\{  7\sqrt{m/r} \le X< \sqrt{2m(1+\rho)}\log L \big\} \big]+ \mathbb E \big[ X\cdot \mathds 1 \big\{  X \ge \sqrt{2m(1+\rho)}\log L \big\}\big]\\
     &\le \zeta'\left(-7\sqrt{\frac{m}{r}}\right)+(1-\zeta-\zeta')7\sqrt{\frac{m}{r}} + \zeta \sqrt{2m(1+\rho)}\log L + Ce^{-c\log ^2 L}\\
     &\le (1-2\zeta')7\sqrt{\frac{m}{r}}+\zeta \sqrt{2m(1+\rho)}\log L + Ce^{-c\log ^2 L}\\
     &\le -3\sqrt{\frac{m}{r}} +\zeta \sqrt{2m(1+\rho)}\log L.
    \end{alignedat}
\end{equation}
This shows that $\zeta \ge c_{\rho }/(\sqrt{r}\log L)$, which implies~\eqref{eq:MW lower bound} using Claim~\ref{claim:connection} and~\eqref{eq:sqrt m r and rho1 max}.
\end{proof}

\subsection{Perturbing the weights (a Mermin--Wagner style argument)}\label{sec:mermin--wagner}
In this section we prove Lemma~\ref{lem:MW corollary for T p J}. Our basic tool is a ``Mermin--Wagner style argument''; by this terminology, we mean the idea of perturbing a distribution (in our case, the edge passage time distribution) in a way which, on the one hand, significantly alters the observable of interest (in our case, the restricted passage time) and, on the other hand, can be usefully compared with the original distribution. This basic (and somewhat vague) approach has been key in many proofs of the Mermin--Wagner theorem in statistical physics, including~\cite{dobrushin1975absence,mcbryan1977decay,dobrushin1980nonexistence, pfister1981symmetry,richthammer2007translation,milos2015delocalization},~\cite[Theorem 9.2]{friedli2017statistical} and~\cite[Section~2.6]{peled2019lectures}), whence the name, but has also been used in other contexts, e.g. in~\cite{schenker2009eigenvector,chatterjee2019general,kozma2021power,elboim2022long}. Our treatment here draws inspiration from~\cite{pfister1981symmetry,richthammer2007translation,milos2015delocalization,kozma2021power} and has the benefit of providing Gaussian lower bounds on the tail probabilities.

\subsubsection{The basic probabilistic estimate}
The following statement is the basic ``Mermin--Wagner style estimate'' that we will use. Given a subset $S$ we write $S^n$ for its $n$-fold Cartesian product and given a probability measure $\nu$ we write $\nu^n$ for its $n$-fold product measure. In our application, the measure $\nu$ will be the distribution $G$ of the edge passage time.

\begin{lem}\label{lem:MW}
  Let $\nu$ be an absolutely-continuous probability measure on $\mathbb{R}$. There exist 
  \begin{itemize}
      \item a Borel $S_\nu\subset\mathbb{R}$ with $\nu(S_\nu)=1$,
      \item Borel subsets $(B_\delta)_{\delta>0}$ of $S_\nu$ with $\lim_{\delta\downarrow 0}\nu(B_\delta)=1$,
      \item for each $\sigma\in [0,1]$, two increasing bijections $g_{\nu,\sigma}^+:S_\nu\to S_\nu$ and $g_{\nu,\sigma}^-:S_\nu\to S_\nu$,
  \end{itemize}
   such that the following holds:
  \begin{enumerate}
      \item For $\sigma\in[0,1]$ and $\delta>0$,
      \begin{equation}\label{eq:inclusion implies shift}
          g_{\nu,\sigma}^+(w)\ge w+\delta\sigma\quad\text{and}\quad g_{\nu,\sigma}^-(w)\le w-\delta\sigma\quad\text{for $w\in B_\delta$}.
      \end{equation}
      \item Given an integer $n\ge 1$ and vector $\tau=(\tau_1,\ldots,\tau_n)\in[0,1]^n$ define two bijections $T_{\nu,\tau}^+:S_\nu^n\to S_\nu^n$ and $T_{\nu,\tau}^-:S_\nu^n\to S_\nu^n$ by \begin{equation}\label{eq:T plus minus def}
          T_{\nu,\tau}^\pm(w)_i=g_{\nu,\tau_i}^\pm(w_i)\quad\text{for $1\le i\le n$}.
      \end{equation}
      Then, for each Borel $A\subset\mathbb{R}^n$,
      \begin{equation}\label{eq:MW probability estimate}
          \sqrt{\nu^n(T_{\nu,\tau}^+(A))\,\nu^n(T_{\nu,\tau}^-(A))}\ge e^{-\frac{1}{2}\|\tau\|_2^2}\nu^n(A),
      \end{equation}
	where we use the notation $T(A):=\{T(a)\colon a\in A\cap S_\nu^n\}$.
  \end{enumerate}
\end{lem}
We remark that measurability of $T_{\nu,\tau}^\pm(A)$ in~\eqref{eq:MW probability estimate} is ensured as $(T_{\nu,\tau}^\pm)^{-1}$ are Borel measurable (since $g_{\nu,\sigma}^\pm$ are increasing bijections, they and their inverses are Borel measurable).

The rest of this section is devoted to the proof of Lemma~\ref{lem:MW}. The first step is to establish the lemma when $\nu$ is the standard Gaussian distribution. This case is simpler and already of interest on its own (cf.~\cite[Section 1.1.1]{milos2015delocalization} for a simple application of this technique to the delocalization of height functions). 

\begin{claim}\label{claim:MW Gaussian}
  Let $\tau=(\tau_1,\ldots, \tau_n)\in[0,\infty)^n$ and let $X = (X_1,\dots ,X_n)$ be a vector of independent standard Gaussian random variables. Then, for every measurable $A\subset\mathbb{R}^n$,
  \begin{equation}
     \sqrt{\mathbb{P}(X+\tau\in A)\mathbb{P}(X-\tau\in A)}\ge e^{-\frac{1}{2}\|\tau\|_2^2} \mathbb{P}(X\in A).
  \end{equation}
\end{claim}
\begin{proof}
 Let $f:\mathbb{R}^n\to[0,\infty)$ be the density of $X$, i.e.,
 \begin{equation}
     f(x):=(2\pi)^{-\frac{1}{2}n}e^{-\frac{1}{2}\|x\|_2^2}.
 \end{equation}
 Observe that the density of $X\pm \tau$ is $f(x\mp \tau)$ and that
 \begin{equation}
     \sqrt{f(x-\tau)f(x+\tau)} = e^{-\frac{1}{2}\|\tau\|_2^2}f(x).
 \end{equation}
 Thus, on the one hand,
 \begin{equation}
     I:=\int_A \sqrt{f(x-\tau)f(x+\tau)}dx = e^{-\frac{1}{2}\|\tau\|_2^2}\int_A f(x)dx = e^{-\frac{1}{2}\|\tau\|_2^2}\cdot\mathbb{P}(X\in A),
 \end{equation}
 while, on the other hand, by the Cauchy-Schwarz inequality,
 \begin{equation}
     I\le \sqrt{\int_A f(x-\tau)dx\int_A f(x+\tau)dx} = \sqrt{\mathbb{P}(X+\tau\in A)\mathbb{P}(X-\tau\in A)}.
 \end{equation}
 The claim follows by combining the previous two displayed equations.
\end{proof}

The second step is to define the bijections $g_{\nu,\sigma}^\pm$. For the rest of the section fix an absolutely-continuous probability measure $\nu$ on $\mathbb{R}$. 

Define the Borel set
\begin{equation}\label{eq:S nu def}
    S_\nu:=\{w\in\mathbb{R}\colon \nu((-\infty,v])<\nu((-\infty,w])\text{ for all $v<w$}\}.
\end{equation}
It is simple to check that $\nu(S_\nu)=1$. Let $\nu_{\text{Gauss}}$ be the standard Gaussian distribution. Let $h:\mathbb{R}\to S_\nu$ be defined by\noeqref{eq:S nu def}
\begin{equation}\label{eq:h bijection def}
    h(x) = w\text{ for the unique $w\in S_\nu$ satisfying $\nu_{\text{Gauss}}((-\infty,x])=\nu((-\infty,w])$}.
\end{equation}
Such a $w$ exists as $\nu$ has no atoms while uniqueness follows from the definition of $S_\nu$. It follows also that $h$ is an increasing bijection satisfying $h(\nu_{\text{Gauss}})=\nu$ and $h^{-1}(\nu)=\nu_{\text{Gauss}}$ (by this we mean that $h(N)\sim \nu $ where $N\sim \nu_{\text{Gauss}}$ and $h^{-1}(X)\sim \nu_{\text{Gauss}}$ where $X\sim \nu$). For $\sigma\in[0,1]$ define $g_{\nu,\sigma}^+:S_\nu\to S_\nu$ and $g_{\nu,\sigma}^-:S_\nu\to S_\nu$ by
\begin{equation}\label{eq:defg}
    g_{\nu,\sigma}^\pm(w) = h(h^{-1}(w) \pm \sigma).
\end{equation}
We note also that $g_{\nu,\sigma}^+(w)\ge w$ and $g_{\nu,\sigma}^-(w)\le w$ for all $w\in S_\nu$.

As the third step, we establish~\eqref{eq:MW probability estimate}. Let $\tau=(\tau_1,\ldots,\tau_n)\in[0,1]^n$. Define $T_{\nu,\tau}^+:S_\nu^n\to S_\nu^n$ and $T_{\nu,\tau}^-:S_\nu^n\to S_\nu^n$ by~\eqref{eq:T plus minus def}. Let $A\subset\mathbb{R}^n$ be Borel. The fact that $h(\nu_{\text{Gauss}})=\nu$ implies that
\begin{equation}
\begin{alignedat}{1}
    \nu^n(T_{\nu,\tau}^\pm(A))&= \nu^n(\{(h(h^{-1}(w_1)\pm\tau_1),\ldots,h(h^{-1}(w_n)\pm\tau_n))\colon w\in A\cap S_\nu^n\})\\
    &= \nu_{\text{Gauss}}^n(\{(h^{-1}(w_1)\pm\tau_1,\ldots,h^{-1}(w_n)\pm\tau_n)\colon w\in A\cap S_\nu^n\}).
\end{alignedat}
\end{equation}
Therefore, Claim~\ref{claim:MW Gaussian} and the fact that $h^{-1}(\nu)=\nu_{\text{Gauss}}$ imply
\begin{equation}
\begin{split}
    \sqrt{\nu^n(T_{\nu,\tau}^+(A))\,\nu^n(T_{\nu,\tau}^-(A))}&\ge e^{-\frac{1}{2}\|\tau\|_2^2}\nu_{\text{Gauss}}^n(\{(h^{-1}(w_1),\ldots,h^{-1}(w_n))\colon w\in A\cap S_\nu^n\})\\
    &= e^{-\frac{1}{2}\|\tau\|_2^2} \nu^n(A\cap S_\nu^n) = e^{-\frac{1}{2}\|\tau\|_2^2} \nu^n(A),
\end{split}
\end{equation}
proving~\eqref{eq:MW probability estimate}.

Lastly, we proceed to define the sets $(B_\delta)$ and establish~\eqref{eq:inclusion implies shift}. We use the following real analysis statement, which follows from the absolute continuity of $\nu$.
\begin{claim}\label{claim:real analysis}
For $\delta >0$, define the set 
\begin{equation}
    A_\delta  :=\left\{ x\in \mathbb R\colon \frac{h(y)-h(x)}{y-x}\ge \delta \text{ for all $y\in[x-1,x+1]\setminus\{x\}$}\right\}. 
\end{equation}
Then
\begin{equation}\label{eq:6}
    \lim_{\delta\downarrow 0}\nu_{{\rm Gauss}}(A_\delta)=1.
\end{equation}
\end{claim}
\begin{proof}
Let $\bar{\nu}:\mathbb{R}\to[0,\infty]$ be the Hardy-Littlewood maximal function of $\nu$, given by 
\begin{equation}
    \bar{\nu}(w):=\sup _{r>0} \frac{1}{2r} \nu((w-r,w+r)).
\end{equation}
For $M>0$ define the set 
\begin{equation}
    D_M:=\big\{ x\in [-M,M] : \bar{\nu}(h(x)) \le M \big\}.
\end{equation}
First we show that
\begin{equation}\label{eq:7}
    \lim_{M\to\infty}\nu_{\text{Gauss}}(D_M)=1.
\end{equation}
Indeed, since $\nu$ is absolutely continuous, the Hardy-Littlewood maximal inequality \cite[Chapter 3]{stein2009real}  implies that the set $\{w\in\mathbb{R}\colon \bar{\nu}(w)=\infty\}$ has zero Lebesgue measure. Thus,
\begin{equation}
    \lim_{M\to\infty}\nu_{\text{Gauss}}(\{x\colon \bar{\nu}(h(x))>M\})\!=\!\lim_{M\to\infty}\nu(\{w\colon \bar{\nu}(w)>M\})\!=\!\nu(\{w\colon \bar{\nu}(w)=\infty\})\!=\!0,
\end{equation}
from which~\eqref{eq:7} follows.

Second, as $h$ is increasing, for any $x\in D_M$ we have for all $x\le y\le x+1$,
\begin{equation}
\begin{split}
      (2\pi )^{-1/2}e^{-\frac{1}{2}(M+1)^2}(y-x)&\le \nu_{\text{Gauss}}((x,y))=\nu((h(x),h(y)))\\
      &\le 2\bar{\nu}(h(x))(h(y)-h(x)) \le 2M (h(y)-h(x)).
\end{split}
\end{equation}
An analogous statement holds with $x-1\le y\le x$.
Thus, for each $M>0$ there exists $\delta >0$ for which $D_M\subseteq A_\delta $. This finishes the proof of \eqref{eq:6} using \eqref{eq:7}.
\end{proof}
Define the sets $(B_\delta)_{\delta>0}$ by $B_\delta:=\{w\in S_\nu\colon h^{-1}(w)\in A_\delta\}$. Since $h(\nu_{\text{Gauss}})=\nu$, the last claim implies that
\begin{equation}
    \lim_{\delta\downarrow0}\nu(B_\delta) = \lim_{\delta\downarrow 0}\nu_{\text{Gauss}}(A_\delta)=1.
\end{equation}
Next, note that the definition of $A_\delta$ implies that for each $\sigma\in[0,1]$ and $w\in B_\delta$, we have $g_{\nu,\sigma}^+(w) = h(h^{-1}(w)+\sigma)\ge w+\delta \sigma$. An analogous statement holds for $g_{\nu,\sigma}^-$, implying~\eqref{eq:inclusion implies shift}.

\begin{remark}\label{rem:AMW} (Asymmetric Mermin--Wagner)
Lemma~\ref{lem:MW} admits a generalization in which the $T^+$ and $T^-$ bijections play asymmetrical roles:
For all $p,q>1$ with $1/p+1/q=1$, inequality~\eqref{eq:MW probability estimate} can be replaced with 
\begin{equation}\label{eq:amw}
 \nu^n(T_{\nu,q\tau/p}^+(A))^{1/q}\,\nu^n(T_{\nu,\tau}^-(A)) ^{1/p}\ge e^{-q\|\tau\|_2^2/2p}\nu^n(A)
\end{equation}
(the case $p=q=2$ is inequality~\eqref{eq:MW probability estimate} itself). To prove \eqref{eq:amw}, one simply
 use H\"older's inequality instead of Cauchy-Schwarz in Claim~\ref{claim:MW Gaussian} to obtain
\begin{equation}
    e^{-\frac{q\|\tau\|_2^2}{2p}}   \mathbb P (X\in A)=\!\!\int _A f(x-\tau )^{\frac{1}{p}} f(x+q\tau /p )^{\frac{1}{q}} \le \mathbb P (X+\tau \in A) ^{\frac{1}{p}} \,\mathbb P (X-q\tau /p\in A)^{\frac{1}{q}}.
\end{equation}
The rest of the proof is identical to that of Lemma~\ref{lem:MW}.

In addition, a different way of writing~\eqref{eq:amw} is sometimes convenient. 
By~\eqref{eq:defg}, we have $(g^+_{\nu ,\sigma })^{-1}=g^-_{\nu ,\sigma }$ so that $(T^+_{\nu,\tau})^{-1}=T^-_{\nu,\tau}$ (see~\eqref{eq:T plus minus def}). Hence we may rewrite~\eqref{eq:amw} as
\begin{equation}
 \nu^n((T_{\nu,q\tau/p}^-)^{-1}(A))^{1/q}\,\nu^n((T_{\nu,\tau}^+)^{-1}(A)) ^{1/p}\ge e^{-q\|\tau\|_2^2/2p}\nu^n(A).
\end{equation}
In particular, $\nu ^n((T^+_{\nu ,\tau })^{-1}(A))\ge \exp (-q\|\tau \|_2^2/2) \nu ^n(A)^p$.
\end{remark}

\begin{remark}
    In Lemma~\ref{lem:MW} and Remark~\ref{rem:AMW}, if we further assume that $\nu$ is the image of the Gaussian distribution under an increasing Lipschitz function (equivalently, the function $h$ in~\eqref{eq:h bijection def} is Lipschitz) then we obtain that $g_{\nu ,\sigma }^+(w)\le w+C\sigma $ and $g_{\nu ,\sigma }^-(w)\ge w-C\sigma $ where $C$ is the Lipschitz constant of $h$. This statement is immediate from the definition of $g_{\nu ,\sigma }^{\pm}$ given in~\eqref{eq:defg}.
\end{remark}

\subsubsection{Application to the restricted passage time $\bar{T}_p(J)$}

Throughout this section we fix a path $p\in\mathcal{P}$ having $(\rho,m)$-bounded slope and an interval $J\subset[0,L]$ with integer endpoints satisfying $m\le |J|\le 2m$.

Our goal is to use Lemma~\ref{lem:MW} to prove (a generalization of) Lemma~\ref{lem:MW corollary for T p J}. As previously mentioned, on an intuitive level, the result may be thought of as saying that the distribution of $\bar{T}_p(J)$ ``contains a Gaussian component with variance of order $\frac{m}{r}$''. The following is our precise statement.

\begin{lem}\label{lem:MW for T p J}
There exist $C_\rho,c_\rho>0$, depending only on $G$ and $\rho$, such that for each $0\le\alpha\le c_\rho\sqrt{m r}$ and real $a,b\in[-\infty,\infty]$ with $a\le b$,
\begin{equation}
    \sqrt{\mathbb{P}\left(\bar{T}_p(J)\ge a +\alpha\sqrt{\frac{m}{r}}\right)\mathbb{P}\left(\bar{T}_p(J)\le b-\alpha\sqrt{\frac{m}{r}}\right)}\ge e^{-C_\rho\alpha^2}(\mathbb{P}\left(a\le \bar{T}_p(J)\le b\right) - e^{-m}).
\end{equation}
\end{lem}

We note that Lemma~\ref{lem:MW corollary for T p J} is the special case $b=\infty$ of this result.
The rest of the section is devoted to the proof of Lemma~\ref{lem:MW for T p J}.

We aim to use Lemma~\ref{lem:MW} to change the weight environment $(t_e)_{e\in E(\mathbb{Z}^2)}$. Since we are only interested in the effect of this change on the restricted passage time $\bar{T}_p(J)$, we restrict attention to a suitable \emph{finite} set of edges $\Sigma$ in $\mathbb{Z}^2$ which contains the edges of all paths $q\in Q_p(J)$ as well as all edges in the set $\mathcal E_p(J)$ below.

Define the set of edges 
\begin{equation}
    \mathcal E_p(J) :=\big\{ e\in E(\mathbb Z ^2) : e\text{ has both endpoints in }\tube_r(p)\cap S_J  \big\}.
\end{equation}
Let $\delta_0>0$ be a small constant, chosen as a function only of $G$ and $\rho$ following Claim~\ref{claim:union bound} below. We apply Lemma~\ref{lem:MW} with $\nu=G$
and with $(\tau_e)_{e\in\Sigma}$ given by
\begin{equation}\label{eq:tau def}
    \tau_e := \begin{cases}\frac{4\alpha}{\delta_0\sqrt{mr}}&e\in \mathcal E_p(J),\\0&e\notin \mathcal E_p(J).
    \end{cases}
\end{equation}
We take $0\le \alpha\le \frac{1}{4}\delta_0\sqrt{mr}$ so that $0\le \tau_e\le 1$ for all $e$. Note that
\begin{equation}\label{eq:tau norm}
    \|\tau\|_2^2 = \frac{16\alpha^2}{\delta^2_0 mr}\big| \mathcal E_p(J) \big| \le \frac{32\alpha^2}{\delta^2_0 mr} \big|\tube_r(p)\cap S_J \big| \le \frac{32\alpha^2}{\delta^2_0 mr}(2r+1)(|J|+1) \le \frac{300\alpha^2}{\delta^2_0}.
\end{equation}

The lemma provides us with two bijections, $T_{G,\tau}^+:S_G^\Sigma\to S_G^\Sigma$ and $T_{G,\tau}^-:S_G^\Sigma\to S_G^\Sigma$, where $S_G\subset\mathbb{R}$ satisfies $G(S_G)=1$. Define new weight environments by
\begin{equation}
    (t_e^+)_{e\in\Sigma} := T_{G,\tau}^+((t_e)_{e\in\Sigma})\quad\text{and}\quad (t_e^-)_{e\in\Sigma} := T_{G,\tau}^-((t_e)_{e\in\Sigma}).
\end{equation}
Recall the events $(B_\delta)$ from Lemma~\ref{lem:MW}. Note that, almost surely,
\begin{equation}\label{eq:monotone weight shift}
    t_e^+\ge t_e\quad\text{and}\quad t_e^-\le t_e\quad\text{for $e\in\Sigma$}
\end{equation}
by~\eqref{eq:inclusion implies shift} and the fact that $\lim_{\delta\downarrow0}G(B_\delta)=1$. Denote by $\bar{T}_p(J)^+$ and $\bar{T}_p(J)^-$ the random variable $\bar{T}_p(J)$ calculated in the environments $(t_e^+)$ and $(t_e^-)$, respectively.

Define the random set of edges
\begin{equation}
    \mathcal E '_{p,\delta_0}(J):=\big\{ e\in \mathcal E _p(J) : t_e\in B_{\delta_0} \big\}
\end{equation}
and the event
\begin{equation}
    \Omega_{\delta_0}:= \{\forall q\in Q_p(J),\, |q\cap \mathcal E'_{p,\delta_0}(J) | \ge m/4\},
\end{equation}
where $|q\cap \mathcal{E}|$ denotes the number of edges in common to the path $q$ and edge set $\mathcal{E}$. Crucially, the passage time of each $q\in Q_p(J)$ can only increase when calculated in $(t_e^+)$ compared to $(t_e)$ (by~\eqref{eq:monotone weight shift}), and on $\Omega_{\delta_0}$ it must increase by at least $\frac{m}{4}\cdot\delta_0\cdot\frac{4\alpha}{\delta_0\sqrt{mr}}=\alpha\sqrt{\frac{m}{r}}$ by~\eqref{eq:inclusion implies shift},~\eqref{eq:T plus minus def} and~\eqref{eq:tau def}. A similar fact holds with the environment $(t_e^-)$. Therefore
\begin{equation}\label{eq:increase on Omega delta}
    \text{when $(t_e)\in\Omega_{\delta_0}$:\quad $\bar{T}_p(J)^+\ge T_p(J)+\alpha\sqrt{\frac{m}{r}}$\quad and\quad $\bar{T}_p(J)^-\le T_p(J)-\alpha\sqrt{\frac{m}{r}}$}.
\end{equation}

We next show that $\Omega_{\delta_0}$ is very likely when $\delta_0$ is sufficiently small.
\begin{claim}\label{claim:union bound}
There exists $\delta_1>0$, depending only on $G$ and $\rho$, such that if $\delta_0\le\delta_1$ then
\begin{equation}
     \mathbb P(\Omega_{\delta_0}^c) \le e^{-m}.
\end{equation}
\end{claim}

\begin{proof}
Fix a path $q\in Q_p(J)$. By the definition of $Q_p(J)$ and $\mathcal E_p(J)$ we have that 
\begin{equation*}
|q \cap \mathcal E _p(J) |\ge \frac{|J|}{2}\ge \frac{m}{2}     
\end{equation*}
Moreover, any edge $e\in q\cap \mathcal E _p(J)$ is in $\mathcal E'_{p,\delta_0}(J)$ with probability $G(B_{\delta_0})$, independently of the other edges. Thus,
\begin{equation}
    |q\cap \mathcal E'_{p,\delta_0}(J) | \succeq \text{Bin}(\lceil m/2\rceil,G(B_{\delta_0})),
\end{equation}
where $\succeq$ denotes stochastic domination.
It follows that 
\begin{equation}
\begin{split}
    \mathbb P \big( |q\cap \mathcal E'_{p,\delta_0}(J) | < m/4 \big) &\le \mathbb P \big( \text{Bin}(\lceil m/2\rceil,G(B_{\delta_0}) ) < m/4 \big)\\
    &= \sum _{k< m/4} \binom{\lceil m/2\rceil}{k} G(B_{\delta_0})^k(1-G(B_{\delta_0}))^{\lceil m/2\rceil-k} \\
      &\leq(1-G(B_{\delta_0}))^{m/4} \sum _{k< m/4} \binom{\lceil m/2\rceil}{k} \le (1-G(B_{\delta_0}))^{m/4}2^{\lceil m/2\rceil}.
\end{split}
\end{equation}
Finally, recall that each $q\in Q_p(J)$ satisfies $|q|\le \rho_2\max\{\|u-v\|_1,\log^2 L\}$ where $u,v$ are the endpoints of $q$, contained in $\tube_r(p)\cap S_J$. In particular, $|q|\le \rho_2 (2m(1+\rho)+2r)\le C_\rho m$ for some $C_\rho>0$, as $p$ has $(\rho,m)$-bounded slope and by~\eqref{eq:assumptions on r} (with $\alpha_\rho\le 1$). Thus,
$|Q_p(J)|\le Cr 4^{C_\rho m}$ for some $C>0$. A union bound now gives
\begin{equation}
    \mathbb P(\Omega_{\delta_0}^c)\le Cr 4^{C_\rho m}(1-G(B_{\delta_0}))^{m/4}2^{\lceil m/2\rceil}
\end{equation}
from which the claim follows by recalling that $\lim_{\delta\downarrow 0}G(B_\delta)=1$.
\end{proof}
Henceforth we fix $\delta_0$ to the value $\delta_1$ of the last claim. We proceed to deduce Lemma~\ref{lem:MW for T p J}. Let $a,b\in[-\infty,\infty]$ with $a\le b$. Set $A:=\Omega_{\delta_0}\cap \{a\le \bar{T}_p(J)\le b\}$. On the one hand, by~\eqref{eq:MW probability estimate},
\begin{equation}
\begin{alignedat}{1}
    \sqrt{G^\Sigma(T_{G,\tau}^+(A)) \,G^\Sigma(T_{G,\tau}^-(A))}&\ge e^{-\frac{1}{2}\|\tau\|_2^2}G^\Sigma(A)=e^{-\frac{1}{2}\|\tau\|_2^2}\mathbb{P}((t_e)_{e\in\Sigma}\in A)\\
    &\ge e^{-\frac{1}{2}\|\tau\|_2^2}(\mathbb{P}(a\le \bar{T}_p(J)\le b) - \mathbb{P}(\Omega_{\delta_0}^c)).
\end{alignedat}
\end{equation}
On the other hand, by~\eqref{eq:increase on Omega delta}, if $(t_e)\in\Omega_{\delta_0}$ and $a\le \bar{T}_p(J)\le b$ then $\bar{T}_p(J)^+\ge a + \alpha\sqrt{\frac{m}{r}}$ and $\bar{T}_p(J)^-\le b - \alpha\sqrt{\frac{m}{r}}$. Therefore
\begin{equation}
\begin{alignedat}{1}
    \sqrt{G^\Sigma(T_{G,\tau}^+(A)) \,G^\Sigma(T_{G,\tau}^-(A))} &= \sqrt{\mathbb{P}((t_e)_{e\in\Sigma}\in T_{G,\tau}^+(A))\,\mathbb{P}((t_e)_{e\in\Sigma}\in T_{G,\tau}^-(A))}\\
    \le &\sqrt{\mathbb{P}\left(\bar{T}_p(J)\ge a + \alpha\sqrt{\frac{m}{r}}\right)\,\mathbb{P}\left(\bar{T}_p(J)\le b - \alpha\sqrt{\frac{m}{r}}\right)}.
\end{alignedat}
\end{equation}
Lemma~\ref{lem:MW for T p J} follows by combining the last two displayed equations with~\eqref{eq:tau norm} and Claim~\ref{claim:union bound}.

\section{Proof of the main theorems}
\label{sec:proofmainthm}
In this section we deduce our main results, Theorem~\ref{mainthm2}, Theorem~\ref{prop:probedgebulk}, Theorem~\ref{thm:density} and Theorem~\ref{thm:density2}, from the attractive geodesics proposition, Proposition~\ref{prop:attractive geodesics intro}. We will also need the following proposition which shows that typically a geodesic does not ``go in the wrong direction for a long time''.
 This proposition will allow us to ``trap'' geodesics.
 
 For $x=(x_1,x_2)\in \mathbb R ^2$ we write $\lfloor x \rfloor :=(\lfloor x_1 \rfloor ,\lfloor x_2 \rfloor)\in \mathbb Z ^2$ where  for $t\in\mathbb R$, $\lfloor t\rfloor$ denotes the largest integer smaller than $t$. Identifying $\mathbb R ^2$ with $\mathbb C$ we have 
\begin{equation}
    \lfloor Re^{i\theta } \rfloor =\big( \lfloor  R \cos \theta \rfloor  , \lfloor R \sin \theta \rfloor \big).
\end{equation}

\begin{prop}\label{prop:limit} Let $k\ge 1$. 
 Suppose that the limit shape is not a polygon with $8k$ sides or less. There exist $\theta_0,\theta_{k}\in(\pi/4,\pi/2)$, $\theta_0<\theta_k$ such that for all $\epsilon >0$, for any $\theta\in[-\theta_0,\theta_0]$, denoting by $\gamma $ the geodesic from $(0,0)$ to $\lfloor ne^ {i\theta}\rfloor$, 
 \begin{equation}
     \mathbb P \Big( \lfloor Re^{i\varphi } \rfloor\in \gamma \text{ for some } R\ge n^{2^{-k}+\epsilon} \text{ and }  \varphi\in(-\pi,\pi], \,|\varphi|\ge \theta_{k}\Big) \le C\exp \big( -n^{c_\epsilon } \big) 
 \end{equation}
 where the constant $c_\epsilon$ may depend on $\epsilon$ and $\theta_0$, $\theta_{k}$.
 
 Moreover if $\theta =0$ then $\theta _{k}$ can be chosen in the interval $(0,\pi /4) $ and $\theta_0=0$.
\end{prop}
This proposition is proved in Section \ref{sec:the limit shape}.

Throughout this section and the next ones we will denote by $C,c$ generic positive constants which may depend only on the edge weight distribution $G$, whose value may change from one appearance to the next, with the value of $C$ increasing and the value of $c$ decreasing. Similarly, labeled constants such as $C_0$ or $c_\epsilon$ (which may also depend on $G$ and additionally on their subscript variables) do not change their value throughout the section where they are defined.

\subsection{Coalescence of geodesics in $\ZZ^2$}
In this section, we prove Theorem \ref{mainthm2} using Propositions \ref{prop:attractive geodesics intro}, \ref{prop:nobigjumps} and \ref{prop:limit}.
\begin{proof}[Proof of Theorem \ref{mainthm2}]
We assume that $\sides(\mathcal B_G)>32$.
Let $\ep\in(0,1/17]$ and $y\in\mathbb \ZZ^2$ with $\|y\|\ge 2$. Without loss of generality, thanks to the symmetry of the lattice, we can only study the case where $y=(y_1,y_2)=\|y\|_2e^{i\theta_0}$ with $0\le \theta_0\le \pi/4$. 
Set \[\ell:=\lfloor \|y\|^{1/8 -\epsilon}\rfloor.\]
Let $\varphi_0,\varphi_4\in(\pi/4,\pi/2)$ depending only on $G$ be as in Proposition~\ref{prop:limit} (corresponding to $\theta_0$, $\theta_k$ in the statement of the proposition) applied to $k=4$ (using Proposition~\ref{prop:limit} with $k=4$ ensures that geodesics cannot travel in the wrong direction to distance $\|y\|^{1/16+\epsilon }<<\ell $ and therefore cannot escape a trap of size $O(\ell )$ around them. See Figure~\ref{fig:coalescence}). Let $\kappa$ be the smallest positive integer (depending only on $\varphi_4$) such that
\[\frac{2}{\kappa -1}\le \tan(\pi/2-\varphi_4)\,.\]
Set 
$w^-:=(-\ell,-\kappa\ell)$, $z^-:=y+(\ell,-\kappa\ell)$, $w^+:=(-\ell,\kappa\ell)$ and $z^+:=y+(\ell,\kappa\ell)$. Let $\gamma^-$  (respectively $\gamma^+$) be the geodesic between $w^-$ and $z^-$ (respectively $w^+$ and $z^+$). Our goal is to prove that the geodesics $\gamma^-$ and $\gamma^+$ coalesce with high probability and that all geodesics starting in $\Lambda_\ell$ and ending in $y+\Lambda_\ell$ are trapped between $\gamma^-$ and $\gamma^+$ and forced to coalesce (see Figure~\ref{fig:coalescence}).
Denote by $\mathcal T$ the event where the following holds
\begin{itemize}
    \item The geodesic $\gamma^+$ stays above $\Lambda_\ell$ and $y+\Lambda_\ell$, that is, $\gamma^+$ does not intersect the set $\{-\ell,\dots,\ell\}\times (-\infty,\ell]\cup (y+\{-\ell,\dots,\ell\}\times (-\infty,\ell])$. 
    \item The geodesic $\gamma^-$ stays below $\Lambda_\ell$ and $y+\Lambda_\ell$, that is, $\gamma^-$ does not intersect the set $\{-\ell,\dots,\ell\}\times [-\ell,+\infty)\cup (y+\{-\ell,\dots,\ell\}\times [-\ell,+\infty))$. 
    \item Any geodesic that starts at a point $u\in\Lambda_\ell$ and ends at a point $v\in(y+\Lambda_\ell)$ does not circle $\gamma^+$ or $\gamma^-$, that is, it does not intersect the following set
\begin{equation}
    V:=\left\{ (-\ell,t): |t|\ge \kappa \ell\right\}\cup \left\{ y+(\ell,t): |t|\ge \kappa \ell\right\}.
\end{equation}
\end{itemize}
On the event $\mathcal T$, the geodesics starting in $\Lambda_\ell$ and ending in $y+\Lambda_\ell$ are "trapped" between $\gamma^+$ and $\gamma^-$.
Let us prove that the event $\mathcal T$ occurs with high probability.
Let us first prove that with high probability, the geodesics $\gamma^-$ and $\gamma^+$ stay respectively below and above $\Lambda_{\ell}$ and $y+\Lambda_\ell$.

Set $\ep_0:=1/3(1/16-1/17)$. In particular, we have $\ep+\ep_0< 1/16-\ep_0$.
Thanks to Proposition \ref{prop:limit} and the invariance under translation, under the assumption $\sides(\mathcal B_G)>32$, we have
 \begin{equation}
     \mathbb P \big( \lfloor Re^{i\theta } \rfloor+w^+\in \gamma^+ \text{ for some } R\ge \|y\|^{1/16+\epsilon_0} \text{ and } |\theta|\ge \varphi_4\big) \le C\exp \big( -\|y\|^{c_{\epsilon_0} } \big) \,.
 \end{equation}
Let $\theta_1$ be the angle in absolute value between the axis $x=-\ell$ and the line $\mathcal D$ that joins $(-\ell,\kappa \ell)$ and $(\ell,\ell)$ (see figure \ref{fig:coalescence}).
 It is easy to check that 
 \begin{equation}
      \tan \theta_1 =\frac{2\ell}{\kappa \ell-\ell}\le \tan(\pi/2-\varphi_4).
 \end{equation}
Let us assume that $\gamma^+$ does not stay above $\Lambda_\ell$. Then $\gamma^+$ intersects the set $\{-\ell,\dots,\ell\}\times (-\infty,\ell]$. 
Recall that
\begin{equation}\label{eq:ineqsides}
    \frac{1}{16}+\ep_0< \frac{1}{8}-\ep-\ep_0.
\end{equation}
It yields that 
\begin{equation}\label{eq:gammalarger}
    (\kappa-1)\ell\ge \|y\|^ {  \frac{1}{16}+\ep_0}
\end{equation}
where $ (\kappa-1)\ell$ corresponds to the distance between $w^+$ and $\Lambda_\ell$.
Thanks to Proposition \ref{prop:limit} since the angle $-\pi/2+\theta_1<-\varphi_4$, 
$\gamma^+$ intersects the set $\{-\ell,\dots,\ell\}\times (-\infty,\ell]$ with probability at most $C\exp(-c\|y\|^{c_{\ep_0}})$.

 It follows that $\gamma^+$ stays above $\Lambda_\ell$ with probability at least $1-C\exp(-\log ^2\|y\|)$. By similar arguments, we conclude with high probability at least $1-C\exp(-c\log ^2\|y\|)$, the geodesic $\gamma^+ $ stays above $\Lambda_\ell +y$ and the geodesic $\gamma^-$ stays below $\Lambda_\ell$ and $\Lambda_\ell +y$.
 \begin{figure}[!ht]
\def\svgwidth{0.8\textwidth}
 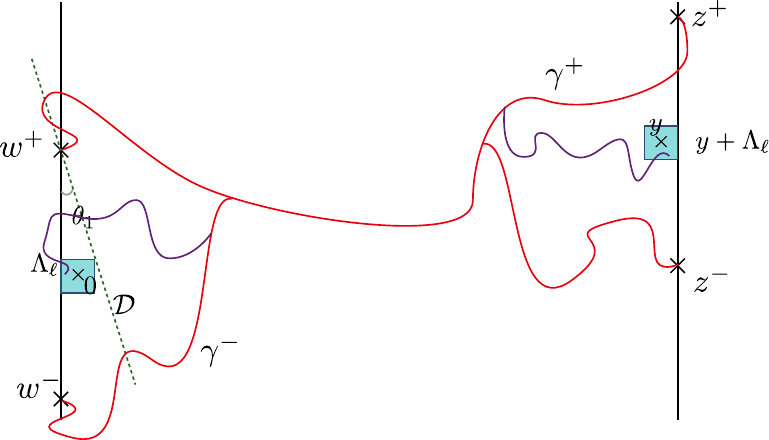
 \caption[figcoalescence]{\label{fig:coalescence}Illustration of the proof of Theorem \ref{mainthm2}} 
\end{figure}
Lastly, we need to prove that any geodesic $\gamma$ starting at a point $u\in\Lambda_\ell$ and ending at a point $v\in(y+\Lambda_\ell)$ cannot exit the trap. The only option for $\gamma$ to exit the trap is to leave the slab and intersect the set $V$. It is easy to check that the direction of the geodesic between $u$ and $v$ is contained in $[-\pi/4,\varphi_0]$ for $\|y\|$ large enough.
 Let $w$ be the last point of intersection of $\gamma$ with the line $x=-\ell$. Thanks to Proposition \ref{prop:limit}, we have that $\|w-u\|\le \|y\|^{1/16+\ep_0}$ with probability at least $1-C\exp(-\|y\|^{c_{\ep_0}})$. It follows from a similar use of  Claim \ref{claim:length} as above, that with probability at least $1-C\exp(-c\log^2\|y\|)$, the geodesic $\gamma$ does not go through a vertex in $V$. By union bound, it follows that any geodesic $\gamma$ from $\Lambda_{\ell}$ to $y+\Lambda_\ell$ cannot exit the trap with probability at least $1-C\exp(-c\log^2\|y\|)$. Finally, we have
 \begin{equation}\label{eq:trap}
     \PP(\mathcal T^c)\le C\exp(-c\log ^2\|y\|)
 \end{equation}
and on the event $\mathcal T$, when $\gamma^-$ and $\gamma^+$ coalesce, any geodesic $\gamma$ from $\Lambda_{\ell}$ to $y+\Lambda_\ell$ will also coalesce : $\gamma^-\cap\gamma^+\subset \gamma$.

 We turn to show that $\gamma^-$ and $\gamma ^+$ coalesce with high probability. Fix $\rho$ such that Proposition~\ref{prop:nobigjumps} holds and $\rho_2$ such that Lemma~\ref{lem:basic geometric control} holds (depending only on $G$). Let $\delta>0$ depending on $\ep$. Set \[N:=\lfloor \|y\|^{(1-\delta)/4}\rfloor\qquad  r:=\alpha_\rho N \log^{-2}\|y\| \quad  \text{and}\qquad m:=\lfloor \|y\|^{3(1-\delta)/4}/(10\rho_2(1+\rho))\rfloor\]
where $\alpha_\rho$ is the constant in Proposition~\ref{prop:attractive geodesics intro}.
 Let $(I_k)_{k=0}^{N-1}$ be intervals of the form $I_k := [a_k, a_{k+1}]$ where $a_k:=-\ell +km$. Let us prove that the geodesics $\gamma^-$ and $\gamma^+$ are $r$-close on most of the intervals $(I_k)_{k=0}^{N-1}$.
We recall that the definition of $r$-closeness interval was defined before Proposition \ref{prop:attractive geodesics intro}. We first need to prove that with high probability $\gamma^+$ stays above $\gamma^-$.
By translation invariance in law of the environment, it yields that
\begin{equation}
\forall x\in \{-\ell,\dots , a_N\} \qquad \EE [f_{\gamma^+}(x)]= \EE[f_{\gamma^-}(x)]+2\kappa \ell\,.
\end{equation}
If there exists $x$ such that $f_{\gamma^+}(x)\le f_{\gamma^-}(x)$, then one of the geodesic has to circle around the other: the event $\cE^-\cup\cE^+$ occurs where
\[\cE^-:=\left\{\gamma^-\cap \{ (-\ell,t): t\ge \kappa \ell\}\cup\{y+ (-\ell,t): t\ge \kappa \ell\}\ne\emptyset\right\}\]
and
\[\cE^+:=\left\{\gamma^+\cap \{ (-\ell,t): t\le \kappa \ell\}\cup\{y+ (-\ell,t): t\le \kappa \ell\}\ne\emptyset\right\}\,.\]
By Proposition \ref{prop:limit}, we have
\[\PP(\cE^-\cup\cE^+)\le C\exp(-c\|y\|^{c_\ep})\,.\]
It yields that
\begin{equation}
\begin{split}
\EE \big[ |f_{\gamma^+}(x)-f_{\gamma^-}(x)| \big] &=\EE \big[ f_{\gamma^+}(x)-f_{\gamma^-}(x) \big]+ 2\EE \big[ (f_{\gamma^-}(x)-f_{\gamma^+}(x))\mathds{1}_{f_{\gamma^-}(x)\ge f_{\gamma^+}(x)} \big] \\
&\le 2\kappa \ell+2\sqrt{\EE\big[ ( f_{\gamma^-}(x)-f_{\gamma^+}(x))^2 \big] \PP(\cE^-\cup\cE^+)}
\end{split}
\end{equation}
where we used Cauchy-Schwarz inequality in the last line.
Thanks to Claim \ref{claim:length}, the quantity $\EE \big[ (f_{\gamma^-}(x)-f_{\gamma^+}(x))^2\big] $ is at most polynomial in $\|y\|$.
Hence, there exists a positive constant $C_0$ such that
for every $x\in\{-\ell,\dots,a_N\}$
\[\EE \big[  |f_{\gamma^+}(x)-f_{\gamma^-}(x)| \big] \le C_0\ell\,.\]
Thus, by Markov's inequality, we have
\begin{equation}\label{695}
\forall x\in \{-\ell,\dots , a_N\} \qquad  \mathbb P \big( |f_{\gamma^+}(x)-f_{\gamma^-}(x)|\ge r \big) \leq \frac{C_0\ell }{r}
\end{equation}
and
\begin{equation}\label{696}
\begin{split}
\mathbb P\bigg( \big|\big\{x\in I_k \ : \ |f_{\gamma^+}(x)-f_{\gamma^-}(x)|\le r \big\} \big| \le & \frac{|I_k|}{2}\bigg)
\leq \mathbb P\bigg(\sum_{x\in I_k} |f_{\gamma^+}(x)-f_{\gamma^-}(x)| \ge  \frac{r|I_k|}{2}  \bigg) \\
 &\le \frac{2}{r|I_k|} \sum_{x\in I_k} \mathbb E \big[ |f_{\gamma^+}(x)-f_{\gamma^-}(x) \big] \leq \frac{2C_0\ell}{r}\,.
\end{split}
\end{equation}
Using \eqref{695} for the endpoints of $I_k$ and \eqref{696} we obtain 
\begin{equation}
\mathbb P\big( \gamma^-\text { is $r$-close to $\gamma^+$ on $I_k$} \big) \geq 1-\frac{4C_0\ell}{r}\,.
\end{equation}
Thus,
\begin{equation}
\mathbb E \left[ \left|\left\{k\in\{0,\dots, N-1\}: \,\gamma^-\text { is not $r$-close to $\gamma^+$ on $I_k$}\right\}\right|  \right] \leq  N\frac{4C_0\ell}{r}.
\end{equation}
Finally, we can control the total number of $r$-close intervals using again Markov's inequality
\begin{equation}\label{eq:markov1}
\mathbb P \left( \left|\left\{k\in\{0,\dots, N-1\}: \,\gamma^-\text { is not $r$-close to $\gamma^+$ on $I_k$}\right\}\right|\geq \xi N \right)\leq \frac{C\log ^2\|y\|}{\|y\|^{\ep-\delta/8}},
\end{equation}
where $\xi =\alpha _\rho /(\sqrt{r}\log \|y\|)$ is from Proposition~\ref{prop:attractive geodesics intro}.
By Proposition \ref{prop:nobigjumps}, thanks to our choice of $\rho$, we have
\begin{equation}\label{eq:noboundedslope}
    \mathbb P \big( \gamma^-\text{ has a $(\rho,m)$-bounded slope} \big) \ge 1- Ce^{-c\log ^2\|y\|}\,.
\end{equation}
Note that if a path has $(\rho,m)$-bounded slope, it is also true for any subpath.

By Proposition~\ref{prop:attractive geodesics intro}, translation invariance and a union bound we have that
\begin{equation}\label{eq:attgeoest}
    \mathbb P \Big( \! \begin{array}{c}\text{for all $|h|\le 2\rho \|y\|$ the geodesic $\gamma (w^-,(a_N,h) )$ is either }\\ 
    \text{attractive or it doesn't have a $(\rho ,m)$ bounded slope }\end{array} \!\! \Big) \ge 1-C e^{-c \log^2 \|y\|}.
\end{equation}

On intersection of the events in \eqref{eq:noboundedslope} and \eqref{eq:attgeoest} and the complement of the events in \eqref{eq:markov1}, the geodesics $\gamma ^-$ and $\gamma ^+$ intersect before reaching the line $\{x=-\ell +Nm\}$. Thus,
\begin{equation}\label{eq:interleft}
\PP \big( \text{$\gamma^-$ and $\gamma^+$ are edge disjoint on the interval $[-\ell,-\ell+Nm]$} \big) \le \frac{C\log ^2\|y\|}{\|y\|^{\ep-\delta/8}}\,.
\end{equation}
By the same arguments, we have
\begin{equation}\label{eq:interright}
\PP \big( \text{$\gamma^-$ and $\gamma^+$ are edge disjoint on the interval $[\ell+y_1-Nm,\ell+y_1]$} \big) \le  \frac{C\log ^2\|y\|}{\|y\|^{\ep-\delta/8}}\,.
\end{equation}
 Let us denote by $\mathcal F$ the event where $\gamma^-$ and $\gamma^+$ intersect on the intervals 
$[-\ell,-\ell+Nm]$ and $[\ell+y_1-Nm,\ell+y_1]$.
Let us now control the symmetric difference of geodesics starting at $\Lambda_\ell$ and ending at $y+\Lambda_\ell$ on the event $\mathcal T\cap\mathcal F$.
Let $\gamma_0$ and $\gamma_1$ be two geodesics starting at a point of $\Lambda_\ell$ and ending at a point in $y+\Lambda_\ell$. On the event $\mathcal T\cap\mathcal F$, the geodesics are trapped: $\gamma_0$ and $\gamma_1$ coalesce on the interval $[-\ell+Nm,\ell+y_1-Nm]$. 

Let $w'$ and $z'$ be respectively the first intersection of $\gamma^-$ with $x=-\ell+Nm$ and $x=\ell+y_1-Nm$. Note that both $\gamma_0$ and $\gamma_1$ also intersect $w'$ and $z'$ and coincide between these two points. We can upper bound the symmetric difference by the length of the subpaths of $\gamma_0$ and $\gamma_1$ from their endpoints to $w'$ and $z'$. With probability at least $1-C e^{-c\log^2\|y\|}$, by inequality \eqref{eq:noboundedslope}, we have $\|w'-w^-\|\le (1+\rho)(Nm+\ell) $ and $\|z'-z^-\|\le (1+\rho)(Nm+\ell)$.
It follows that for $\|y\|$ large enough (depending only on $G$)
\begin{equation}
  \sup_{w\in\Lambda_\ell}\|w-w'\|\le \sup_{w\in\Lambda_\ell}\|w-w^-\|+\|w'-w^-\|\le (\kappa+2d)\ell  + (1+\rho)(Nm+\ell)\le \frac{1}{8\rho_2}\|y\|^{1-\delta}. 
\end{equation}
Similarly, we have
\begin{equation}
  \sup_{z\in(y+\Lambda_\ell)}\|z-z'\|\le  \frac{1}{8\rho_2}\|y\|^{1-\delta}. 
\end{equation}
Finally, combining the two previous inequalities, on the event $\OB$ (with $L=\|y\|$), we have
\begin{equation}\label{eq:contsymdif}
|\gamma_0\triangle\gamma_1|\le 2\rho_2 \bigg(  \sup_{w\in\Lambda_\ell}\|w-w'\|+ \sup_{z\in(y+\Lambda_\ell)}\|z-z'\|\bigg)\le\frac{1}{2} \|y\|^{1-\delta}\,.
\end{equation}

By combining the previous inequality with inequalities \eqref{eq:trap}, \eqref{eq:noboundedslope}, \eqref{eq:interleft}, \eqref{eq:interright} and Lemma \ref{lem:basic geometric control}, there exists a constant $C$ depending only on $G$ such that
 \begin{equation}
      \mathbb P \Big(\exists u,z\in \Lambda_{\|y\|^{1/8 -\epsilon}}\, \exists v,w\in  (y+\Lambda_{\|y\|^{1/8 -\epsilon}})\quad \!\! \big| \gamma (u,v) \triangle \gamma (z,w) \big| > \frac{1}{2}\|y\|^{1-\delta } \Big) \le \frac{C\log^2 \|y\|} {\|y\|^{\ep-\delta/8 }}
  \end{equation}
  as needed.
\end{proof}
When the endpoints of the geodesics are getting closer ($\ep$ increasing), we need more sides on the limit shape for the trap to be efficient. Indeed, it is easier to circle the other geodesic when the endpoints are getting closer. For our later application of this theorem to prove the quantified version of BKS, we will need to use it for $\ep=1/16$. We state here another version of the theorem that will be sufficient for this application.
\begin{thm}\label{mainthm2bis}
   Suppose $G$ satisfies \eqref{eq:assumption i}, \eqref{eq:assumption ii}. Under the assumption $\sides(\mathcal B_G)>32$, for each $\ep\in(0,1/16)$, there exists $C_\ep>0$ (depending only on $G$ and $\ep$) such that for all $y\in\ZZ^2$ with $\|y\|\ge 2$,
  \begin{equation}
      \mathbb P \Big(\exists u,z\in \Lambda_{\|y\|^{1/8 -\epsilon}}\, \exists v,w\in  (y+\Lambda_{\|y\|^{1/8 -\epsilon}})\quad \!\! \big| \gamma (u,v) \triangle \gamma (z,w) \big| > \frac{1}{4}\|y\| \Big) \le \frac{C_\ep\log^2 \|y\|} {\|y\|^{\ep }}.
  \end{equation}
Under the assumption $\sides(\mathcal B_G)>40$, there exists $C>0$ such that for any $\ep\in(0,1/16]$,  for all $y\in\ZZ^2$,
  \begin{equation}
      \mathbb P \Big(\exists u,z\in \Lambda_{\|y\|^{1/8 -\epsilon}}\, \exists v,w\in  (y+\Lambda_{\|y\|^{1/8 -\epsilon}})\quad \!\! \big| \gamma (u,v) \triangle \gamma (z,w) \big| > \frac{1}{4}\|y\| \Big) \le \frac{C\log^2 \|y\|} {\|y\|^{\ep }}.
  \end{equation}
\end{thm}
To prove this version, one needs to slightly adapt the proof of Theorem \ref{mainthm2}. Under the assumption $\sides(\mathcal B_G)>32$, one needs to choose $\ep_0$ depending on $\ep$. As a result, the inequality \eqref{eq:gammalarger} will hold for $\|y\|$ large enough depending on $\ep$.

If we assume that $\sides(\mathcal B_G)>40$, we can chose $\ep\le 1/12$ and $\ep_0:=1/128$. We apply Proposition \ref{prop:limit} for $k=5$.
The inequality \eqref{eq:ineqsides} becomes
\[
\frac{1}{32}+\ep_0< \frac{1}{8}-\ep\]
and inequality \eqref{eq:gammalarger} will hold for $\|y\|$ large enough depending only on $G$.

\subsection{BKS midpoint problem}
In this section, we prove Theorem \ref{prop:probedgebulk} using the quantitative coalescence result of Theorem \ref{mainthm2bis}.
\begin{proof}[Proof of Theorem \ref{prop:probedgebulk}]We assume that $\sides(\mathcal B_G)>40$. Let $\ep\in(0,1/12]$. 
Let $u,v,z\in\ZZ^2$. Without loss of generality let us assume that $2\le \|u-z\|\le \|v-z\|$.
Set \[\ell:=\|u-z\|^{1/8-\ep}\,.\]
We will use here an averaging trick by considering all geodesics from $u+\Lambda_\ell$ to $v+\Lambda_\ell$.
Set $\cE_0$, $\cE_1$ be the following coalescence events
\[\cE_0 := \left\{\forall w_0,w_1\in (u+\Lambda_{\ell})\, \forall v_0,v_1\in  (z+\Lambda_{\ell})\quad  \big| \gamma (w_0,v_0) \triangle \gamma (w_1,v_1) \big| \le \frac{1}{2}\|u-z\|  \right\}\,\]
and
\[\cE_1 := \left\{\forall w_0,w_1\in (z+\Lambda_{\ell})\, \forall v_0,v_1\in  (v+\Lambda_{\ell})\quad  \big| \gamma (w_0,v_0) \triangle \gamma (w_1,v_1) \big| \le \frac{1}{2}\|v-z\| \right\}\,.\]
Thanks to Theorem \ref{mainthm2bis},
we have
\begin{equation}\label{eq:conte0e1}
    \PP(\cE_0\cap\cE_1)\ge 1-\frac{2C\log ^2 \|u-z\|}{\|u-z\|^{\ep}}\,.
\end{equation}
Let $\gamma_1$ and $\gamma_2$ be two geodesics with starting points in $u+\Lambda_\ell$ and ending points in $v+\Lambda_\ell$. On the event $\cE_0\cap \cE_1$, if $\gamma_1\cap (z+\Lambda_\ell)\ne\emptyset$ and $\gamma_2\cap (z+\Lambda_\ell)\ne\emptyset$, then $\gamma_1$ and $\gamma_2$ must intersect before and after intersecting $z+\Lambda_\ell$. Hence,
\[\gamma_1\cap (z+\Lambda_\ell)=\gamma_2\cap (z+\Lambda_\ell)\,.\]
By translation invariance, we have
\begin{equation}
    \begin{split}
    |\Lambda_\ell|\cdot \mathbb P(z\in\gamma&(u,v))=\sum_{w\in \Lambda_\ell}\mathbb P (z+w\in\gamma(u+w,v+w))\\
    &=\EE\Big[ \sum_{w\in \Lambda_\ell}\ind_{z+w\in\gamma(u+w,v+w)}\Big]\\
    &=\EE\Big[ \sum_{w\in \Lambda_\ell}\ind_{z+w\in\gamma(u+w,v+w)}\ind_{\cE_0^c\cup\cE_1^c}\Big]+\EE\Big[ \sum_{w\in \Lambda_\ell}\ind_{z+w\in\gamma(u+w,v+w)}\ind_{\cE_0\cap\cE_1}\Big]\\
    &\leq |\Lambda_\ell|\cdot \mathbb P \big( \cE_0^c\cup \cE_1^c   \big)+\EE\Big[ \max_{w,x\in\partial \Lambda_\ell}|\gamma(w,x)\cap \Lambda_\ell| \Big] \,.
    \end{split}
\end{equation}
By Lemma~\ref{lem:basic geometric control} (applied for $L=\|y\|$), we have for $\|y\|$ large enough (depending on $G$)
\begin{equation}
\begin{split}
 \EE\Big[ \max_{w,x\in\partial \Lambda_\ell}|\gamma(w,x)\cap \Lambda_\ell| \Big]&\le\EE\Big[ \max_{w,x\in\partial \Lambda_\ell}|\gamma(w,x)\cap \Lambda_\ell|\ind_{\OB} \Big]+\EE\Big[ | \Lambda_\ell|\ind_{\OB^c} \Big]\\
 &\le 6\rho_2 \ell +|\Lambda_\ell|C\exp(-c\log^2\|y\|)\le 7\rho_2\ell
 \end{split}
\end{equation}
Combining the two previous inequalities together with \eqref{eq:conte0e1}, it follows that
\[\mathbb P(z\in\gamma(u,v))\le C\left( \frac{\log ^2 \|u-z\|}{\|u-z\|^{\ep}}+\frac{1}{\|u-z\|^{1/8-\ep}}\right)\,.\]
By taking $\ep=1/16$, we get
\[\mathbb P(z\in\gamma(u,v))\le  \frac{2C\log ^2 \|u-z\|}{\|u-z\|^{1/16}}=\frac{2C\log ^2 \min\{\|u-z\|,\|v-z\|\}}{\min\{\|u-z\|,\|v-z\|\}^{1/16}}\,.\]
The result follows.
\end{proof}
Under the weaker assumption $\sides(\mathcal B_G)>32$, thanks to Theorem \ref{mainthm2bis}, we can prove that for every $\ep>0$, there exists $C_\ep>0$ (depending on $G$ and~$\ep$) such that for all $u,v,z\in\ZZ^2$, \begin{equation}\label{eq:BKS variant}
    \mathbb P\big(z\in\gamma(u,v)\big)\le\frac{C_\ep(\log  \min\{\|u-z\|,\|v-z\|\})^3}{\min\{\|u-z\|,\|v-z\|\}^{\frac{1}{16}-\ep}}.
\end{equation}

\subsection{The density of visited points on the vertical axis}
In this section we prove Theorem \ref{thm:density}, which, for fixed $n$, provides a quantitative control on the density of points on the vertical axis which are visited by a geodesic between $(-n,s)$ and $(n,s)$ for some $s$. 

\begin{proof}[Proof of Theorem \ref{thm:density}] We assume that $\sides(\mathcal B_G)>40$.
Let $n\ge 2$ and $m=\lfloor n^{1/24}\rfloor$. 
Set
\[Z_{n,m}:=\bigg\{ t\in\ZZ :  (0,t)\in\bigcup_{ s\in\{0,\dots,m\}}\gamma((-n,s),(n,s))\bigg\}.\]
Using translation invariance, we have
\begin{equation}\label{eq:bound on expectation of Z n m}
\begin{split}
   &m \EE[|Z_{n,m}|]=m \EE\Big[\Big| \big\{ (0,t): t\in\ZZ \big\} \cap  \bigcup_{s\in\{0,\dots,m\}}\gamma((-n,s),(n,s))\Big|\Big]\\
   &\ \ \ \ \le \EE\Big[\Big| \big\{ (x,t):  x\in\{0,\dots,m\}, t\in\ZZ \big\} \cap  \bigcup_{ x,s\in\{0,\dots,m\} }\gamma((-n+x,s),(n+x,s))\Big|\Big].
\end{split}
\end{equation}
Denote $\Lambda_0:= [0,m]^2+(-n,0)$, $\Lambda_1:= [0,m]^2+(n,0)$ and $E:=\{(x,t):  x\in\{0,\dots,m\}, t\in\ZZ\}$.
Let us prove that for any $x,s\in \{0, \dots ,m\}$ the geodesic $\gamma((x-n,s),(x+n,s))$ does not stay too long in the set $E$. Fix $x,s$ in $\{0, \dots ,m\}$. Denote by $z_0=(0,a)$ the first intersection point of the geodesic $\gamma:=\gamma((x-n,s),(x+n,s))$ with the set $E$ and by $z_1=(m,b)$ the last intersection point with $E$ (note that $z_0$ and $z_1$ implicitly depend on $x$ and $s$). We proceed to prove that there exists a constant $\kappa$ depending on $G$ such that with high probability $|b-a|\le \kappa m$. 

To this end, denote by $\theta_1$ the (random) angle that the line between $(x-n,s)$ and $z_0$ forms with the horizontal line.  
By the second part of Proposition~\ref{prop:limit} (when $\theta =0$, $k=2$ and $\ep=1/4$) there exists $\varphi _2 \in (0,\pi /4)$ such that 
\begin{equation}
    \mathbb P\left( \begin{array}{c} w+(-n,0)+\lfloor Re^{i\theta } \rfloor\in \gamma(w+(-n,0),w+(n,0))\\ \text{ for some } w\in \{0,\dots ,m\}^2,  |\theta|\ge\varphi_2 \text{ and } R\ge n^{1/2}\end{array}\right)\le C\exp(-n^{c}),
\end{equation}   
where in here we also union bound over the points in $\Lambda _0$.
Hence, with probability at least $1-C\exp(-n^{c})$, we have $|\theta_1|\le \varphi_2\le \pi/4$ and therefore $|a-s|\le |x-n||\tan(\theta_1)|\le n$.

Denote by $\cE$ the following event
\begin{equation*}
    \cE:=\bigcup_{x,s\in\{0,\dots,m\}}\bigcup_{l:|l-s|\le n}\left\{\! \! \begin{array}{c} (0,l) +\lfloor Re^{i\theta } \rfloor\in \gamma((0,l),(n+x,s))\\ \text{ for some } R\ge n^{2^{-5}+1/128} \text{ and } |\theta|\ge \varphi_5\end{array}\right\}
\end{equation*}
with $\varphi_5\in(\pi/4,\pi/2)$ the angle from Proposition~\ref{prop:limit} for $k=5$ (corresponding to $\theta_5)$.
For $x,s\in\{0,\dots,m\}$ and $|l-s|\le n$, denote by $\theta_2$ the angle that the line between $(0,l)$ and $(n+x,s)$ forms with the horizontal line. We have
\[|\tan\theta_2|= \frac{|l-s|}{n+x}\le 1\]
and $|\theta_2|\le\pi/4$.
Using a union bound over $l$ and $x,s \in \{0,\dots ,m\}$ and Proposition~\ref{prop:limit} with $k=5$, $\ep=1/128$ and $\theta _0=\theta _2$ we have $\mathbb P(\cE)\le C\exp(-n^ {c})$.
Finally, on the event 
\begin{equation}
 \mathcal G:=\cE^ c\cap \big\{ \forall x,s\in\{0,\dots,m\}\quad z_0\in \{(0,l):|l-s|\le n \}\big\}   
\end{equation}
 since $\|z_0-z_1\|\ge m\ge n^ {1/32+1/128}$ (note that $1/32+1/128< 1/24$), we have
 \[|b-a|\le \kappa m\,\]
 where $\kappa=\tan(\varphi_5)$.

We may now continue~\eqref{eq:bound on expectation of Z n m}. Let $\mathcal C $  be the event that all the geodesics with starting point in  $\Lambda_0$ and ending point in  $\Lambda _1$ coalesce and have the same intersection with $E$. By Theorem \ref{mainthm2bis} with $\epsilon=1/12$, we have $\mathbb P(\mathcal C)\ge 1-C\log ^2n/n^{1/12}$.

By similar computations as in the proof of Theorem \ref{prop:probedgebulk}, we get
\begin{equation*}
\begin{split}
   m \EE[|Z_{n,m}|]&\le C\rho_2\kappa m+ \rho_2\kappa \frac{m^3 C\log ^2 n}{ n^\frac{1}{12}}+ C\rho_2 nm^2 \cdot \exp (-n^c),
\end{split}
\end{equation*}
where the first term is the contribution to the expectation in the right hand side of \eqref{eq:bound on expectation of Z n m} from the event $\mathcal G\cap \mathcal C$, the second term from the event $\mathcal G\cap \mathcal C^c$ and the last term from the event $\mathcal G^c$.
Hence, we get for $n$ large enough $ \EE[|Z_{n,m}|]\le 2C\rho_2\kappa \log ^2 n$.
Thus,
\begin{equation*}
\begin{split}
    n&\cdot \PP \Big( (0,0)\in \bigcup_{s\in\ZZ}\gamma((-n,s),(n,s)) \Big) \\
    &\le \mathbb E \Big[ \Big| \Big\{ 1\le t\le n :  (0,t)\in \bigcup_{s\in\ZZ}\gamma((-n,s),(n,s)) \Big\} \Big|  \Big]\\
    &\le  \mathbb E \Big[ \Big| \Big\{ 1\le t\le n :  (0,t)\in \bigcup_{|s|\le (1+\rho _2)n}\gamma((-n,s),(n,s)) \Big\} \Big|  \Big] +Ce^{-c\log ^2 n}\\
    &\le \mathbb E \Big[ \Big| \Big\{ t \in \mathbb Z :  (0,t)\in \bigcup_{|s|\le (1+\rho _2)n}\gamma((-n,s),(n,s)) \Big\} \Big|  \Big] +Ce^{-c\log ^2 n} \\
    &\le \frac{C n}{m}\EE[|Z_{n,m}|]+ Ce^{-c\log ^2 n} \le C n ^{1-1/24}\log ^2n.
    \end{split}
\end{equation*}
where in the first inequality we used translation invariance, in the second inequality we used Lemma~\ref{lem:basic geometric control} and in the fourth inequality we used translation invariance once again. This finishes the proof of the theorem.
\end{proof}
\subsection{The density of long geodesics starting at the origin}
In this section, we prove Theorem~\ref{thm:density2}, which, for fixed $n$, provides a quantitative control on the density of points on long geodesics from the origin.

\begin{proof}[Proof of Theorem \ref{thm:density2}] 
We first suppose that Assumption~\eqref{eq:assumption not ell1} is satisfied, from which we will conclude that for all integer $n\ge 2$,
\begin{equation}\label{eq:long geodesic density bound}
    \mathbb E\left[\frac{|\mathcal T_{2n}\cap \Lambda_n|}{|\Lambda_n|}\right]\le C n^{-1/8}\log^2 n.
\end{equation}
Afterwards we will consider the alternative possibility. Following these, we will conclude the proof of the theorem.

Assume that~\eqref{eq:assumption not ell1} holds. Suppose, without loss of generality, that $n$ is sufficiently large for the following arguments.
Define 
\begin{equation}\label{eq:mathcal X}
  \mathcal X:=\min \bigg\{|E|: E\subset \partial\Lambda_{2n}, \ \ \mathcal T_{2n} \cap \Lambda_n\subset \bigcup_{x\in E}\gamma(0,x)\bigg\}.  
\end{equation}
Our goal will be to show that $\mathcal X$ has sub-linear size with high probability. We let $E$ be a random set achieving the minimum in \eqref{eq:mathcal X}. In case, there are several such sets, we choose one according to a deterministic rule. Note that, by minimality of $E$, it holds that $\gamma(0,x)\cap \Lambda_n\ne \gamma(0,y)\cap \Lambda_n$ for all $x,y\in E$ with $x\neq y$.

For each $x\in E$, we associate $\tau(x)$, the first point in $\partial \Lambda_{2n}$ hit by $\gamma(0,x)$ (as it is traversed from $0$ to $x$). In particular, if $\gamma(0,x)$ remains inside $\Lambda_{2n}$ then $\tau(x)=x$.
Set 
\[\widetilde E:=\{\tau(x):x\in E\}.\]
First, let us show that with high probability,  $\widetilde E$ is also a set attaining the minimum in the definition of $\mathcal X$. That is, we show that there exists $c>0$, depending only on $G$, such that
\begin{equation}\label{eq:inclusion}
    \mathbb P\left(\bigcup_{x\in \widetilde E}\gamma(0,x)\cap \Lambda_n = \bigcup_{x\in E}\gamma(0,x)\cap \Lambda_n\right)\ge 1-e^{-cn}.
\end{equation}
Indeed, if this event does not occur, then there exists $y\in\partial \Lambda_{n}$ such that $\gamma(0,y)$ intersects $\partial \Lambda_{2n}$. 
Inequality~\eqref{eq:inclusion} then follows easily by Talagrand's inequality (Theorem \ref{thm:talagrand}) using that for all $y\in\partial\Lambda_n$ and $z\in\partial \Lambda_{2n}$ we have 
\[ \mu(y)\le \mu(n(e_1+e_2)) \le 2n\mu(e_1)\quad\text{and}\quad  \mu(z)+\mu(y-z)\ge \mu(2n e_1) + \mu(n e_1) = 3n\mu(e_1).\]
Here, $\mu$ is the time constant as defined in Section~\ref{sec:geometry in limit norm}, i.e., the norm for which $\mathcal B_G$ is the unit ball, and the inequalities are a simple consequence of the convexity of $\mathcal B_G$ and its invariance to lattice symmetries (see Claim~\ref{claim:ineqnu} for the second inequality).

Next, we enumerate the points in $\tilde{E}$ along the left boundary $\{-2n\}\times [-2n,2n]$ by $(-2n,y_i)$ where $-2n\le y_1<\dots<y_M\le 2n$. We start by bounding $M$ with high probability. To this end we define the following events. Fix $\rho >0$ sufficiently large so that Proposition~\ref{prop:nobigjumps} holds (this is where we use Assumption~\eqref{eq:assumption not ell1}) and let $m:=\lfloor n^{3/4} \rfloor $. Define the event
\begin{equation}
    \mathcal A := \bigcap _{|y|\le 2n}\big\{ \gamma ( (-2n,y),0 ) \text{ has }(\rho ,m) \text{ bounded slope} \big\}.
\end{equation}
By Proposition~\ref{prop:nobigjumps}, translation invariance and a union bound we have that $\mathbb P (\mathcal A ) \ge 1-\exp (-c\log ^2 n)$.

Next, let $N:=\lceil n /m \rceil -1$ and $r=\frac{\alpha_\rho m^ {1/3}}{\log^2 n}$ where $\alpha_\rho\in(0,1]$ is given in Proposition~\ref{prop:attractive geodesics intro}. For any $1\le k\le m$ define the intervals
\begin{equation}
I_i^k:=
    \begin{cases}
        [-2n,-2n+k+m] \quad i=0  \\
        [-2n+k+(i-1)m,-2n+k+im] \quad   1\le i\le N-2 \\
        [-2n+k+(i-1)m,-n]\quad i=N-1
    \end{cases}.
\end{equation}
Recall the definition of attractive geodesics given in Proposition~\ref{prop:attractive geodesics intro}. Define the event
\begin{equation}
    \mathcal B := \bigcap _{|y|\le 2n} \bigcap _{|x|\le 2n} \bigcap _{k=1}^{m} \Bigg\{ \!\!\! \begin{array}{c}\text{the geodesic $\gamma ( (-2n,y),(-n,x) )$ is either}\\ \text{attractive with respect to the intervals $(I_i^k)_i$ }\\
    \text{or it doesn't have a $(\rho ,m)$ bounded slope }\end{array} \!\!\! \Bigg\}.
\end{equation}
By Proposition~\ref{prop:attractive geodesics intro}, translation invariance and a union bound we have that $\mathbb P (\mathcal B ) \ge 1-\exp (-c\log ^2 n)$.

We claim that on the event $\mathcal A \cap \mathcal B $ we have that $M\le C n^{7/8}\log^2 n$.

To see this, let $\gamma_j =\gamma((-2n,y_j),0)$ for $j\le M$. By definition of $\widetilde E$, the geodesics $\gamma_j$ stay inside the square $\Lambda_{2n}$. Note that the geodesics $\gamma_j$ cannot intersect before reaching $\Lambda_n$, as that would contradict the minimality of $E$. It follows that the geodesics are ordered in the sense that $f_{\gamma_i}(x)< f_{\gamma_j}(x)$ for $i\le j$ and $x\in[-2n,-n]$ (recall that $f_\gamma(x)$ denotes a pioneer point as defined in Section \ref{c}).
We have
\[\sum_{j=1}^{M-1}\sum_{x=-2n}^{-n}f_{\gamma_{j+1}}(x)-f_{\gamma_{j}}(x)=\sum_{x=-2n}^{-n}f_{\gamma_M}(x)-f_{\gamma_1}(x)\le 4n(n+1). \]
By the pigeon-hole principle, it follows that there exists $1\le j_0\le M-1$ for which 
\begin{equation}\label{eq:toncotradictdensity}
    \sum_{x=-2n}^{-n}f_{\gamma_{j_0+1}}(x)-f_{\gamma_{j_0}}(x)\le \frac{4n(n+1)}{M-1}.
\end{equation}
There exists $1\le k_0\le m$ such that
\[\# \big\{ 1\le i\le N-2: f_{\gamma_{j_0+1}}(-2n+k_0+(i-1)m)-f_{\gamma_{j_0}}(-2n+k_0+(i-1)m)\ge r \big\} \le \frac{4n(n+1)}{mr(M-1)} \]
since otherwise it would contradict \eqref{eq:toncotradictdensity}.

Similarly, we have
\begin{equation}
    \begin{split}
        \#&\Big\{ 1\le i\le N-2: \#\{\ell\in I^{k_0}_i: f_{\gamma_{j_0+1}}(\ell)-f_{\gamma_{j_0}}(\ell)\ge r-1\}\ge \frac{m}{4} \Big\} \le \frac{16n(n+1)}{m(r-1)(M-1)}.
    \end{split}
\end{equation}
It follows that among the $(I^{k_0}_i)_i$ there are at most $2+\frac{8n(n+1)}{mr(M-1)}+\frac{16n(n+1)}{m(r-1)(M-1)}\le \frac{25n^2}{mr(M-1)}$ intervals where $\gamma_{j_0}$ and $\gamma_{j_0+1}$ are not $r$-close.

On the event $\mathcal A\cap \mathcal B$, the sub geodesic of $\gamma_{j_0}$ between the points $(-2n,y_{j_0})$ and $(-n,f_{\gamma_{j_0}}(-n))$ shares an edge with any geodesic that is $r$-close to it on at least $\big( 1-\frac{\alpha_\rho}{\sqrt r\log n} \big)N$ of the intervals. Thus, since the geodesics $\gamma_{j_0}$ and $\gamma_{j_0+1}$ don't intersect above the interval $[-2n,-n]$ it follows that
\[\frac{25n^2}{mr(M-1)}\ge \frac{\alpha_\rho}{\sqrt r\log n}N\]
and therefore
\begin{equation}\label{eq:bound on M}
    M\le 1+\frac{25n^2\log n}{\alpha_\rho m\sqrt r N}\le C n^{7/8}\log^2 n
\end{equation}
as long as $n$ is sufficiently large. This shows that $\mathbb P (M> C n^{7/8}\log^2 n) \le e^{-c\log ^2 n}$.
Finally, by \eqref{eq:inclusion} and using the $90^{\circ}$ rotation invariance we obtain
\begin{equation}
\label{eq:boundX}
\begin{split}
\mathbb P(\mathcal X> 4C n^{7/8}\log^2 n)&\le \mathbb P( |\tilde{E}|> 4C n^{7/8}\log^2 n) +e^{-cn} \\
&\le 4 \cdot \mathbb P (M> C n^{7/8}\log^2 n) +e^{-cn}  \le e^ {-c \log^2 n}.
\end{split}
\end{equation}
Using Lemma \ref{lem:basic geometric control} and \eqref{eq:boundX}, we get
\begin{equation}
\begin{split}
    \mathbb E\left[\frac{|\mathcal T_{2n}\cap \Lambda_n|}{|\Lambda_n|}\right]& \le   \mathbb P(\mathcal X> 4C n^{7/8}\log^2 n)+ \mathbb E\Big[\max_{x\in\partial\Lambda_{2n}}|\gamma(0,x)| \Big]  \frac{C n^{7/8}\log^2 n}{n^2} \le C n^{-1/8}\log^2 n.
    \end{split}
\end{equation}
This concludes the proof of~\eqref{eq:long geodesic density bound}, under Assumption~\eqref{eq:assumption not ell1}.

As the next step, we suppose that Assumption~\eqref{eq:assumption not ell1} is violated. In other words, we suppose that the limit shape $\mathcal B_G$ \emph{is} a dilation of the $\ell_1$ unit ball. To continue to apply Proposition~\ref{prop:attractive geodesics intro} and Proposition~\ref{prop:nobigjumps} under this assumption, we will employ a $45$-degree rotation of the $\mathbb{Z}^2$ lattice (as remarked following Proposition~\ref{prop:nobigjumps}).

Denote by $R:\mathbb{Z}^2\to\mathbb{R}^2$ the 45 degree rotation and scaling operation $R(x,y) = (x-y, x+y)$. Denote by $\tilde{\mathbb{Z}}^2:=R(\mathbb{Z}^2)$ the rotated and scaled lattice. Our hypothesis that $\mathcal B_G$ is a dilation of the $\ell_1$ unit ball implies that the limit shape for $\tilde{\mathbb{Z}}^2$ is a dilation of the $\ell_\infty$ unit ball. As remarked after Proposition~\ref{prop:nobigjumps}, the proofs of Proposition~\ref{prop:attractive geodesics intro} and Proposition~\ref{prop:nobigjumps} continue to apply in the rotated coordinate system (modifying the definitions in the beginning of Section~\ref{c} so that they are based on projections to the $x$-coordinate in $\tilde{\mathbb{Z}}^2$). Consequently, we may use the same arguments as in the proof of~\eqref{eq:long geodesic density bound} in order to deduce that for integer $n\ge 2$,
\begin{equation}\label{eq:long geodesic density bound2}
    \mathbb E\left[\frac{|\tilde{\mathcal T}_{2n}\cap \tilde{\Lambda}_n|}{|\tilde{\Lambda}_n|}\right]\le C n^{-1/8}\log^2 n,
\end{equation}
where
\begin{equation}
    \tilde{\Lambda}_n:=[-n,n]^2\cap \tilde{\mathbb{Z}}^2\quad\text{and}\quad\tilde{\mathcal T}_n:=\bigcup_{x\in \partial \tilde{\Lambda}_n}\tilde{\gamma}(0,x)
\end{equation}
and where $\partial\tilde{\Lambda}_n = \{x\in\tilde{\mathbb Z}^2\colon \|x\|_\infty = n\}$ and, for $z,w\in\tilde{\mathbb{Z}}^2$, $\tilde{\gamma}(z,w)$ denotes the geodesic in $\tilde{\mathbb{Z}}^2$ (i.e., $\tilde{\gamma}(z,w):=R(\gamma(R^{-1}z,R^{-1}w))$).

We now conclude the proof of Theorem~\ref{thm:density2}. If Assumption~\eqref{eq:assumption not ell1} is satisfied then the theorem follows directly from~\eqref{eq:long geodesic density bound} by noting that $|\mathcal T_{m}\cap \Lambda_n|$ is non-increasing in $m\ge n$. Suppose that Assumption~\eqref{eq:assumption not ell1} is violated. By inequality~\eqref{eq:long geodesic density bound2} and the fact that $R$ is one-to-one we have for each integer $m\ge 2$,
\begin{equation}\label{eq:long geodesic density bound3}
    \mathbb E\left[\frac{|R^{-1}(\tilde{\mathcal T}_{2m}\cap \tilde{\Lambda}_m)|}{|R^{-1}(\tilde{\Lambda}_m)|}\right]\le C m^{-1/8}\log^2 m.
\end{equation}
It remains to note that, for integer $n\ge 1$,
\begin{equation}\label{eq:containment for long geodesics}
    \frac{|\mathcal T_{4n}\cap \Lambda_n|}{|\Lambda_n|}\le \frac{|\mathcal T_{4n}\cap R^{-1}(\tilde{\Lambda}_{2n})|}{|R^{-1}(\tilde{\Lambda}_n)|}\le \frac{|R^{-1}(\tilde{\mathcal T}_{4n})\cap R^{-1}(\tilde{\Lambda}_{2n})|}{|R^{-1}(\tilde{\Lambda}_n)|}\le C\frac{|R^{-1}(\tilde{\mathcal T}_{4n}\cap \tilde{\Lambda}_{2n})|}{|R^{-1}(\tilde{\Lambda}_{2n})|}.
\end{equation}
where the first inequality uses the inclusions $R^{-1}(\tilde{\Lambda}_n)\subset \Lambda_n\subset R^{-1}(\tilde{\Lambda}_{2n})$, and the second inequality further uses the fact that the endpoints of the geodesics in $R^{-1}(\tilde{\mathcal{T}}_{4n})$ are contained in $\Lambda_{4n}$. Theorem~\ref{thm:density2} follows from~\eqref{eq:containment for long geodesics} and~\eqref{eq:long geodesic density bound3} with $m=2n$.
\end{proof}

\section{From the limit shape to the geometry of geodesics}\label{sec:the limit shape}

In this section, we assume some properties on the limit shape and derive properties of the geodesic in the limiting norm. From these properties, we can control the asymptotic behavior of geodesics. In particular, we show that typically a geodesic does not ``go in the wrong direction for long'' (Proposition \ref{prop:limit}). We also prove Proposition \ref{prop:nobigjumps}.

\subsection{Geometry of the geodesics in the limiting norm} \label{sec:geometry in limit norm}
 In this section, we prove a characterisation of being in the same flat edge of the limit shape and deduce some properties of the geodesics in the limiting norm. 
 The following theorem states that under some mild assumptions on the distribution $G$, one can prove that asymptotically when $n$ is large, the random variable $T(0,nx)$ behaves like $n\cdot \mu(x)$ where $\mu(x)$ is a deterministic constant depending only on the distribution $G$ and the point $x$. More precisely, we have the following theorem.
\begin{namedthm*}{Time constant}\label{thm:defmu}Let $G$ be a distribution such that $\EE [t_e]<\infty$.
There exists a deterministic function $\mu$ on $\mathbb R^2$ depending on $G$ such that
\begin{equation}\label{eq:defmu}
    \forall x \in\ZZ^d\qquad \lim_{n\rightarrow \infty}\frac{T(0,nx)}{n}=\mu(x)\quad\text{a.s. and in $L^1$.}
\end{equation}
The constant $\mu(x)$ is called the time constant of $x$.
\end{namedthm*}
 This constant may be interpreted as an inverse speed in the direction $x$.
Kesten proved in  \cite{Kesten:StFlour} that $\mu$ is a norm if and only if $G(\{0\})<p_c(d)$. Under our assumption \eqref{eq:assumption ii}, the function $\mu$ is a norm. In particular, one can prove that $\mathcal B_G$ is the unit ball for the norm $\mu$.

 We quantify how far a path that ``goes in the wrong direction'' is from being a $\mu$-geodesic.  
 The following claim is a useful property of the time constant. It will be used in the proof of Proposition \ref{prop:nobigjumps}.
 \begin{claim}\label{claim:ineqnu} We have
 \begin{equation}
     \forall x \in\mathbb R ^2\qquad \mu(x)\ge \|x\|_\infty \mu(1,0).
 \end{equation}
 \end{claim}
 \begin{proof} Set $x=(a,b)\in\mathbb R^2$. Without loss of generality assume $|a|\ge |b|$. By the triangle inequality, symmetry and homogeneity of $\mu$,
\begin{equation}
  |a|\mu(1,0)= \mu(a,0)\leq \mu(a/2,b/2)+\mu(a/2,-b/2)=2\mu(a/2,b/2)=\mu(a,b)
\end{equation}
The result follows.
 \end{proof}
 
 To lighten notation, in this section we shorten $\mathcal B_G$ to $\mathcal B$.
We let $\mathcal B (\theta ) $ be the unique $x$ such that $xe^{i\theta }$ is on the boundary of $\mathcal B$, where as usual we identify $\mathbb R ^2$ with $\mathbb C$.
The unit ball of the norm $\mu $ is $\mathcal B$ and therefore $\mu (Re^{i\theta })=R/\mathcal B (\theta )$.

We say that directions $\theta _1$ and $\theta _2$ are on the same flat edge of the limit shape if the interior of $\mathcal B$ does not intersect the line connecting $\mathcal B (\theta _1 )e^{i\theta _1}$ and $\mathcal B (\theta _2)e^{i\theta _2}$ (this is the line connecting the two points on the boundary of the limit shape that are at angles $\theta _1$ and $\theta _2 $). We say that $\theta $ is a vertex direction if there are two distinct lines passing through $\mathcal B (\theta )e^{i\theta }$ such that the interior of $\mathcal B$ does not intersect any of them (these are precisely the directions in which the limit shape is not differentiable).

\begin{claim}\label{claim:basic claim on limit shape}
  The following three statements are equivalent.
  \begin{enumerate}
      \item 
      The directions $\theta _1$ and $\theta _2$ are on the same flat edge of the limit shape.
      \item 
      For all $R_1,R_2\ge 0$ we have  
\begin{equation}\label{eq:equality in triangle}
    \mu \big( R_1e^{i\theta _1}+R_2e^{i\theta _2} \big) =\mu (R_1e^{i\theta _1})+\mu (R_2e^{i\theta _2}).
\end{equation}
\item 
For some $R_1,R_2>0$ we have  
\begin{equation}
    \mu \big( R_1e^{i\theta _1}+R_2e^{i\theta _2}\big) =\mu (R_1e^{i\theta _1})+\mu (R_2e^{i\theta _2}).
\end{equation}
  \end{enumerate}
\end{claim}
Note that by the triangle inequality and homogeneity, the right hand side of  \eqref{eq:equality in triangle} is always larger than the left hand side of \eqref{eq:equality in triangle}.
\begin{proof}
Statement (3) clearly follows from (2). We start by showing that (2) follows from (1). Suppose $\theta _1$ and $\theta _2$ are on the same flat edge.  Substituting $\mu (R_1e^{i\theta _1})=R_1/\mathcal B(\theta _1)$ and $\mu (R_2e^{i\theta _2})=R_2/\mathcal B (\theta _1) $ we obtain 
\begin{equation}\label{eq:affine}
    \frac{R_1e^{i\theta _1}+R_2e^{i\theta _2}}{\mu (R_1e^{i\theta _1})+\mu (R_2e^{i\theta _2})}=\frac{R_1\mathcal B(\theta _2)}{R_1\mathcal B(\theta _2)+R_2\mathcal B(\theta _1)} \mathcal B(\theta _1)e^{i\theta _1}+\frac{R_2\mathcal B(\theta _1)}{R_1\mathcal B(\theta _2)+R_2\mathcal B(\theta _1)} \mathcal B(\theta _2)e^{i\theta _2}.
\end{equation}
This is a linear combination of the points  $ \mathcal B(\theta _1)e^{i\theta _1}$  and $ \mathcal B(\theta _2)e^{i\theta _2}$ with coefficients that sum up to $1$. Thus, the point in the left hand side of \eqref{eq:affine} is on the flat edge containing $\theta _1$ and $\theta _2$ and its norm has to be $1$. This proves the equality in \eqref{eq:equality in triangle}.

The proof that (1) follows from (3) is similar. Indeed, if \eqref{eq:equality in triangle} holds for some $R_1$ and $R_2$ then the point in the left hand side of \eqref{eq:affine} is on the boundary of the limit shape. Since the point is on the line containing $\mathcal B(\theta _1)e^{i\theta _1}$  and $ \mathcal B(\theta _2)e^{i\theta _2}$, the angles $\theta _1$ and $\theta _2$ have to be on the same flat edge of the limit shape.
\end{proof}

The following claim is a quantitative version of Claim~\ref{claim:basic claim on limit shape}. Let $0\le \theta _1 <\theta _2 <\pi /2$ be two angles not on the same flat edge or such that $\theta_1$ is a vertex direction. The following claim proves that any path in $\mathbb R ^2$ from $(0,0)$ to $ R_1e^{i\theta _1} $ that contains the point $ R_2e^{i\theta _2}$ is very far from being a $\mu $-geodesic. 

\begin{claim}\label{claim:limit}
 Let $0\le \theta _1 <\theta _2 <\pi /2$ such that either $\theta _1$ and $\theta _2 $ are not on the same flat edge of the limit shape or $\theta _1$ is a vertex direction. There exists a constant $c_4$ depending on the edge distribution, $\theta _1$ and $\theta _2$ such that for all $R_1,R_2$
  \begin{equation}\label{eq:12}
      \mu (R_1e^{i\theta _1}) +c_4 R_2 \le \mu (R_2e^{i\theta _2}) +\mu (R_1e^{i\theta _1}-R_2e^{i\theta _2})
  \end{equation}
  and 
  \begin{equation}\label{eq:13}
      \mu (R_1e^{i\theta _1}) +c_4 R_2 \le \mu (R_2e^{-i\theta _2}) +\mu (R_1e^{i\theta _1}-R_2e^{-i\theta _2}).
  \end{equation}
\end{claim}

\begin{proof}
The assumption on $\theta _1$ and $\theta _2$ exactly ensures that there exists a line $l$ with a negative or infinite slope, that passes through $\mathcal B (\theta _1) e^{i\theta _1}$, is disjoint from the interior of $\mathcal B$ and does not contain $\mathcal B (\theta _2)e^{i\theta _2}$. Let $r_2 > 0$ be the unique radius for which $r_2e^{i\theta _2}$ is on $l$. Note that $r_2>\mathcal B (\theta _2)$ and therefore $\mu (r_2e^{i\theta _2})>1$. Thus, letting $R_1':=R_2\mathcal B (\theta _1)/r_2$ we have

\begin{equation}\label{eq:10}
    \mu (R_2e^{i\theta _2})-\mu (R_1'e^{i\theta _1} )= R_2/r_2\big( \mu (r_2e^{i\theta _2})-\mu (\mathcal B (\theta _1) e^{i\theta _1} ) \big) = c_4R_2,
\end{equation}
where 
\begin{equation}
    c_4:=\big( \mu (r_2e^{i\theta _2})-1 \big) /r_2
\end{equation}
is independent of $R_1$ and $R_2$.

Next, suppose that $R'_1\neq R_1$. In this case, by substituting the definition of $R_1'$ we get that the point
\begin{equation}
\begin{split}
    \frac{\mathcal B (\theta _1)}{R_1-R_1'}\big( R_1e^{i\theta _1}-R_2e^{i\theta _2} \big)&=\frac{r_2R_1}{r_2R_1-R_2\mathcal B (\theta _1)} \mathcal B (\theta _1) e^{i\theta _1} -\frac{\mathcal B (\theta _1)R_2}{r_2R_1-R_2\mathcal B (\theta _1) } r_2e^{i\theta _2}\\
    &= \mathcal B (\theta _1) e^{i\theta _1}+\frac{\mathcal B (\theta _1)R_2}{r_2R_1-R_2\mathcal B (\theta _1) }(\mathcal B (\theta _1) e^{i\theta _1}-r_2e^{i\theta _2})
    \end{split}
\end{equation}
is on the line $l$. See Figure~\ref{fig5}. 

\begin{figure}[!ht]
\def\svgwidth{0.5\textwidth}
 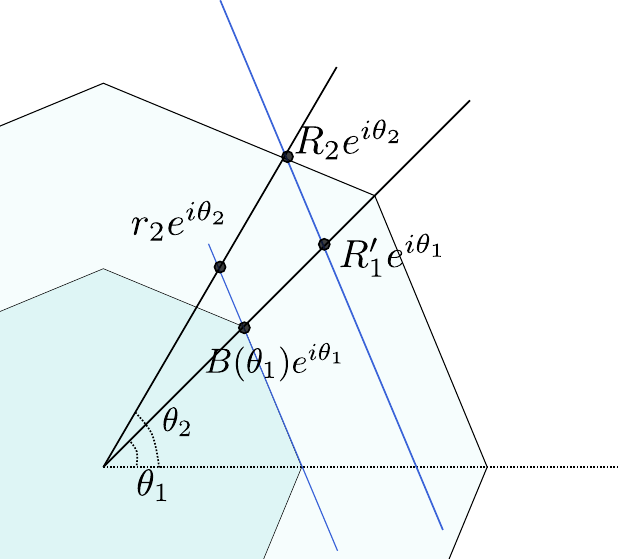
 \caption[fig5]{\label{fig5}Illustration of the proof of Claim \ref{claim:limit} }
\end{figure}
It follows that the norm of this point is at least $1$ and therefore 
\begin{equation}\label{eq:9}
    \mu \big( R_1e^{i\theta _1}-R_2e^{i\theta _2} \big) \ge \frac{|R_1-R_1'|}{\mathcal B (\theta _1)}=\mu \big(  (R_1-R_1')e^{i\theta _1} \big).
\end{equation}
Note that \eqref{eq:9} clearly holds also in the case that $R_1=R_1'$. Combining \eqref{eq:9} and \eqref{eq:10} we obtain 
\begin{equation}
    \mu (R_2e^{i\theta _2})+\mu (R_1e^{i\theta _1}-R_2e^{i\theta _2}) \ge \mu (R'_1e^{i\theta _1} )+c_4R_2+\mu \big( (R_1-R_1')e^{i\theta _1} \big) \ge \mu (R_1e^{i\theta _1})+c_4R_2.
\end{equation}
This finishes the proof of \eqref{eq:12}. The inequality \eqref{eq:13} follows from \eqref{eq:12}. Indeed, for all $x,y_1,y_2$ such that $|y_1|\le |y_2|$ we have that $\mu(x,y_1)\le \mu (x,y_2)$. Thus, $\mu (R_1e^{i\theta _1}-R_2e^{i\theta _2})\le \mu (R_1e^{i\theta _1}-R_2e^{-i\theta _2})$.
\end{proof}
The following claim relates the number of sides in the limit shape to the number of vertex directions in $[0,\pi/4]$.
\begin{claim}\label{claim:4}
  Suppose that the limit shape is not a polygon with $8k$ sides or less. Then, there exist
\begin{equation}
    \frac \pi4<\theta_0<\theta_1<  \dots< \theta _{k}<\frac\pi 2
\end{equation} 
such that for every $j\in\{0,\dots,k-1\}$, $\theta _j$ and $\theta _{j+1}$ are not on the same flat edge of the limit shape. 
\end{claim}

\begin{proof}
If the limit shape is not a polygon then this claim clearly follows. Indeed, in this case one can find infinitely many directions in $[\pi/4,\pi/2)$ such that no two of them are on the same flat edge of the limit shape. Next, suppose that the limit shape is a polygon with strictly more than $8k$ sides. By symmetry, the number of vertex directions is a multiple of $4$. Hence, there are at least $8k+4$ vertex directions. Let $m$ denote the number of vertex directions in $(\pi/4,\pi/2)$. By symmetry, there are at most $8m+8$ vertex directions in total. It follows that $m\ge k$. Let $\varphi_1<\dots<\varphi_k$ be vertex directions in $(\pi/4,\pi/2)$. Fix $\theta_0\in(\pi/4, \varphi_1)$, $\theta_j\in(\varphi_{j},\varphi_{j+1})$ for $j\in\{1,\dots,k-1\}$ and $\theta_{k}\in(\varphi_k,\pi/2)$. Since there exists a vertex direction between two consecutive $\theta_j$, they are not on the same flat edge. This concludes the proof.
\end{proof}


\subsection{Geometry of geodesics in FPP} In this section we prove Proposition \ref{prop:limit}. We deduce from the results on the geometry of the geodesics in the limiting norm, results on the geometry of geodesics in FPP under some assumption on the limit shape. The following result due to Alexander (Theorem 3.2 in \cite{alexander}), that bounds the ``non-random fluctuations'' of the passage time, enables to make the connection between geodesics in the limiting norm and geodesics in FPP.
\begin{thm}\label{claim:non random}
  If the edge distribution satisfies $\mathbb E [ e^{\alpha t_e}]< \infty $ for some $\alpha >0$. Then, for every $x\in \mathbb Z^2$ with $\|x\|\ge 2$ we have that
  \begin{equation}
      \big| \mathbb E [T (0,x)]-\mu (x) \big| \le C\sqrt{||x||} \log (||x||).
  \end{equation}
\end{thm}
The following lemma allows us to control the behavior of geodesics using properties of the limit shape. This lemma will be the key ingredient to prove Proposition \ref{prop:limit}.
\begin{lem}\label{lem:theta_1}
Let $0\le \theta _1 <\theta _2 <\pi /2$ satisfy either that $\theta _1$ and $\theta _2 $ are not on the same flat edge of the limit shape or that $\theta _1$ is a vertex direction. Let $\epsilon >0$ and let $\gamma $ be the geodesic from $(0,0)$ to $\lfloor R_1e^{i\theta _1 } \rfloor $. Then,
\begin{equation}\label{eq:exp bound}
\mathbb P \Big(  \lfloor R_2e^{i\theta _2} \rfloor  \in \gamma \text{ or }\lfloor R_2e^{-i\theta _2} \rfloor  \in \gamma  \text{ for some } R_2\ge R_1 ^{1/2+\epsilon } \Big) \le C \exp \big(-R_1^{c_\epsilon}\big),
\end{equation}
where the constant $c_\epsilon $ may depend on $\epsilon $, $\theta _1$ and $\theta _2 $.
\end{lem}

We postpone the proof of Lemma~\ref{lem:theta_1} for now and first prove Proposition~\ref{prop:limit}.
\begin{figure}[!ht]
\def\svgwidth{0.8\textwidth}
 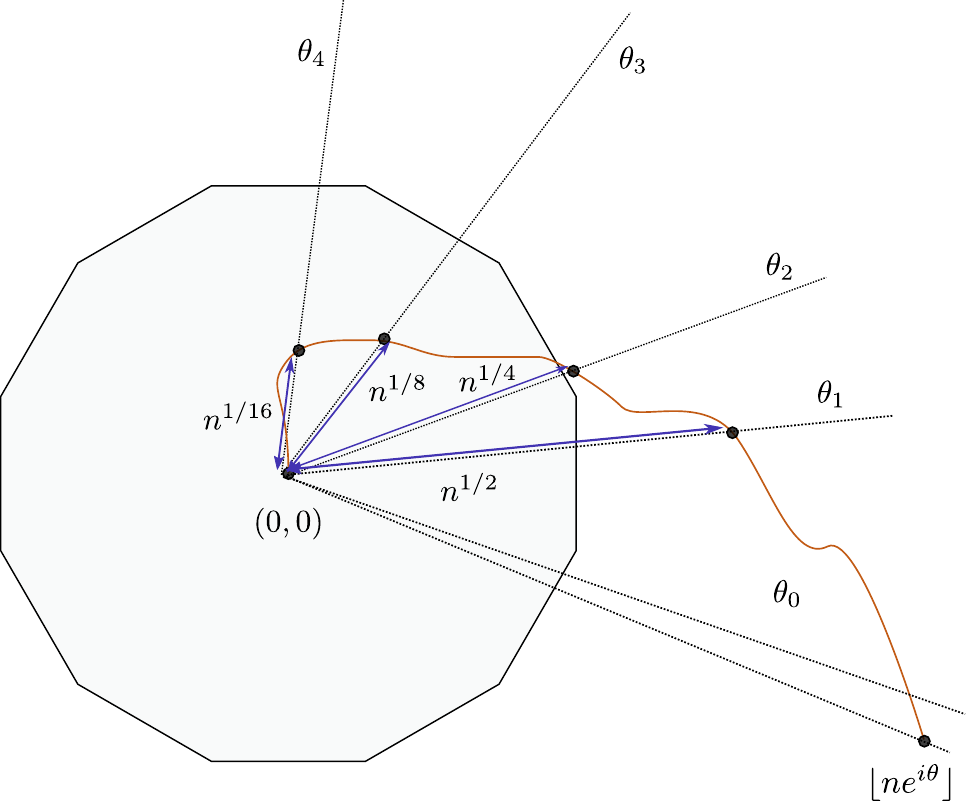
 \caption[fig4]{\label{fig4}The proof of Proposition \ref{prop:limit}. For illustration purposes the angles $\theta_0,\dots,\theta_4$ do not belong to $(\pi/4,\pi/2)$ as in the proof of Proposition \ref{prop:limit}}
\end{figure}
\begin{proof}[Proof of Proposition~\ref{prop:limit}]
 We let $\theta _0<\theta_1<\dots<\theta_{k}$ be the directions from Claim~\ref{claim:4} and 
 let $\theta\in[-\theta_0,\theta_0]$.
 Set $\varphi_0=\theta$ and $\varphi_j=\theta_j$ for $j\in \{1,\dots,k\}$. In particular, two consecutive $\varphi_j$ are not on the same flat edge.
  For all $0\le j\le k$ we let $\gamma _j(R)$ and $\gamma _j'(R)$ be the geodesics from $(0,0)$ to $\lfloor Re^{i\varphi _j} \rfloor$ and from $(0,0)$ to $\lfloor Re^{-i\varphi _j}\rfloor $ respectively. For $0\le j\le k-1$ define the events 
\begin{equation}
    \mathcal A _j=\!\bigg\{ \!\!\! \begin{array}{cc}
        \exists \  n^{2^{-j-1}}\!\!\le R_j\le n ^{2^{-j}+\epsilon },   R_{j+1}\ge n^{2^{-j-1} +\epsilon } \text{ such that }  \lfloor R_{j+1}e^{i\varphi _{j+1}} \rfloor \in \gamma _j(R_j) \text{ or } &  \\
         \lfloor R_{j+1}e^{-i\varphi_{j+1}} \rfloor \in \gamma _j(R_j) \text{ or }\lfloor R_{j+1}e^{i\varphi _{j+1}} \rfloor \in \gamma ' _j(R_j) \text{ or } \lfloor R_{j+1}e^{-i\varphi_{j+1}} \rfloor \in \gamma ' _j(R_j)  & 
    \end{array} \!\!\!\!\!\!\!\!\!\!\! \bigg\} .
\end{equation}
It follows from Lemma~\ref{lem:theta_1} and a union bound that $\mathbb P (\mathcal A _j )\le C\exp (-n^{c_\epsilon })$. 

Next, for $1\le j\le k$ define
\begin{equation}
    \mathcal B_j:= \Big\{ \exists \, 0 \le R_j\le n \text{ with } | \gamma _j(R_j) |\ge  (R_j+1)n^{\epsilon /2} \text{ or } | \gamma ' _j(R_j) |\ge  (R_j+1)n^{\epsilon /2} \Big\}.
\end{equation}
By Claim~\ref{claim:length} we have that $\mathbb P ( \mathcal B _j )\le C\exp (-n^{c_\epsilon })$ for all $j\le k$. 

Next, we claim that 
\begin{equation}\label{eq:15}
    \big\{ \lfloor Re^{i\varphi } \rfloor\in \gamma \text{ for some } R\ge n^{2^{-k-1} +2\epsilon } \text{ and } |\varphi | > \varphi _{k+1} \big\} \subseteq \bigcup _{j=0}^{k-1} \mathcal A _j \cup \bigcup _{j=0}^{k} \mathcal B _j.
\end{equation}
To this end, for $1\le j\le k$ define the sets 
\begin{equation}
    A_j:=\big\{ \lfloor Re^{i\varphi _j} \rfloor : R>0  \big\} \cup \big\{ \lfloor Re^{-i\varphi_{j}} \rfloor : R>0  \big\} \subseteq \mathbb Z ^2.
\end{equation}
Suppose that the event on the left hand side of \eqref{eq:15} holds and let $z\in \gamma $ be a point of the form $z=\lfloor Re^{i\varphi } \rfloor$ with $R\ge n^{2^{-k}+2\epsilon }$ and $|\varphi |>\varphi_{k}=\theta_k$ ($\varphi\in(-\pi,\pi)$). Let $p_j$ be the last intersection of $\gamma $ with $A_j$ before reaching $\lfloor ne^ {i\theta}\rfloor$. Note that it is possible that some of these points are identical. Clearly, on the path $\gamma $ from the origin to $\lfloor ne^ {i\theta}\rfloor$ we first visit $z$ then $p_{k}$, $p_{k-1}$,$\dots$, and lastly $p_1$. Without loss of generality assume that $p_j=\lfloor R_je^{i\varphi _j} \rfloor $.

First consider the case where $R_{k}< n^{2^{-k}+\epsilon}$. Since the geodesic between $0$ and $p_{k}$ contains $z$, it yields that its length is at least $R\ge (R_{k}+1)n^{\ep/2}$ and the event $\mathcal B_{k}$ occurs.
Otherwise, set
\[j_0\coloneqq\min \{j\in\{ 1,\dots,k\}: R_j\ge  n^{2^{-j}+\epsilon}\}. \]
In particular, the minimum is taken over a non-empty set.
If $j_0=1$, then $R_1\ge n^{1/2+\epsilon }$ and since $p_1$ belongs to the geodesic $\gamma =\gamma_0(n)$, the event $\mathcal A_0$ occurs.
If $j_0>1$, then we have $R_{j_0}\ge  n^{2^{-j_0}+\epsilon}$ and $R_{j_0-1}< n^{2^{-j_0+1}+\epsilon}$. We consider two cases, either $R_{j_0-1}\ge n^{2^{-j_0}}$ and the event $\mathcal A_{j_0-1}$ occurs. Otherwise, if $R_{j_0-1}< n^{2^{-j_0}}$, since the geodesic between $0$ and $p_{j_0-1}$ passes through $p_{j_0}$, its length is at least $ n^{2^{-j_0}+\epsilon}\ge (R_{j_0-1}+1)n^{\ep/2}$ and the event $\mathcal B_{j_0}$ occurs.
This proves inequality \eqref{eq:15} and concludes the proof.

When $\theta=0$, we can choose the $\theta_j$ from Claim ~\ref{claim:4} in $(0,\pi/4)$ and use the same arguments to conclude.

\end{proof}

Next, we turn to prove Lemma~\ref{lem:theta_1}. 

\begin{proof}[Proof of Lemma~\ref{lem:theta_1}]
Fix some $R_2\ge R_1^{1/2+\epsilon }$. There exists $C$ depending on $\mu $ such that $ \big|\mu (R_1e^{i\theta _1})-  \mu (\lfloor R_1e^{i\theta _1} \rfloor) \big|\le C$ and therefore by Theorem~\ref{claim:non random}
\begin{equation}\label{eq:11}
    \big| \mu (R_1e^{i\theta _1})-\mathbb ET(0,\lfloor R_1e^{i\theta _1} \rfloor) \big| \le C +\big| \mu (\lfloor R_1e^{i\theta _1} \rfloor)-\mathbb ET(0,\lfloor R_1e^{i\theta _1} \rfloor ) \big| \le C\sqrt{R_1}\log (R_1).
\end{equation}
Thus, by Talagrand's inequality
\begin{equation}
    \mathbb P \big( \big| T(0,\lfloor R_1e^{i\theta _1} \rfloor)-\mu ( R_1e^{i\theta _1}) \big| \ge R_1^{1/2+\epsilon /2} \big) \le C\exp  \big(-R_1^{c_\epsilon } \big).
\end{equation}
By the same argument we also have 
\begin{equation}
    \mathbb P \big( \big| T(0,\lfloor R_2e^{i\theta _2} \rfloor)-\mu ( R_2e^{i\theta _2}) \big| \ge R_2^{1/2+\epsilon /2} \big) \le C\exp  \big(-R_1^{c_\epsilon } \big).
\end{equation}
and
\begin{equation}
    \mathbb P \big( \big| T \big( \lfloor R_2e^{i\theta _2} \rfloor ,\lfloor R_1e^{i\theta _1} \rfloor \big) -\mu ( R_1e^{i\theta _1}-R_2e^{i\theta _2}) \big| \ge \max (R_1,R_2)^{1/2+\epsilon /2} \big) \le C\exp  \big(-R_1^{c_\epsilon } \big).
\end{equation}
Since $R_2\ge R_1^{1/2+\epsilon }$, on the complement of these bad events, we see that all error terms are negligible compared to $R_2$ and therefore one can switch all the $\mu$ terms in Claim~\ref{claim:limit} with the corresponding passage times and change $c_4$ slightly. Thus, 
\begin{equation} 
T \big( 0, \lfloor R_2e^{i\theta _2} \rfloor \big) +T(\lfloor R_2e^{i\theta _2}\rfloor,\lfloor R_1e^{i\theta _1}\rfloor) \big) \ge c_5R_2 + T \big( 0, \lfloor R_1e^{i\theta _1} \rfloor \big),
\end{equation}
with probability at least $1-C\exp (-R_1^{c_\epsilon })$. Of course, on this event we cannot have $\lfloor R_2e^{i\theta _2} \rfloor \in \gamma $ and therefore 
\begin{equation}
    \mathbb P \big(  \lfloor R_2e^{i\theta _2} \rfloor \in \gamma \big) \le C\exp \big( -R_1^{c_\epsilon }\big).
\end{equation}
Using a union bound over all points of the form $\lfloor R_2e^{i\theta _2} \rfloor $ for some $R_1^{1/2+\epsilon }\le R_2\le R_1^2$ and the fact that with very high probability $|\gamma |\le R_1^2$ and therefore it is unlikely that $\lfloor R_2e^{i\theta _2} \rfloor \in \gamma $ for some $R_2\ge R_1^2$ we get
\begin{equation}
\mathbb P \Big(  \lfloor R_2e^{i\theta _2} \rfloor  \in \gamma \text{ for some } R_2\ge R_1 ^{1/2+\epsilon } \Big) \le C \exp \big(- R_1^{c_\epsilon } \big).
\end{equation}
By the same arguments we also have 
\begin{equation}
\mathbb P \Big(  \lfloor R_2e^{-i\theta _2} \rfloor  \in \gamma \text{ for some } R_2\ge R_1 ^{1/2+\epsilon } \Big) \le C \exp \big(- R_1^{c_\epsilon } \big).
\end{equation}
This finishes the proof of the lemma.
\end{proof}

\subsection{Controlling vertical jumps in the geodesic}\label{part:geometric control}
The aim of this section is to prove that with high probability the geodesic between $(0,0)$ and $(n,s)$ with $|s|\le n
$ does not make large vertical jumps (Proposition \ref{prop:nobigjumps}).

 Let $x\in\mathbb R^2$ and $\theta_1<\theta_2$. Let $\mathrm{Cone}(x,\theta_1,\theta_2)$ denote the cone centered at $x$ between the angle $\theta_1$ and $\theta_2$, that is,
 \begin{equation}
     \mathrm{Cone}(x,\theta_1,\theta_2):=\{ x+re^{i\theta}: r\ge 0, \,\theta\in [\theta_1,\theta_2]\}.
 \end{equation}
 We will need in the proof the following claim. It follows easily from the definition of a cone.
 \begin{claim}\label{claim: cone}
 Let $\theta_0<\theta_1$. For any $x,y\in\mathbb R^2$, we have that 
 $x\in \mathrm{Cone}(y,\theta_0,\theta_1)$ if and only if $y\in \mathrm {Cone}(x,\theta_0+\pi,\theta_1+\pi)$.
 \end{claim}
Denote by $\mathrm {int}(\cB_G)$ be the directions in the interior of a flat edge:
\[\mathrm {int}(\cB_G):=\{\theta\in[0,2\pi]: \exists \theta_0,\theta_1,\quad \theta\in(\theta_0,\theta_1),  \text{$[\theta_0,\theta_1]$ is a flat edge of $\cB_G$}\}\,.\]
Note that if $\theta\notin\mathrm{int}(\cB_G)$, then it corresponds to an extreme point of the limit shape.
 \begin{proof}[Proof of Proposition \ref{prop:nobigjumps}] We assume here \eqref{eq:assumption not ell1}. Let $n\ge 1$, $|s|\le n$, let $\delta>0$ and $m\ge n ^{1/2+\delta}$. Without loss of generality, we can assume that $s\ge 0$. Let $\gamma$ be the geodesic between $(0,0)$ and $(n,s)$.
  Set \[\theta_0:=\sup\{\theta\in[0,\pi/4]: \theta\notin\mathrm {int}(\cB_G)\}.\]
 Under the assumption \eqref{eq:assumption not ell1} we have $\theta_0>0$. Indeed, otherwise, by symmetry, \[\mathrm {int}(\cB_G)=[0,2\pi)\setminus \left\{0,\frac{\pi}{2},\pi,\frac{3\pi}{2}\right\}\] which implies that the limit shape is a dilation of the $\ell_1$ unit ball.
Set \[\ep:=\frac{\theta_0}{2}\,.\]
 
 We define $[b(\theta),u(\theta)]$ to be the flat edge of $\theta$ for $\theta\in \mathrm {int}(\cB_G)$. We set $u(\theta)=b(\theta):=\theta$ for $\theta\notin \mathrm {int}(\cB_G)$. Note that by symmetry, we always have $u(\theta)-b(\theta)\le \pi/2$.
 
 For $x,y\in \mathbb Z^2$, we denote by $E_{x,y}$ the set of unreachable points for the geodesic between $x$ and $y$ (with high probability the geodesic from $x$ to $y$ does not go through these points), defined as follows. Write $y=x+re^{i\theta}$, set
 \[E_{x,y}:=\left\{x+Re^ {i\theta'}\in \mathbb Z^2: R\ge r^{1/2+\delta}, \theta' \in [0,2\pi]\setminus [b(\theta)-\ep,u(\theta)+\ep]\right\}\,.\]
 Note that this definition is not symmetric in $x$ and $y$.
 Consider the following event
 \begin{equation}
     \cE:=\bigg\{ \!\! \begin{array}{c}
        \forall x,y \in [-n^2,n^2]^2\cap \ZZ^2 \text{ such that } \|x-y\|_2\ge m \text{, the}   \\
          \text{ geodesic between $x$ and $y$ does not  intersect $E_{x,y}$}
     \end{array}  \!\!  \bigg\} .
 \end{equation}
On this event, the geodesics behave well, they don't deviate too much from the optimal direction. Note that in the case of a flat edge, there is a cone of optimal directions.
  It is easy to check thanks to Lemma \ref{lem:theta_1} that for $n$ large enough depending on $\delta$ and $\ep$, by a union bound
 \begin{equation}\label{eq:contEc}
     \PP(\cE^c)\le \exp(-(\log n )^2)\,.
 \end{equation}
 
Next, we show that on the event $\mathcal E\cap\OB $ the geodesic $\gamma $ is $(\rho ,m)$-bounded for some constant $\rho>0$ depending only on the limit shape. Note that on the event $\OB$, for $n$ large enough the geodesic remains inside the box $[-n^2,n^2]^2$.
Let $x=(a,b)$ and $y=(a+d,b+c)$ with $0\le a+d\le n$ and $d\ge m$ such that the geodesic $\gamma$ goes through $x$ and $y$ in this order. We aim to bound $|c|$ on the event $\cE\cap\OB$. Write \[(n,s)=x+R_xe^{i\theta_x}\quad \text{with}\quad R_x\ge 0\quad \text{and}\quad -\frac{\pi}{2}\le\theta_x\le \frac{\pi}{2}\,.\]
Since by symmetry and by definition of $\theta_0$, we have $(\pi/2,\pi/2+\theta_0]\setminus \mathrm {int}(\cB_G)\ne \emptyset$ and $[-\pi/2-\theta_0,-\pi/2)\setminus \mathrm {int}(\cB_G)\ne \emptyset$. Thus, by definition of $b$ and $u$, we have 
 \begin{equation}
     -\frac{\pi}{2}-\theta_0\le b(\theta_x)\le u(\theta_x)\le \frac{\pi}{2}+\theta_0\,.
 \end{equation}
 Similarly, write
 \[0=y+R_ye^{i\theta_y}\quad \text{with}\quad R_y\ge 0\quad \text{and}\quad \frac{\pi}{2}\le\theta_y\le \frac{3\pi}{2}\,.\]
We have
 \begin{equation}
     \frac{\pi}{2}-\theta_0\le b(\theta_y)\le u(\theta_y)\le \frac{3\pi}{2}+\theta_0\,.
 \end{equation}
Recall $\theta_0=2\ep$, since $(b(\theta_x),u(\theta_x))\subset  \mathrm {int}(\cB_G)$ then $(b(\theta_x),u(\theta_x))\cap \{-\pi/2+2\ep,\pi/2-2\ep\}=\emptyset$. Thus, we have three possible cases for $\theta_x$ : $[b(\theta_x),u(\theta_x)]\subset [-\pi/2-2\ep,-\pi/2+2\ep]$, $[b(\theta_x),u(\theta_x)]\subset [-\pi/2+2\ep,\pi/2-2\ep]$ and $[b(\theta_x),u(\theta_x)]\subset [\pi/2-2\ep,\pi/2+2\ep]$.
 First consider the case where $[b(\theta_x),u(\theta_x)]\subset [-\pi/2+2\ep,\pi/2-2\ep]$, then on the event $\cE\cap\OB$, we have
 \[y\in \mathrm {Cone}\left(x,-\frac{\pi}{2}+\ep,\frac{\pi}{2}-\ep\right)\quad\text{and}\quad|c|\le \left|\tan\left(\frac{\pi}{2}-\ep\right)\right|d\,.\] 
Similarly,  we have three possible cases for $\theta_y$ : $[b(\theta_y),u(\theta_y)]\subset [\pi/2-2\ep,\pi/2+2\ep]$, $[b(\theta_y),u(\theta_y)]\subset [\pi/2+2\ep,3\pi/2-2\ep]$ and $[b(\theta_y),u(\theta_y)]\subset [3\pi/2-2\ep,3\pi/2+2\ep]$. In the case where $[b(\theta_y),u(\theta_y)]\subset [\pi/2+2\ep,3\pi/2-2\ep]$, we can prove that on the event $\cE\cap\OB$\[|c|\le \left|\tan\left(\frac{\pi}{2}-\ep\right)\right|d.\]
Finally, let us prove that on the event $\cE\cap \OB$, it is not possible that we have both $[b(\theta_x),u(\theta_x)]\not\subset [-\pi/2+2\ep,\pi/2-2\ep]$ and $[b(\theta_y),u(\theta_y)]\not\subset [\pi/2+2\ep,3\pi/2-2\ep]$.

$\bullet$ Let us first assume that $[b(\theta_x),u(\theta_x)]\subset [-\pi/2-2\ep,-\pi/2+2\ep]$ and $[b(\theta_y),u(\theta_y)]\subset [-\pi/2-2\ep,-\pi/2+2\ep]$. On the event $\cE\cap\OB$, since $x$ is on the geodesic from $0$ to $y$ and $y$ is on the geodesic from $x$ to $(n,s)$, we have $x\in \mathrm{Cone}(y, b(\theta_y)-\ep,u(\theta_y)+\ep)$ and and $y\in \mathrm{Cone}(x, b(\theta_x)-\ep,u(\theta_x)+\ep)$. Though thanks to Claim \ref{claim: cone}, we cannot have both $x\in \mathrm{Cone}(y, -\pi/2-3\ep,-\pi/2+3\ep)$ and $y\in \mathrm{Cone}(x, -\pi/2-3\ep,-\pi/2+3\ep)$ (recall that $3\ep<\pi/2$). It follows that this case cannot occur on the event $\cE\cap \OB$. We conclude similarly that the case $[b(\theta_x),u(\theta_x)]\subset [\pi/2-2\ep,\pi/2+2\ep]$ and $[b(\theta_y),u(\theta_y)]\subset [\pi/2-2\ep,\pi/2+2\ep]$ cannot occur. 


$\bullet$ Let us assume that $[b(\theta_x),u(\theta_x)]\subset [-\pi/2-2\ep,-\pi/2+2\ep]$ and $[b(\theta_y),u(\theta_y)]\subset [\pi/2-2\ep,\pi/2+2\ep]$.  
Thanks to Claim \ref{claim: cone}, we have
 \[y\in \mathrm{Cone}(0,b(\theta_y)+\pi,u(\theta_y)+\pi)\]
 with 
$[b(\theta_y)+\pi,u(\theta_y)+\pi]\subset [3\pi/2-2\ep ,3\pi/2+2\ep]$. On the event $\cE\cap\OB$, since $x$ is on the geodesic from $0$ to $y$, we have that \[x\in \mathrm{Cone}(0,b(\theta_y)+\pi-\ep,u(\theta_y)+\pi+\ep)\subset \mathrm{Cone}\left(0,\frac{3\pi}{2}-3\ep,\frac{3\pi}{2}+3\ep\right)\subset (\pi,2\pi)\,. \] 
 In particular, we have $b\le 0$ and since $s\ge 0$ we must have $\theta_x\ge 0$. This is a contradiction with the fact that $[b(\theta_x),u(\theta_x)]\subset [-\pi/2-2\ep,-\pi/2+2\ep]$.
 Hence, this case cannot occur on the event $\cE\cap\OB$. We conclude similarly, that the case $[b(\theta_x),u(\theta_x)]\subset [\pi/2-2\ep,\pi/2+2\ep]$ and $[b(\theta_y),u(\theta_y)]\subset [-\pi/2-2\ep,-\pi/2+2\ep]$
 cannot occur on the event  $\cE\cap\OB$.

 It follows that on the event $\cE\cap \OB$, we have the following control 
  \[|c|\le \left|\tan\left(\frac{\pi}{2}-\ep\right)\right|d\] 
 where $\ep$ depends only on the limit shape. The result follows from \eqref{eq:contEc} together with Lemma \ref{lem:basic geometric control}.

Let us no longer assume \eqref{eq:assumption not ell1}, we now assume that
\begin{equation}\label{eq:condm}
    m\ge |s|+\sqrt{n}\log^2 n
\end{equation}
Let $x=(a,b)$ and $y=(a+d,b+c)$ with $0\le a+d\le n$ and $d\ge m$ such that the geodesic $\gamma$ goes through $x$ and $y$ in this order.
Using Claim \ref{claim:ineqnu}, we have
\begin{equation}
    \mu((n,s))\le (n+|s|)\mu((1,0))
\end{equation}
and 
\begin{equation}
\begin{split}
    \mu(x)+\mu(y-x)+\mu((n,s)-y)&\ge (a+ |c|+n-a-d)\mu((1,0))\\
    &\ge (n+|c|-d)\mu((1,0)).
    \end{split}
\end{equation}
On the event $\OT\cap\OB$ (for $L=n$), thanks to Theorem \ref{claim:non random}, 
we have 
\[T(0,x)+T(x,y)+T(y,(n,s))\ge  (n+|c|-d)\mu((1,0))-6\sqrt{n}\log^2 n\]
and
\[T(0,(n,s))\le (n+|s|)\mu((1,0))+2\sqrt{n}\log^2n\]
It follows that 
\[(n+|s|)\mu((1,0))\ge (n+|c|-d)\mu((1,0))-8\sqrt{n}\log^2 n\]
and by inequality \eqref{eq:condm} and $d\ge m\ge \sqrt n\log^ 2n$, it yields that
\[8\sqrt{n}\log^2 n\ge (|c|-2d-\sqrt{n}\log^2 n)\mu((1,0))\]
and
\[|c|\le (3+8/\mu((1,0)))d.\]
Hence, the geodesic $\gamma$ has $(\rho,m)$-bounded slope on the event $\OT\cap\OB$ with $\rho=(3+8/\mu((1,0)))$.
The result follows from Lemmas \ref{lem:basic geometric control} and \ref{lem:Talagrand control}.
 
 \end{proof}

\section{The number of sides of the limit shape}\label{sec:log}

In this section we prove Theorem~\ref{thm:log over loglog sides}. Recall that $X$ is a random variable supported on $[0,1]$ with $\text{Var}(X)=\sigma ^2$ and let $E:=\mathbb E [X]$.  Let $G_{\epsilon }$ be the distribution of $1+\epsilon X$ and let $\mathcal B _{\epsilon }$ and $\mu _{\epsilon }$ be the limit shape and limiting norm corresponding to $G_{\epsilon }$. The limiting norm of a distribution is defined in Theorem~\ref{thm:defmu}. Recall the definitions and notations at the beginning of Section~\ref{sec:the limit shape}.  We start with the following proposition.
 
\begin{prop}\label{lem:1}
 There exists $\delta _0(\sigma )>0$ such that for all $ \epsilon <\delta _1 <\delta _2<\delta _0 (\sigma )$ with  $\delta _1 \log ^2 (1/\delta _1 )\le \delta _2 $, the directions $(1,\delta _1 )$ and $(1,\delta _2)$ are not on the same flat edge of $\mathcal B _{\epsilon }$. 
\end{prop}

We turn to prove Theorem~\ref{thm:log over loglog sides} using the proposition.

\begin{proof}[Proof of Theorem~\ref{thm:log over loglog sides}]
 Suppose that $\mathcal B_\epsilon $ is a polygon and let 
 \begin{equation}
     k:=\Big\lfloor  \frac{\log (1/\epsilon )}{7\log \log (1/\epsilon )} \Big\rfloor
 \end{equation}
 and for all $1\le i \le k$,
 \begin{equation}
     \delta _i :=\epsilon \log ^{3i} (1/\epsilon ).
 \end{equation}
 As long as $\epsilon $ is sufficiently small we have that $\epsilon <\delta _1 \le \cdots \le \delta _k \le \sqrt{\epsilon }$ and that $\delta _i$ and $\delta _{i+1}$ satisfy the assumptions of Proposition~\ref{lem:1} for all $1\le i\le k$. This gives that the directions $(1,\delta _i )$ and $(1,\delta _{i+1})$ are not on the same flat edge of $\mathcal B_\epsilon $ and therefore $\mathcal B_\epsilon$ has at least $k-1$ vertices between the angles $\arg (1,\epsilon )$ and $\arg (1,\sqrt{\epsilon })$. It follows from the symmetries of the lattice that $\mathcal B_\epsilon $ has at least $8(k-1)$ vertices. Indeed, by the reflection symmetry around the $y=x$ line, there are $2(k-1)$ vertices in the first quadrant and by the $90$-degrees rotation symmetry there is at least this number of vertices in all other quadrants. The theorem follows by taking $\epsilon $ sufficiently small.
\end{proof}

In order to prove Proposition~\ref{lem:1} we need the following lemmas. To this end, let $\delta >\epsilon$ and note that 
\begin{equation}\label{eq:trivial upper bound}
\mathbb E \big[ T\big( (0,0), (n,\delta n) \big) \big] \le (1+\delta )(1+\epsilon E)n.    
\end{equation}
Indeed, any fixed path of length $(1+\delta )n$ has expected weight $(1+\delta )(1+\epsilon E)n$. It follows that 
\begin{equation}
    \mu _{\epsilon } (1,\delta )\le (1+\delta )(1+\epsilon E).
\end{equation}
The following lemmas give upper and lower bounds for the difference $(1+\delta )(1+\epsilon E)-\mu_\epsilon (1,\delta )$.

\begin{lem}\label{claim:two directions upper}
  There exists $\delta _0(\sigma )>0$ such that for all $ \epsilon <\delta <\delta _0(\sigma ) $ such that $1/\delta $ is an integer we have 
  \begin{equation}
        \mu _{\epsilon } (1,\delta ) \le  (1+\delta )\big( 1+\epsilon E \big) -c_{\sigma }\epsilon \sqrt{ \delta },
  \end{equation}
where $c_\sigma >0$ is a constant depending only on $\sigma $.
\end{lem}

\begin{lem}\label{claim:two directions lower}
 For all $ \epsilon <\delta <1/2 $ such that $1/\delta $ is an integer we have 
  \begin{equation}
        \quad \quad \quad \ \ \mu _{\epsilon } (1,\delta ) \ge  (1+\delta )  \big( 1+\epsilon E  \big) -C\epsilon \sqrt{\delta \log (1/\delta )}.
  \end{equation}
\end{lem}
 
We start with a proof of Lemma~\ref{claim:two directions upper}. To this end we need the following claim.

\begin{claim}\label{claim:Berry}
  There are constants $c$ and $M_0$ depending only on $\sigma $ such that for all $M\ge M_0$ the following holds. Let $X_1,\dots ,X_M$ be i.i.d. copies of $X$ and let $S:=\sum _{i=1}^M X_i$. Then
  \begin{equation}
      \mathbb P \bigg( \sum _{j=1}^M X_i \le ME - \sqrt{M} \bigg)\ge c.
  \end{equation}
\end{claim}

Claim~\ref{claim:Berry} follows from the Berry-Esseen Theorem, see \cite[Theorem 3.4.17]{durrett2019probability}.
Throughout the proofs of Lemma~\ref{claim:two directions upper} and Lemma~\ref{claim:two directions lower}  we take $n\ge 1$ sufficiently large depending on all other parameters such that $\delta n$ is an integer and give upper and lower bounds on the passage time 
 \begin{equation}
     T_n:=T\big( (0,0),(n,\delta n) \big).
 \end{equation}
 
 \begin{proof}[Proof of Lemma~\ref{claim:two directions upper}]
Let $\delta _0:=1/M_0$ where $M_0$ is the constant from Claim~\ref{claim:Berry}. Consider the sequence of points $x_i:=(i/\delta , i)\in \mathbb Z ^2$ for $0\le i\le \delta n$. Let $p_i$ be the path from $x_{i-1}$ to $x_i$ that goes one step up and then $1/\delta $ steps to the right and let $q_i$ be the path from $x_{i-1}$ to $x_i$ that goes $1/\delta $ steps to the right and then one step up. Finally, define $\bar{T}_i:=\min (T(p_i),T(q_i))$. By the triangle inequality we have  
\begin{equation}\label{eq:36}
     T_n \le \sum _{i=1}^{\delta n} T(x_{i-1},x_i)  \le  \sum _{i=1}^{\delta n} \bar{T }_i.
\end{equation}
Thus, it suffices to bound the expectation of $\bar{T}_i$. We show that 
\begin{equation}\label{eq:35}
    \mathbb{E}[\bar{T }_i] \le (1+1/\delta)(1+\epsilon E ) -c_{\sigma }\epsilon /\sqrt{\delta }.
\end{equation}
Since $p_i$ and $q_i$ are disjoint and have the same length, the random variables $T(p_i)$ and $T(q_i)$ are independent and identically distributed with $\mathbb E [T(p_i)]=\mathbb E [T(q_i)]=(1+1/\delta )(1+\epsilon E)$. Moreover, these variables equal in distribution to $1+1/\delta $ plus $\epsilon $ times a sum of $1+1/\delta $ i.i.d. variables that are independent of $\delta $ and $\epsilon $. Define the event
\begin{equation}
  \mathcal A :=\big\{ T(p_i) \le (1+1/\delta)(1+\epsilon E ) - \epsilon /\sqrt{\delta } \big\}.
\end{equation}
By Claim~\ref{claim:Berry} with $M:=1+1/\delta $ we have that $\mathbb P (\mathcal A) \ge c_{\sigma }$ and therefore 
 \begin{equation}
 \begin{split}
     \mathbb E \big[ \min \big (  T(p_i),&T(q_i) \big) \big] \le \mathbb E \big[ \mathds 1 _{\mathcal A }  \big( (1+1/\delta) (1+\epsilon E ) - \epsilon /\sqrt{\delta } \big) +\mathds 1 _{\mathcal A ^c}  \cdot T(q_i) \big] \\
     &\le \mathbb P ( \mathcal A )  \cdot \big( (1+1/\delta)(1+\epsilon E ) -  \epsilon /\sqrt{\delta } \big) +\mathbb P (\mathcal A ^c) \cdot  (1+\epsilon E )(1+1/\delta) \\
     &\le (1+1/\delta)(1+\epsilon E ) -c_{\sigma }\epsilon /\sqrt{\delta }.
 \end{split}
 \end{equation}
 This finishes the proof of \eqref{eq:35}. Finally, by \eqref{eq:36} we have that
  \begin{equation}
      \mathbb E \big[  T_n  \big] \le \big( (1+\delta )(1+\epsilon E)-c_{\sigma }\epsilon \sqrt{\delta }  \big)n. 
  \end{equation}
  The lemma follows from the last inequality. 
  \end{proof}
  \begin{proof}[Proof of Lemma~\ref{claim:two directions lower}]
  Let $\mathcal P _{\delta ,n}$ be the set of paths from $(0,0)$ to $(n,\delta n)$ of length at most $(1+2\delta )n$. Each path in $\mathcal P _{\delta ,n}$ has at most $2\delta n$ edges which are not directed to the right. Thus, summing over the possible lengths of the paths we obtain
  \begin{equation}\label{eq:bound on P}
      \big| \mathcal P _{\delta ,n} \big| \le \sum _{k=n}^{n+2\delta n} \binom{k}{2\delta n} 3^{2\delta n} \le e^{C\delta n}\binom{n+2\delta n}{2\delta n} \le \exp \big( C\delta \log (1/\delta ) n \big), 
  \end{equation}
  where the last inequality follows from Stirling's formula.
  Next, for each path $p\in \mathcal P _{\delta ,n }$, the length of $p$ is at least $(1+\delta  )n$  and therefore $\mathbb E [T(p)]\ge (1+\epsilon E)(1+\delta )n$. Thus, by Azuma--Hoeffding inequality for all $C_1>0$ we have 
  \begin{equation}\label{eq:weight1}
    \mathbb P \Big( T(p)\le (1+\delta )(1+\epsilon E)n -C_1\epsilon \sqrt{\delta \log (1/\delta )}n \Big) \le \exp \big( - C_1^2\delta \log (1/\delta )n/5  \big).
  \end{equation}
  Taking $C_1$ sufficiently large and using \eqref{eq:weight1}, \eqref{eq:bound on P} and a union bound, we get that with high probability 
  \begin{equation}\label{eq:20}
      \forall p \in \mathcal P _{\delta ,n} ,\quad T(p)\ge (1+\delta )(1+\epsilon E)n -C_1\epsilon \sqrt{\delta \log (1/\delta )}n.
  \end{equation}
  Finally, we claim that the geodesic $\gamma $ from $(0,0)$ to $(n,\delta n)$ is in $\mathcal P _{\delta ,n}$ with high probability. Indeed, since the weights are almost surely at least $1$, on the event $\big\{ T_n<(1+2\delta )n\big\}$ the length of the geodesic is at most $(1+2\delta )n$. Thus,
  \begin{equation}\label{eq:21}
      \mathbb P \big( \gamma \notin  \mathcal P _{\delta ,n} \big) \le \mathbb P \big( T_n \ge (1+2\delta )n \big) \to 0, \quad n\to \infty 
  \end{equation}
  where the limit follows from \eqref{eq:trivial upper bound} and Talagrand's inequality. Combining \eqref{eq:20} and \eqref{eq:21} we get that with high probability 
  \begin{equation}
      T_n \ge (1+\delta )(1+\epsilon E)n -C_1\epsilon \sqrt{\delta \log (1/\delta )}n.
  \end{equation}
  The lemma follows from this.
 \end{proof}
 
 We turn to prove Proposition~\ref{lem:1}.
 
 \begin{proof}[Proof of Proposition~\ref{lem:1}]
 Let $0<\delta _0<1/2$ sufficiently small and let $ \epsilon <\delta _1 <\delta _2 <\delta _0 $ such that $\delta _1 \log ^2 (1/\delta _1) \le 2\delta _2 $ and such that $1/\delta _1$ and $1/\delta _2$ are integers. It suffices to show that $(1,\delta _1)$ and $(1,\delta _2)$ are not on the same flat edge of $\mathcal B _{\epsilon }$. We take some $\delta _1 <\delta _3<\delta _2 $ whose exact value will be chosen later. We have that 
 \begin{equation}
     (1,\delta _3 )=\frac{\delta _2 -\delta _3 }{\delta _2 -\delta _1 } (1,\delta _1 ) + \frac{\delta _3 -\delta _1 }{\delta _2 -\delta _1 }(1,\delta _2)
 \end{equation}
 and therefore, by the triangle inequality,
 \begin{equation}\label{eq:22}
     \mu_\epsilon (1,\delta _3 )\le \frac{\delta _2 -\delta _3 }{\delta _2 -\delta _1 } \mu_\epsilon (1,\delta _1 ) + \frac{\delta _3 -\delta _1 }{\delta _2 -\delta _1 }\mu_\epsilon (1,\delta _2).
 \end{equation}
 By Claim~\ref{claim:basic claim on limit shape}, the directions $(1,\delta _1)$, $(1,\delta _2)$ are on the same flat edge of the limit shape if and only if \eqref{eq:22} holds as an equality. Thus, it remains to show that, for a suitable choice of $\delta _3$, \eqref{eq:22} holds as a strict inequality. By Lemma~\ref{claim:two directions upper} we have that 
 \begin{equation}\label{eq:23}
     \mu_\epsilon (1,\delta _3) \le (1+\delta _3 )\big( 1+\epsilon E \big) -c\epsilon \sqrt{ \delta _3 }.
 \end{equation}
 By Lemma~\ref{claim:two directions lower} we have
 \begin{equation}\label{eq:24}
 \begin{split}
 \frac{\delta _2 -\delta _3 }{\delta _2 -\delta _1 } &\mu_\epsilon (1,\delta _1 ) + \frac{\delta _3 -\delta _1 }{\delta _2 -\delta _1 }\mu_\epsilon (1,\delta _2) \\
 &\ge \frac{\delta _2 -\delta _3 }{\delta _2 -\delta _1 }  \Big( (1+\delta _1 ) \big( 1+\epsilon E \big)  -C\epsilon \sqrt{\delta _1 \log (1/\delta _1)}  \Big)  +\\
    &\quad \quad \quad \quad \quad \quad \quad \quad \quad +\frac{\delta _3 -\delta _1 }{\delta _2 -\delta _1 } \Big( (1+\delta _2 )  \big( 1+\epsilon E \big)  -C\epsilon \sqrt{\delta _2 \log (1/\delta _2)}  \Big)\\
    &=   (1+\delta _3 )  \big( 1+\epsilon E \big)  -C \epsilon \frac{\delta _2 -\delta _3 }{\delta _2 -\delta _1 }   \sqrt{\delta _1 \log (1/\delta _1 )}-C \epsilon \frac{\delta _3 -\delta _1 }{\delta _2 -\delta _1 }   \sqrt{\delta _2 \log (1/\delta _2 )}\\
    &\ge (1+\delta _3 ) \big( 1+\epsilon E \big)  -C \epsilon  \sqrt{\delta _1 \log (1/\delta _1 )}-C \epsilon \frac{\delta _3  }{\delta _2  }   \sqrt{\delta _2 \log (1/\delta _2 )}
 \end{split}
 \end{equation}
 where in the last inequality we used that $\delta _1 <\delta _3<\delta _2 $. 
 Finally, the right hand side of \eqref{eq:23} is strictly smaller than the right hand side of \eqref{eq:24} as long as 
 \begin{equation}
     \delta _1 \log (1/\delta _1 ) \le c_1 \delta _3 \quad \text{and} \quad  \delta _3 \log (1/\delta _2) \le c_1 \delta _2,
 \end{equation}
 for a sufficiently small $c_1>0$. Such a choice of $\delta _3$ is clearly possible when $\delta _1 \log ^2(1/\delta _1) \le \delta _2$ and $\delta _0$ is sufficiently small.
 \end{proof}

\subsection*{Acknowledgements}The research of B.D. is partially funded by the SNF Grant 175505 and the ERC Starting Grant CriSP (grant agreement No 851565) and is part of NCCR SwissMAP. The research of R.P. is supported by the Israel Science Foundation grant
1971/19, by the European Research Council starting grant 678520 (LocalOrder) and by the European Research Council Consolidator grant 101002733 (Transitions).

Part of this work was completed while R.P. was a Cynthia and Robert Hillas Founders' Circle Member of the Institute for Advanced Study and a visiting fellow at the Mathematics Department of Princeton University. R.P. is grateful for their support.

We thank Paul Dario for important contributions in the early stages of this work. We are grateful to Jean-Baptiste Gou\'er\'e for pointing out a small mistake in the previous proof of~\eqref{eq:MW lower bound}. We thank Itai Benjamini for stimulating discussions and thank Allan Sly and Lingfu Zhang for helpful comments. Finally, we thank the two referees of the paper for excellent comments which greatly improved the presentation.

\bibliographystyle{plain}
\bibliography{biblio}

\end{document}